\documentclass[11pt, a4paper, twoside]{amsart}
\usepackage{amssymb}
\usepackage{amsfonts, graphicx, color}
\usepackage{mathrsfs}

\theoremstyle{plain}
\newtheorem{theorem}{Theorem}

\newtheorem{corollary}{Corollary}

\newtheorem{lemma}{Lemma}

\newtheorem{proposition}{Proposition}

\numberwithin{equation}{section}
\begin{document}
\title[Riesz means and bilinear Riesz means on M\'{e}tivier groups]{Riesz
means and bilinear Riesz means on M\'{e}tivier groups}

\begin{abstract}
In this paper, we investigate the $L^{p}$-boundedness of the Riesz means and
the $L^{p_{1}}\times L^{p_{2}}\rightarrow L^{p}$ boundedness of the bilinear
Riesz means on M\'{e}tivier groups. M\'{e}tivier groups are generalization
of Heisenberg groups and general H-type groups. Because general M\'{e}tivier
groups only satisfy the non-degeneracy condition and have high-dimensional
centre, we have to use different methods and techniques from those on
Heisenberg groups and H-type groups.
\end{abstract}

\author{Min Wang$^{1}$}
\address{School of Science, China University of Geosciences in Beijing, Beijing,
100083, P. R. China, }
\email{wangmin09150102@163.com}

\author{Hua Zhu$^{1,*}$}
\address{Department of Basic Sciences, Beijing International Studies University, Beijing,
100024, P. R. China, }
\email{zhuhua@pku.edu.cn}
\thanks{$1$\ These two authors contributed to this work equally and should be regarded as co-first authors. }
\thanks{$*$\ This author is supported by ``Young top-notch talent'' Program concerning teaching faculty development of universities affiliated to Beijing Municipal Government(2018-2020).}
\date{}
\maketitle

\allowdisplaybreaks

\section{Introduction}


M\'{e}tivier groups, introduced by M\'{e}tivier \cite{Metivier} in his study
of analytic hypoellipticity, are two-step nilpotent Lie groups satisfying a
non-degeneracy condition. They are also characterized by the property that
the quotients with respect to the hyperplanes contained in the centre are
general Heisenberg groups. H-type groups, introduced by Kaplan \cite{Kaplan}%
, are typical examples of M\'{e}tivier groups, but there are many M\'{e}%
tivier groups which are not isomorphic to H-type groups, for example, see
\cite{Muller-Seeger}. Heisenberg groups are the simplest M\'{e}tivier
groups. Casarino and Ciatti \cite{Cas} investigated the spectral resolution
of the sub-Laplacian on M\'{e}tivier groups, and proved a Stein-Tomas
restriction theorem in terms of the mixed norms. So, we can use the the
spectral decomposition of the sub-Laplacian to define the Riesz means and
the bilinear Riesz means on M\'{e}tivier groups. Mauceri \cite{Mau} and M%
\"{u}ller \cite{Mull2} obtained the same results on the $L^{p}$-boundedness
of the Riesz means on Heisenberg groups by using different methods. We
proved the $L^{p_{1}}\times L^{p_{2}}\rightarrow L^{p}$ boundedness of the
bilinear Riesz means on Heisenberg groups \cite{LW-H}. To extend these
results to general M\'{e}tivier groups, we notice that there are two
essential difficulties: one is that Heisenberg groups have one dimensional
centre but the centre dimension of M\'{e}tivier groups is in general bigger
than one; the other is that on H-type groups, the non-degeneracy condition
becomes a better orthogonality condition, which is not true on general M\'{e}%
tivier groups. Therefore, in this paper, we shall introduce some different
techniques from those used in \cite{Mau} or \cite{Mull2} and \cite{LW-H} to
obtain the $L^{p}$- boundedness of the Riesz means and $L^{p_{1}}\times
L^{p_{2}}\rightarrow L^{p}$ boundedness of the bilinear Riesz means on M\'{e}%
tivier groups.

This paper is organized as follows. In Section 2, we recall the spectral
decomposition of the sub-Laplacian on Metivier groups and define the Riesz
means and bilinear Riesz means. In Section 3, we give the pointwise
estimates for the kernel of the Riesz means and the kerner of the bilinear
Riesz means. In Section 4, we present the $L^p$-boundedness of the Riesz
means. In the rest of this paper, we study the $L^{p_{1}}\times
L^{p_{2}}\rightarrow L^{p}$ boundedness of the bilinear Riesz means, which
for the case of $1\leq p_{1},p_{2}\leq 2$ in Section 5 and for some
particular cases in Section 6. In Section 7, we outline the bilinear
interpolation method and obtain the results in other cases.

\section{Preliminaries}

We first recall M\'{e}tivier groups. Let $\mathfrak{g}$ be a real finite
dimensional two step nilpotent Lie algebra, equipped with an inner product $%
\left\langle \cdot ,\cdot \right\rangle $. Then
\begin{equation*}
\mathfrak{g}\,=\mathfrak{g}_{1}\oplus \mathfrak{g}_{2},
\end{equation*}%
where $[\mathfrak{g}_{1},\mathfrak{g}_{1}]=\mathfrak{g}_{2}$, $[\mathfrak{g}%
_{1},\mathfrak{g}_{2}]=[\mathfrak{g}_{2},\mathfrak{g}_{2}]=0$ and $\dim
\mathfrak{g}_{1}=d$, $\dim \mathfrak{g}_{2}=m$. Let $\{X_{1},X_{2},\cdots
,X_{d}\}$ be an orthonormal basis of $\mathfrak{g}_{1}$, $%
\{U_{1},U_{2},\cdots ,U_{m}\}$ be an orthonormal basis of $\mathfrak{g}_{2}$.

Define for $\mu \in \mathfrak{g}_{2}^{* }$ the skew symmetric form $\omega
_{\mu }$ on $\mathfrak{g}_{1}$ by
\begin{equation*}
\omega _{\mu }(X,Y)=\mu ([X,Y])
\end{equation*}%
and the corrsponding metrix
\begin{equation*}
(J_{\mu })_{jk}=\omega _{\mu }(X_{j},X_{k}).
\end{equation*}%
We say that the corresponding connected, simply connected Lie group $\mathbb{%
G}$ of $\mathfrak{g}$ is a M\'{e}tivier group (M-type group) if $\omega
_{\mu }$ is non-degenerate for all $\mu \neq 0$. Under this non-degenerate
hypothesis, $J_{\mu }$ is a $d\times d$ skew symmetric, invertible metrix
for all $\mu \neq 0$. It follows that $d=2n$ is even. Since the exponential
mapping is a bijection, we shall parametrise the element $g=\exp \left(
\sum_{i=1}^{2n}x_{i}X_{i}+\sum_{j=1}^{m}u_{j}U_{j}\right) \in \mathbb{G}$ by
$(x,u)$ where $x=(x_{1},\cdots ,x_{2n})\in \mathfrak{%
\mathbb{R}
}^{2n}$, $u=(u_{1},\cdots ,u_{m})\in \mathfrak{%
\mathbb{R}
}^{m}$. By the Baker-Campbell-Hausdorff formula, the group multiplication is
given by

\begin{equation*}
(x,u)\cdot (y,v)=(x+y,u+v+\frac{1}{2}[x,y]),
\end{equation*}%
where $[x,y]=(\left\langle x,J_{1}y\right\rangle ,\left\langle
x,J_{2}y\right\rangle ,\cdots ,\left\langle x,J_{m}y\right\rangle )\in
\mathbb{R}
^{m}$, $J_{k}$ is $2n\times 2n$ skew-symmetric, non-degenerate metrix for
any $k=1,\cdots ,m$. The identity element of $\mathbb{G}$ is $(0,0)$ and the
inversion of $(x,u)$ is denoted by $(x,u)^{-1}=(-x,-u)$.

A homogeneous structure on $\mathbb{G}$ is obtained by defining the
dilations $\delta _{t}(x,u)=(tx,t^{2}u)$, $t>0$. The homogeneous dimension
of $\mathbb{G}$ is
\begin{equation*}
Q=2n+2m.
\end{equation*}%
The Haar measure on $\mathbb{G}$ coincides with the Lebesgue measure on $%
\mathfrak{g}$ denoted by $dxdu$. It is easy to verify that the Jacobian
determinant of the dilations $\delta _{t}$, $t>0$ is constant, equal to $%
t^{Q}$. Define a homogeneous norm of degree one under the dilations $%
\{\delta _{t}$, $t>0\}$ on $\mathbb{G}$ by
\begin{equation*}
\left\vert \omega \right\vert =\left\vert (x,u)\right\vert =\left( \frac{1}{%
16}\left\vert x\right\vert ^{4}+\left\vert u\right\vert ^{2}\right) ^{\frac{1%
}{4}},\text{ }\omega =(x,u)\in \mathbb{G}\text{.}
\end{equation*}%
This norm satisfies the triangle inequality $\left\vert \omega \cdot \omega
^{\prime }\right\vert \leq \left\vert \omega \right\vert +\left\vert \omega
^{\prime }\right\vert $ and leads to a left-invariant distance $d(\omega
,\omega ^{\prime })=\left\vert \omega ^{-1}\cdot \omega ^{\prime
}\right\vert $.

For any $f,g\in L^{1}(\mathbb{G})$, their convolution is defined by%
\begin{equation*}
f\ast g(x,u)=\int_{\mathbb{G}}f((x,u)\cdot (y,v)^{-1})g(y,v)dydv.
\end{equation*}%
The $\mu $-twisted convolution of two suitable functions or distributions on
$\mathfrak{g}_{1}$ is defined by%
\begin{equation*}
\phi \times _{\mu }\psi (x)=\int_{\mathfrak{g}_{1}}\phi (x-y)\psi (y)e^{%
\frac{i}{2}\mu ([x,y])}d_{\mu }y,
\end{equation*}%
where $d_{\mu }y=\sqrt{\det J_{\mu }}dy$. Given $f\in L^{1}(\mathbb{G})$, we
define the Fourier transform with respect to the central variables by
\begin{equation*}
f^{\mu }(x)=\int_{\mathfrak{g}_{2}}f(x,u)e^{i\mu (u)}du.
\end{equation*}%
It is easy to verify that
\begin{equation*}
(f\ast g)^{\mu }=f^{\mu }\times _{\mu }g^{\mu }.
\end{equation*}

For any $X\in \mathfrak{g}_{1}$, we obtain a left-invariant vector field $X$
on $\mathbb{G}$ defined by
\begin{equation*}
Xf(g)=\frac{d}{dt}f(g\cdot (tX))|_{t=0\text{, \ \ \ \ \ \ \ }}f\in C^{\infty
}(\mathbb{G})\text{, }g\in \mathbb{G}.
\end{equation*}%
Then, the orthonormal basis of $\mathfrak{g}_{1}$, $\{X_{1},X_{2},\cdots
,X_{2n}\}$, is associated to the left-variant vector field:
\begin{equation*}
X_{j}=\frac{\partial }{\partial x_{j}}+\frac{1}{2}\sum_{k=1}^{m}\left\langle
x,J_{k}e_{j}\right\rangle \frac{\partial }{\partial u_{k}},\text{ \ \ }%
e_{j}=(0,\cdots ,1_{j},0,\cdots ,0_{2n})^{T},\ \ j=1,2,\cdots,2n.
\end{equation*}%
The sub-Laplacian on $\mathbb{G}$ is the left-invariant hypoelliptic
operator defined by
\begin{equation*}
\mathcal{L=-}\sum_{j=1}^{2n}X_{j}^{2}=-\Delta
_{x}-\sum_{k=1}^{m}\left\langle x^{T}J_{k},\nabla _{x}\right\rangle \frac{%
\partial }{\partial u_{k}}-\frac{1}{4}\sum_{k,l=1}^{n}\left\langle
x^{T}J_{k},x^{T}J_{l}\right\rangle \frac{\partial ^{2}}{\partial
u_{p}\partial u_{q}},
\end{equation*}%
where
\begin{equation*}
\Delta _{x}=\sum_{j=1}^{2n}\frac{\partial ^{2}}{\partial x_{j}^{2}},\quad
\quad \nabla _{x}=\left( \frac{\partial }{\partial x_{1}},\frac{\partial }{%
\partial x_{2}},\cdots ,\frac{\partial }{\partial x_{2n}}\right) ^{T}.
\end{equation*}%
For any $f\in \mathscr{S}(\mathbb{G})$ and $j=1,\cdots ,n$, the Fourier
transform of $X_{j}f$ with respect ot the central variables is given by
\begin{eqnarray*}
\left( X_{j}f\right) ^{\mu }(x) &=&\int_{\mathfrak{g}_{2}}\frac{\partial f}{%
\partial x_{j}}(x,u)e^{i\mu (u)}du+\frac{1}{2}\sum_{k=1}^{m}\left\langle
x,J_{k}e_{j}\right\rangle \int_{\mathfrak{g}_{2}}\frac{\partial f}{\partial
u_{k}}(x,u)e^{i\mu (u)}du \\
&=&\left( \frac{\partial }{\partial x_{j}}-\frac{i}{2}\left\langle x,J_{\mu
}e_{j}\right\rangle \right) f^{\mu }(x).
\end{eqnarray*}%
Setting
\begin{equation*}
X_{j}^{\mu }=\frac{\partial }{\partial x_{j}}-\frac{i}{2}\left\langle
x,J_{\mu }e_{j}\right\rangle ,
\end{equation*}%
we have that
\begin{equation*}
\left( X_{j}f\right) ^{\mu }=X_{j}^{\mu }f^{\mu }.
\end{equation*}%
Define
\begin{equation*}
\Delta ^{\mu }=-\sum_{j=1}^{2n}\left( X_{j}^{\mu }\right) ^{2}.
\end{equation*}%
It follows that
\begin{equation*}
\left( \mathcal{L}f\right) ^{\mu }=\Delta ^{\mu }f^{\mu }.
\end{equation*}

Next, we let $\mathbb{S}^{m-1}=\{\eta \in \mathfrak{g}_{2}^{\star
}:\left\vert \eta \right\vert =1\}$ denote the unit sphere of $%
\mathbb{R}
^{m}\cong \mathfrak{g}_{2}^{\star }$. Since that for any fixed $\eta \in
\mathbb{S}^{m-1}$, the corresponding metrix $J_{\eta }$ is skew symmetric
and non-degenerate, then there exists a $2n\times 2n$ invertible matrix $%
A_{\eta }$ such that%
\begin{equation*}
J_{\eta }=A_{\eta }^{T}J_{2n}A_{\eta }
\end{equation*}%
where $J_{2n}=\left(
\begin{array}{cc}
0_{n} & I_{n} \\
-I_{n} & 0%
\end{array}%
\right) $. Clearly, $\left( \det A_{\eta }\right) ^{2}=\det J_{\eta }.$
Since $\det J_{\eta }$, which is a polynomial function of the components of $%
\eta $, never vanishes on $\mathbb{S}^{m-1}$, then there exists a positive
constant $K$ such that for any $\eta \in \mathbb{S}^{m-1}$
\begin{equation}
\frac{1}{K}\leq \left\vert \det A_{\eta }\right\vert ^{2}=\left\vert \det
J_{\eta }\right\vert \leq K\text{. }  \label{unitsphere}
\end{equation}%
Let $(Z_{1},\cdots ,Z_{2n})=(X_{1},\cdots ,X_{2n})A_{\eta }^{-1}$. $%
\{Z_{1},\cdots ,Z_{2n}\}$ is an orthonormal basis of $\mathfrak{g}_{1}$ such
that the corresponding metrix of $\omega _{\eta }$ is $J_{2n}$, i.e.
\begin{equation*}
\omega _{\eta }(Z_{j},Z_{k})=(J_{2n})_{jk}.
\end{equation*}%
The new coordinates of the element in $\mathfrak{g}_{1}$ in term of the
basis $\{Z_{1},\cdots ,Z_{2n}\}$ is $z=A_{\eta }x$. For any $\lambda >0$, we
have that

\begin{equation*}
\left( Z_{j}f\right) ^{\lambda \eta }(z)=\left( \frac{\partial }{\partial
z_{j}}-\frac{i\lambda }{2}\left\langle z,J_{2n}e_{j}\right\rangle \right)
f^{\lambda \eta }(z).
\end{equation*}%
Set
\begin{equation*}
Z_{j}^{\lambda \eta }=\frac{\partial }{\partial z_{j}}-\frac{i\lambda }{2}%
\left\langle z,J_{2n}e_{j}\right\rangle .
\end{equation*}
More explicitly,
\begin{equation*}
Z_{j}^{\lambda \eta }=\frac{\partial }{\partial z_{j}}+\frac{i\lambda }{2}%
z_{n+j}\text{, \ \ }Z_{j+n}^{\lambda \eta }=\frac{\partial }{\partial z_{j+n}%
}-\frac{i\lambda }{2}z_{j}\text{, \ \ }j=1,\cdots ,n\text{.}
\end{equation*}%
It follows that
\begin{eqnarray*}
L^{\lambda \eta } &=&-\sum_{j=1}^{2n}\left( Z_{j}^{\lambda \eta }\right)
^{2}=-\sum_{j=1}^{2n}\frac{\partial ^{2}}{\partial z_{j}^{2}}-i\lambda
\sum_{j=1}^{n}\left( z_{n+j}\frac{\partial }{\partial z_{j}}-z_{j}\frac{%
\partial }{\partial z_{n+j}}\right) +\frac{\lambda ^{2}}{4}%
\sum_{j=1}^{2n}z_{j}^{2} \\
&=&-\Delta _{z}-i\lambda \sum_{j=1}^{n}\left( z_{n+j}\frac{\partial }{%
\partial z_{j}}-z_{j}\frac{\partial }{\partial z_{n+j}}\right) +\frac{%
\lambda ^{2}}{4}\left\vert z\right\vert ^{2},
\end{eqnarray*}%
and
\begin{equation*}
\left( \mathcal{L}f\right) ^{\lambda \eta }=L^{\lambda\eta }f^{\lambda \eta
}.
\end{equation*}%
Notice that $L^{\lambda \eta }$ coincides with the the $\lambda $-scaled
special Hermite operator $L^{\lambda }$ on $\mathfrak{g}_{1}$. The $\lambda $%
-twisted convolution of $f,g\in \mathscr{S}(\mathfrak{g}_{1})$ in terms of
the coordinates $z$ is given by
\begin{eqnarray*}
\phi \times _{\lambda }\psi (z) &=&\phi \times _{\lambda \eta }\psi (z) \\
&=&\int_{\mathfrak{g}_{1}}\phi (z-w)\psi (w)e^{\frac{i}{2}\lambda \eta
([z,w])}dw \\
&=&\int_{\mathfrak{g}_{1}}\phi (z-w)\psi (w)e^{\frac{i}{2}\lambda
\sum_{j=1}^{n}z_{n+j}w_{j}-z_{j}w_{n+j}}dw.
\end{eqnarray*}

Next, we give the spectral decomposition of the sub-Laplacian $\mathcal{L}$
on $\mathbb{G}$. We have to recall some facts about the special Hermite
expansion. The Hermite polynomials $H_{k}$ on $\mathbb{\mathbb{R}}$ is
defined by
\begin{equation*}
H_{k}(t)=(-1)^{k}\frac{d^{k}}{dx^{k}}\left( e^{-t^{2}}\right)
e^{t^{2}},\qquad k=0,1,2,\cdots .
\end{equation*}%
Let
\begin{equation*}
h_{k}(t)=\left( 2^{k}k!\sqrt{\pi }\right) \,^{-\frac{1}{2}}H_{k}(t)\,e^{-%
\frac{1}{2}t^{2}},\qquad k=0,1,2,\cdots .
\end{equation*}%
For any multiindex $\alpha $, the Hermite functions $\Phi _{\alpha }$ on $%
\mathbb{R}^{n}$ is defined by
\begin{equation*}
\Phi _{\alpha }(y)=\prod_{j=1}^{n}h_{\alpha _{j}}(y_{j}),\quad
y=(y_{1},\cdots y_{n})\in \mathbb{R}^{n}.
\end{equation*}%
For each pair of multiindices $\alpha $, $\beta $, the special Hermite
function $\Phi _{\alpha \beta }$ on $\mathbb{C}^{n}$ is given by
\begin{equation*}
\Phi _{\alpha \beta }(z)=(2\pi )^{-\frac{n}{2}}\int_{\mathbb{R}^{n}}e^{2\pi
ix\cdot \xi }\Phi _{\alpha }\left( \xi +\frac{y}{2}\right) \Phi _{\beta
}\left( \xi -\frac{y}{2}\right) \,d\xi ,\quad z=x+iy\in \mathbb{C}^{n}.
\end{equation*}%
The special Hermite functions form a complete orthonormal system for $L^{2}(%
\mathbb{C}^{n}\mathbb{)}$. Since that \thinspace $\mathfrak{g}_{1}\simeq
\mathbb{C}^{n}$, then for any $g\in L^{2}(\mathfrak{g}_{1}\mathbb{)}$ we
have the special Hermite expansion
\begin{equation*}
g(z)=\sum_{\alpha \in \mathbb{N}^{n}}\sum_{\beta \in \mathbb{N}^{n}}\langle
f,\Phi _{\alpha \beta }\rangle \,\Phi _{\alpha \beta }(z).
\end{equation*}%
We define the Laguerre functions on $\mathbb{C}^{n}$ by
\begin{equation*}
\varphi _{k}(z)=L_{k}^{n-1}\left( \frac{1}{2}|z|^{2}\right) e^{-\frac{1}{4}%
|z|^{2}},\text{ }k=0,1,2,\cdots ,
\end{equation*}%
where
\begin{equation*}
L_{k}^{\nu }(x)=\frac{1}{k!}\frac{d^{k}}{dx^{k}}\left( e^{-x}x^{k+\nu
}\right) e^{x}x^{-\nu },\quad k=0,1,2,\cdots ,\,\,\nu >-1,
\end{equation*}%
is the Laguerre polynomial on $%
\mathbb{R}
$ of type $\nu $ and degree $k$. It follows that
\begin{equation*}
\sum_{|\alpha |=k}\Phi _{\alpha \alpha }(z)=(2\pi )^{-\frac{n}{2}}\varphi
_{k}(z),
\end{equation*}%
and the special Hermite expansion can be written as%
\begin{equation*}
g(z)=(2\pi )^{-n}\sum_{k=0}^{\infty }g\times _{\lambda =1}\varphi _{k}(z).
\end{equation*}%
Set $\varphi _{k}^{\lambda }(z)=\varphi _{k}(\sqrt{\lambda }z)$ for any $%
\lambda >0$. By a dilation argument, we have the expansion
\begin{equation}
g(z)=\left( \frac{\lambda }{2\pi }\right) ^{n}\sum_{k=0}^{\infty }g\times
_{\lambda }\varphi _{k}^{\lambda }(z),  \label{de1}
\end{equation}%
where $g\times _{\lambda }\varphi _{k}^{\lambda }$ is the eigenfunction of
the operator $L^{\lambda }$ with the eigenvalue $\lambda (2k+n)$.

Let $g_{\eta }=g\circ A_{\eta }^{-1}$. (\ref{de1}) yields that
\begin{equation}
g(x)=g_{\eta }(A_{\eta }x)=g_{\eta }(z)=\left( \frac{\lambda }{2\pi }\right)
^{n}\sum_{k=0}^{\infty }g_{\eta }\times _{\lambda }\varphi _{k}^{\lambda
}(z)=\left( \frac{\lambda }{2\pi }\right) ^{n}\sum_{k=0}^{\infty }\left(
\left( g_{\eta }\times _{\lambda }\varphi _{k}^{\lambda }\right) \circ
A_{\eta }\right) (x).  \label{de2}
\end{equation}%
So, we can get the following results:

\begin{proposition}
\cite{Cas} \label{eigenvalue}Let $\eta \in \mathbb{S}^{m-1}$ and $\lambda
\in \mathbb{R}.$ For any $g\in \mathscr{S}(\mathfrak{g}_{1})$, $\Pi
_{k}^{\lambda \eta }g=\left( g_{\eta }\times _{\lambda }\varphi
_{k}^{\lambda }\right) \circ A_{\eta }$ is the eigenfunction of the operator
$\Delta ^{\lambda \eta }$ with the eigenvalue $\lambda (2k+n\dot{)}$, and $%
e^{-i\lambda \eta (u)}\Pi _{k}^{\lambda \eta }g=e^{-i\lambda \eta (u)}\left(
g_{\eta }\times _{\lambda }\varphi _{k}^{\lambda }\right) \circ A_{\eta }$
is the eigenfunction of the sub-Laplacian $\mathcal{L}$ on $\mathbb{G}$ with
the eigenvalue $\lambda (2k+n)$.
\end{proposition}

\begin{proof}
Since that $(X_{1}^{\lambda \eta },\cdots ,X_{2n}^{\lambda \eta
})=(Z_{1}^{\lambda \eta },\cdots ,Z_{2n}^{\lambda \eta })A_{\eta }$, then
\begin{equation*}
\Delta ^{\lambda \eta }\left( g\circ A_{\eta }\right) =\left( L^{\lambda
}g\right) \circ A_{\eta }.
\end{equation*}%
Notice that $\Pi _{k}^{\lambda \eta }g=g_{\eta }\times _{\lambda }\varphi
_{k}^{\lambda }$ is the eigenfunction of $L^{\lambda }$ with the eigenvalue $%
\lambda (2k+n)$. We have
\begin{equation*}
\Delta ^{\lambda \eta }\left( \left( g_{\eta }\times _{\lambda }\varphi
_{k}^{\lambda }\right) \circ A_{\eta }\right) =L^{\lambda \eta }(g_{\eta
}\times _{\lambda }\varphi _{k}^{\lambda })\circ A_{\eta }=\lambda
(2k+n)(g_{\eta }\times _{\lambda }\varphi _{k}^{\lambda })\circ A_{\eta }.
\end{equation*}%
This implies $(g_{\eta }\times _{\lambda }\varphi _{k}^{\lambda })\circ
A_{\eta }$ is the eigenfunction of the operator $\Delta ^{\lambda \eta }$
with the eigenvalue $\lambda (2k+n)$.

Since that for any $f\in \mathscr{S}(\mathbb{G})$, $\left( \mathcal{L}%
f\right) ^{\lambda \eta }=\Delta ^{\lambda \eta }f^{\lambda \eta }$, then we
have that
\begin{eqnarray*}
\left( \mathcal{L}\left( e^{-i\lambda \eta (u)}\Pi _{k}^{\lambda \eta
}g\right) \right) ^{\lambda \eta } &=&\Delta ^{\lambda \eta }\left(
e^{-i\lambda \eta (u)}\Pi _{k}^{\lambda \eta }g\right) ^{\lambda \eta
}=\left( e^{-i\lambda \eta (u)}\right) ^{\lambda \eta }\Delta ^{\lambda \eta
}\left( \Pi _{k}^{\lambda \eta }g\right) \\
&=&\lambda (2k+n)\left( e^{-i\lambda \eta (u)}\right) ^{\lambda \eta }\Pi
_{k}^{\lambda \eta }g=\lambda (2k+n)\left( e^{-i\lambda \eta (u)}\Pi
_{k}^{\lambda \eta }g\right) ^{\lambda \eta }.
\end{eqnarray*}%
This yields that $\mathcal{L}\left( e^{-i\lambda \eta (u)}\Pi _{k}^{\lambda
\eta }g\right) =\lambda (2k+n)e^{-i\lambda \eta (u)}\Pi _{k}^{\lambda \eta
}g $, and we can conclude that $e^{-i\lambda \eta (u)}\Pi _{k}^{\lambda \eta
}g$ is the eigenfunction of $\mathcal{L}$ with the eigenvalue $\lambda
(2k+n) $.
\end{proof}

Define
\begin{equation*}
e_{k}^{\lambda \eta }(x,u)=e^{-i\lambda \eta (u)}\left( \varphi
_{k}^{\lambda }\circ A_{\eta }\right) (x)\left\vert \det A_{\eta
}\right\vert .
\end{equation*}%
We can verify that
\begin{equation*}
\left( f\ast e_{k}^{\lambda \eta }\right) (x,u)=e^{-i\lambda \eta (u)}\left(
\left( f_{\eta }^{\lambda \eta }\times _{\lambda }\varphi _{k}^{\lambda
}\right) \circ A_{\eta }\right) (x)\text{. }
\end{equation*}%
The Proposition \ref{eigenvalue} tells that $f\ast e_{k}^{\lambda \eta }$ is
the eigenfunction of $\mathcal{L}$ with the eigenvalue $\lambda (2k+n)$. Let
$\widetilde{e}_{k}^{\lambda \eta }=e_{k}^{\frac{\lambda }{2k+n}\eta }$.
Clearly, $f\ast \widetilde{e}_{k}^{\lambda \eta }$ is the eigenfunction of $%
\mathcal{L}$ with the eigenvalue $\lambda $.

The inversion formula of the Fourier transform and (\ref{de2}) imply that
for any $f\in \mathscr{S}(\mathbb{G})$,
\begin{eqnarray}
f(x,u) &=&\frac{1}{\left( 2\pi \right) ^{m}}\int_{0}^{\infty }\int_{\mathbb{S%
}^{m-1}}f^{\lambda \eta }(x)e^{-i\lambda \eta (u)}\lambda ^{m-1}d\sigma
(\eta )d\lambda  \notag \\
&=&\frac{1}{\left( 2\pi \right) ^{m}}\int_{0}^{\infty }\int_{\mathbb{S}%
^{m-1}}\left( \frac{\lambda }{2\pi }\right) ^{n}\sum_{k=0}^{\infty
}e^{-i\lambda \eta (u)}\left( \left( f_{\eta }^{\lambda \eta }\times
_{\lambda }\varphi _{k}^{\lambda }\right) \circ A_{\eta }\right) (x)\lambda
^{m-1}d\sigma (\eta )d\lambda  \notag \\
&=&\frac{1}{\left( 2\pi \right) ^{m}}\int_{0}^{\infty }\int_{\mathbb{S}%
^{m-1}}\left( \frac{\lambda }{2\pi }\right) ^{n}\sum_{k=0}^{\infty }\left(
f\ast e_{k}^{\lambda \eta }\right) (x,u)\lambda ^{m-1}d\sigma (\eta )d\lambda
\notag \\
&=&\int_{0}^{\infty }\left( \sum_{k=0}^{\infty }\frac{\lambda ^{n+m-1}}{%
\left( 2\pi (2k+n)\right) ^{n+m}}\int_{\mathbb{S}^{m-1}}f\ast \widetilde{e}%
_{k}^{\lambda \eta }(x,u)d\sigma (\eta )\right) d\lambda .  \label{expansion}
\end{eqnarray}%
Here $d\sigma (\eta )$ denotes the induced Lebesgue measure on the unit
sphere $\mathbb{S}^{m-1}.$ Define
\begin{equation*}
P_{\lambda }f(x,u)=\sum_{k=0}^{\infty }\frac{\lambda ^{n+m-1}}{\left( 2\pi
(2k+n)\right) ^{n+m}}\int_{\mathbb{S}^{m-1}}f\ast \widetilde{e}_{k}^{\lambda
\eta }(x,u)d\sigma (\eta ).
\end{equation*}%
$P_{\lambda }f$ is the eigenfunction of the sub-Laplacian $\mathcal{L}$ with
the eigenvalue $\lambda $. Then, we have that
\begin{equation*}
f=\int_{0}^{\infty }P_{\lambda }fd\lambda .
\end{equation*}%
The spectral decomposition of the sub-Laplacian $\mathcal{L}$ on $\mathbb{G}$
is given by
\begin{equation*}
\mathcal{L}f=\int_{0}^{\infty }\lambda P_{\lambda }fd\lambda .
\end{equation*}

We now define the Riesz means associated to the sub-Laplacian $\mathcal{L}$
for $f\in \mathscr{S}(\mathbb{G})$ by
\begin{equation*}
S_{R}^{\delta }f=\int_{0}^{\infty }\left( 1-\frac{\lambda }{R}\right)
_{+}^{\delta }P_{\lambda }fd\lambda .
\end{equation*}%
Since that $P_{\lambda }$ is the convolution operator, we have
\begin{equation*}
S_{R}^{\delta }f(x,u)=\int_{\mathbb{G}}f((x,u)\cdot (y,v)^{-1})S_{R}^{\delta
}(y,v)dydv,
\end{equation*}%
where the kernel is given by
\begin{eqnarray*}
S_{R}^{\delta }(x,u) &=&\int_{0}^{\infty }\left( 1-\frac{\lambda }{R}\right)
_{+}^{\delta }\left( \sum_{k=0}^{\infty }\frac{\lambda ^{n+m-1}}{\left( 2\pi
(2k+n)\right) ^{n+m}}\int_{\mathbb{S}^{m-1}}\widetilde{e}_{k}^{\lambda \eta
}(x,u)d\sigma (\eta )\right) d\lambda \\
&=&\frac{1}{(2\pi )^{\frac{Q}{2}}}\sum_{k=0}^{\infty }\int_{0}^{\infty
}\left( 1-\frac{(2k+n)\lambda }{R}\right) _{+}^{\delta }\left( \int_{\mathbb{%
S}^{m-1}}e_{k}^{\lambda \eta }(x,u)d\sigma (\eta )\right) \lambda
^{n+m-1}d\lambda \\
&=&\frac{1}{(2\pi )^{\frac{Q}{2}}}\sum_{k=0}^{\infty }\int_{%
\mathbb{R}
^{m}}\left( 1-\frac{(2k+n)\left\vert \mu \right\vert }{R}\right)
_{+}^{\delta }e_{k}^{\mu }(x,u)\left\vert \mu \right\vert ^{n}d\mu .
\end{eqnarray*}%
Because
\begin{equation*}
S_{R}^{\delta }(x,u)=R^{\frac{Q}{2}}S_{1}^{\delta }\left( \sqrt{R}%
x,Ru\right) ,
\end{equation*}%
we see that the $L^{p}$-boundedness of $S_{R}^{\delta }$ can be deduced from
the $L^{p}$-boundedness of $S_{1}^{\delta }$. We write $S^{\delta
}=S_{1}^{\delta }$ and it is suffices to consider the bounds of $S^{\delta }$%
.

The bilinear Riesz means associated to $\mathcal{L}$ for $f,g\in \mathscr{S}(%
\mathbb{G})$ is defined by
\begin{equation*}
S_{R}^{\alpha }(f,g)=\int_{0}^{\infty }\int_{0}^{\infty }\left( 1-\frac{%
\lambda _{1}+\lambda _{2}}{R}\right) _{+}^{\alpha }P_{\lambda
_{1}}fP_{\lambda _{2}}gd\lambda _{1}d\lambda _{2}.
\end{equation*}%
It is easy to see that
\begin{eqnarray*}
S_{R}^{\alpha }(f,g)(x,u) &=&\int_{\mathbb{G}}\int_{\mathbb{G}}f((x,u)\cdot
(x_{1},u_{1})^{-1})g((x,u)\cdot (x_{2},u_{2})^{-1}) \\
&&\times S_{R}^{\delta
}((x_{1},u_{1}),(x_{2},u_{2}))dx_{1}du_{1}dx_{2}du_{2},
\end{eqnarray*}%
where the kernel is given by
\begin{eqnarray*}
&&S_{R}^{\alpha }((x_{1},u_{1}),(x_{2},u_{2})) \\
&=&\frac{1}{(2\pi )^{Q}}\sum_{k=0}^{\infty }\sum_{l=0}^{\infty
}\int_{0}^{\infty }\int_{0}^{\infty }\left( 1-\frac{(2k+n)\lambda
_{1}+(2l+n)\lambda _{2}}{R}\right) _{+}^{\alpha }\int_{\mathbb{S}%
^{m-1}}e_{k}^{\lambda _{1}\eta _{1}}(x_{1},u_{1})d\sigma (\eta _{1}) \\
&&\text{ \ \ \ \ \ \ \ \ \ \ \ \ \ \ \ \ \ \ \ \ \ \ }\times \int_{\mathbb{S}%
^{m-1}}e_{k}^{\lambda _{2}\eta _{2}}(x_{2},u_{2})d\sigma (\eta _{2})\lambda
_{1}^{n+m-1}\lambda _{2}^{n+m-1}d\lambda _{1}d\lambda _{2} \\
&=&\frac{1}{(2\pi )^{Q}}\sum_{k=0}^{\infty }\sum_{l=0}^{\infty }\int_{%
\mathbb{R}
^{m}}\int_{%
\mathbb{R}
^{m}}\left( 1-\frac{(2k+n)\left\vert \mu _{1}\right\vert +(2l+n)\left\vert
\mu _{2}\right\vert }{R}\right) _{+}^{\alpha } \\
&&\text{ \ \ \ \ \ \ \ \ \ \ \ \ \ \ \ \ \ \ \ \ \ \ \ }\times e_{k}^{\mu
_{1}}(x_{1},u_{1})e_{k}^{\mu _{2}}(x_{2},u_{2})\left\vert \mu
_{1}\right\vert ^{n}\left\vert \mu _{2}\right\vert ^{n}d\mu _{1}d\mu _{2}.
\end{eqnarray*}%
Note that
\begin{equation*}
S_{R}^{\alpha }((x_{1},u_{1}),(x_{2},u_{2}))=R^{Q}S_{1}^{\alpha }\left( (%
\sqrt{R}x_{1},Ru_{1}),(\sqrt{R}x_{2},Ru_{2})\right) .
\end{equation*}%
By a dilation argument, the $L^{p_{1}}\times L^{p_{2}}\rightarrow L^{p}$
boundedness of $S_{R}^{\alpha }$ is deduced from the $L^{p_{1}}\times
L^{p_{2}}\rightarrow L^{p}$ boundedness of $S_{1}^{\alpha }$ as $%
1/p=1/p_{1}+1/p_{2}$. Thus, we will concentrate on the operator $%
S_{1}^{\alpha }$ and write $S^{\alpha }=S_{1}^{\alpha }$.

\section{Pointwise estimate for the kernel}

In this section, we investigate the pointwise estimates for the kernel of
the Riesz means $S^{\delta }$ and the bilinear Riesz means $S^{\alpha }$.

\begin{theorem}
\label{kernel}Let $S^{\delta }(x,u)$ be the kernel of the Riesz means $%
S^{\delta }$. Assume that $N$ is a positive integer. If $\delta >2N-1$,
then,
\begin{equation*}
\left\vert S^{\delta }(x,u)\right\vert \leq C_{N}\left( 1+\left\vert \left(
A_{\frac{\mu }{\left\vert \mu \right\vert }}x,u\right) \right\vert \right)
^{-2N}.
\end{equation*}%
for any $\omega =(x,u)\in \mathbb{G}$ and $\mu \in
\mathbb{R}
^{m}\backslash \{0\}$.
\end{theorem}

\begin{proof}
Since that
\begin{equation*}
e_{k}^{\mu }(x,u)=e^{-i\mu (u)}\left( \varphi _{k}^{\left\vert \mu
\right\vert }\circ A_{\frac{\mu }{\left\vert \mu \right\vert }}\right)
(x)\left\vert \det A_{\frac{\mu }{\left\vert \mu \right\vert }}\right\vert ,
\end{equation*}%
we have
\begin{eqnarray*}
S^{\delta }(x,u) &=&\frac{1}{(2\pi )^{n+m}}\sum_{k=0}^{\infty }\int_{%
\mathbb{R}
^{m}}\left( 1-(2k+n)\left\vert \mu \right\vert \right) _{+}^{\delta
}e_{k}^{\mu }(x,u)\left\vert \mu \right\vert ^{n}d\mu \\
&=&\frac{1}{(2\pi )^{n+m}}\sum_{k=0}^{\infty }\int_{%
\mathbb{R}
^{m}}e^{-i\mu (u)}\left( 1-(2k+n)\left\vert \mu \right\vert \right)
_{+}^{\delta }\varphi _{k}^{\left\vert \mu \right\vert }\left( A_{\frac{\mu
}{\left\vert \mu \right\vert }}x\right) \left\vert \det A_{\frac{\mu }{%
\left\vert \mu \right\vert }}\right\vert \left\vert \mu \right\vert ^{n}d\mu
.
\end{eqnarray*}%
Let $z=A_{\frac{\mu }{\left\vert \mu \right\vert }}x$ and
\begin{equation}
F(z,u)=\frac{1}{(2\pi )^{n+m}}\sum_{k=0}^{\infty }\int_{%
\mathbb{R}
^{m}}e^{-i\mu (u)}\left( 1-(2k+n)\left\vert \mu \right\vert \right)
_{+}^{\delta }\varphi _{k}^{\left\vert \mu \right\vert }\left( z\right)
\left\vert \det A_{\frac{\mu }{\left\vert \mu \right\vert }}\right\vert
\left\vert \mu \right\vert ^{n}d\mu .  \label{expansion_S}
\end{equation}%
$F(z,u)$ is a radial fuction with respect to $z$. We set $r=\left\vert
z\right\vert $ and define
\begin{equation}
R_{k}(\mu ,F)=\frac{2^{(1-n)}k!}{(k+n-1)!}\int_{0}^{\infty }F^{\mu
}(r)\varphi _{k}^{\left\vert \mu \right\vert }(r)r^{2n-1}dr,  \label{R_0}
\end{equation}%
where
\begin{equation*}
F^{\mu }(r)=\int_{\mathbb{R}^{m}}F(r,u)e^{i\mu (u)}du.
\end{equation*}%
Using the inversion formula of the Fourier transform and the Laguerre
expension of the radial function (see \cite{Thang}), we have that
\begin{equation}
F(z,u)=\sum_{k=0}^{\infty }\int_{\mathbb{R}^{m}}e^{-i\mu (u)}R_{k}(\mu
,F)\varphi _{k}^{\left\vert \mu \right\vert }(z)\left\vert \mu \right\vert
^{n}d\mu .  \label{laguerre}
\end{equation}%
Notice that
\begin{equation*}
\left\Vert \varphi _{k}\right\Vert _{\infty }=c_{n}\frac{(k+n-1)!}{k!}.
\end{equation*}%
So, if we can show that
\begin{equation*}
\sum_{k=0}^{\infty }\int_{\mathbb{R}^{m}}\left\vert R_{k}\left( \mu
,F\right) \right\vert \frac{\left( k+n-1\right) !}{k!}\left\vert \mu
\right\vert ^{n}d\mu <\infty
\end{equation*}%
then $F$ is bounded. Similarly, if we can show that for any $N\in
\mathbb{N}
^{+}$, there exists a constant $C_{N}>0$ such that for each $j=1,2,\cdots
,m, $
\begin{equation}
\sum_{k=0}^{\infty }\int_{\mathbb{R}^{m}}\left\vert R_{k}\left( \mu ,\left(
iu_{j}-\frac{\mu _{j}}{\left\vert \mu \right\vert }\frac{1}{4}\left\vert
z\right\vert ^{2}\right) ^{N}F\right) \right\vert \frac{\left( k+n-1\right) !%
}{k!}\left\vert \mu \right\vert ^{n}d\mu <C_{N}\text{,}  \label{R_1}
\end{equation}%
then
\begin{equation*}
\left\vert (z,u)\right\vert ^{4N}\left\vert F\right\vert ^{2}=\left(
\sup_{j}\left\vert u_{j}\right\vert ^{2}+\sup_{j}\frac{\left\vert \mu
_{j}\right\vert ^{2}}{\left\vert \mu \right\vert ^{2}}\frac{1}{16}\left\vert
z\right\vert ^{4}\right) ^{N}\left\vert F\right\vert ^{2}=\sup_{j}\left\vert
\left( iu_{j}-\frac{\mu _{j}}{\left\vert \mu \right\vert }\frac{1}{4}%
\left\vert z\right\vert ^{2}\right) ^{N}F\right\vert ^{2}\leq C_{N}^{2},
\end{equation*}%
namely,
\begin{equation*}
\left\vert F(z,u)\right\vert \leq C_{N}\left( 1+\left\vert (z,u)\right\vert
\right) ^{-2N}.
\end{equation*}%
This yields that for any $\omega =(x,u)\in \mathbb{G}$ and $\mu \in
\mathbb{R}
^{m}\backslash \{0\}$,%
\begin{equation*}
\left\vert S^{\delta }(x,u)\right\vert =\left\vert F\left( A_{\frac{\mu }{%
\left\vert \mu \right\vert }}x,u\right) \right\vert \leq C_{N}\left(
1+\left\vert \left( A_{\frac{\mu }{\left\vert \mu \right\vert }}x,u\right)
\right\vert \right) ^{-2N}.
\end{equation*}%
Thus, to obtain Theorem \ref{kernel}, it suffices to prove (\ref{R_1}).
Fixing $j=1,\cdots ,m$. We first let $N=1$ and calculate $R_{k}\left( \mu
,\left( iu_{j}-\frac{\mu _{j}}{\left\vert \mu \right\vert }\frac{1}{4}%
\left\vert z\right\vert ^{2}\right) F\right) $. From (\ref{R_0}), we know
that
\begin{equation*}
R_{k}(\mu ,iu_{j}F)=\frac{2^{(1-n)}k!}{(k+n-1)!}\int_{0}^{\infty }\left(
iu_{j}F\right) ^{\mu }(r)\varphi _{k}^{\left\vert \mu \right\vert
}(r)r^{2n-1}dr.
\end{equation*}%
Since that%
\begin{equation*}
\left( iu_{j}F\right) ^{\mu }(r)=\frac{\partial }{\partial \mu _{j}}F^{\mu
}(r),
\end{equation*}%
then
\begin{equation*}
R_{k}(\mu ,iu_{j}F)=\frac{\partial }{\partial \mu _{j}}R_{k}(\mu ,F)-\frac{%
2^{(1-n)}k!}{(k+n-1)!}\int_{0}^{\infty }F^{\mu }(r)\frac{\partial }{\partial
\mu _{j}}\varphi _{k}^{\left\vert \mu \right\vert }(r)r^{2n-1}dr.
\end{equation*}%
Noticing
\begin{equation*}
\frac{\partial }{\partial \mu _{j}}\varphi _{k}^{\left\vert \mu \right\vert
}(r)=\frac{1}{2}r^{2}\frac{\mu _{j}}{\left\vert \mu \right\vert }\left(
L_{k}^{n-1}\right) ^{\prime }(\frac{1}{2}\left\vert \mu \right\vert
r^{2})e^{-\frac{1}{4}\left\vert \mu \right\vert r^{2}}-\frac{1}{4}r^{2}\frac{%
\mu _{j}}{\left\vert \mu \right\vert }L_{k}^{n-1}(\frac{1}{2}\left\vert \mu
\right\vert r^{2})e^{-\frac{1}{4}\left\vert \mu \right\vert r^{2}},
\end{equation*}%
and using the recursion formula (see \cite{Thang})
\begin{equation*}
r\frac{d}{dr}L_{k}^{n-1}(r)=kL_{k}^{n-1}(r)-(k+n-1)L_{k-1}^{n-1}(r),
\end{equation*}%
we have that
\begin{equation}
\frac{\partial }{\partial \mu _{j}}\varphi _{k}^{\left\vert \mu \right\vert
}(r)=\frac{\mu _{j}}{\left\vert \mu \right\vert }\left( \left\vert \mu
\right\vert ^{-1}k\varphi _{k}^{\left\vert \mu \right\vert }(r)-\left\vert
\mu \right\vert ^{-1}(k+n-1)\varphi _{k-1}^{\left\vert \mu \right\vert }(r)-%
\frac{1}{4}r^{2}\varphi _{k}^{\left\vert \mu \right\vert }(r)\right) .
\label{prime}
\end{equation}%
Then,
\begin{equation*}
R_{k}\left( \mu ,\left( iu_{j}-\frac{\mu _{j}}{\left\vert \mu \right\vert }%
\frac{1}{4}r^{2}\right) F\right) =\frac{\partial }{\partial \mu _{j}}%
R_{k}(\mu ,F)-\frac{\mu _{j}}{\left\vert \mu \right\vert }\left( \frac{k}{%
\left\vert \mu \right\vert }R_{k}(\mu ,F)-\frac{k}{\left\vert \mu
\right\vert }R_{k-1}(\mu ,F)\right) .
\end{equation*}%
(\ref{expansion_S}) and (\ref{laguerre}) yield that
\begin{equation*}
R_{k}(\mu ,F)=(1-(2k+n)\left\vert \mu \right\vert )_{+}^{\delta }\left\vert
\det A_{\frac{\mu }{\left\vert \mu \right\vert }}\right\vert .
\end{equation*}%
We set \thinspace $\sigma =(2k+n)\left\vert \mu \right\vert $ and $\psi
(\sigma )=(1-\sigma )_{+}^{\delta }$. Then, $R_{k}\left( \mu ,\left( iu_{j}-%
\frac{\mu _{j}}{\left\vert \mu \right\vert }\frac{1}{4}r^{2}\right) F\right)
$ can be rewritten as
\begin{eqnarray*}
&&R_{k}\left( \mu ,\left( iu_{j}-\frac{\mu _{j}}{\left\vert \mu \right\vert }%
\frac{1}{4}r^{2}\right) F\right) \\
&=&\frac{\partial }{\partial \mu _{j}}\left( \left\vert \det A_{\frac{\mu }{%
\left\vert \mu \right\vert }}\right\vert \psi (\sigma )\right) -\left\vert
\det A_{\frac{\mu }{\left\vert \mu \right\vert }}\right\vert \frac{\mu _{j}}{%
\left\vert \mu \right\vert }\left( \frac{k}{\left\vert \mu \right\vert }\psi
(\sigma )-\frac{k}{\left\vert \mu \right\vert }\psi (\sigma -2\left\vert \mu
\right\vert )\right) \\
&=&\frac{\partial }{\partial \mu _{j}}\left( \left\vert \det A_{\frac{\mu }{%
\left\vert \mu \right\vert }}\right\vert \psi (\sigma )\right) -\left\vert
\det A_{\frac{\mu }{\left\vert \mu \right\vert }}\right\vert k\frac{\mu _{j}%
}{\left\vert \mu \right\vert ^{2}}\frac{\partial }{\partial k}\psi \left(
(2k+n)\left\vert \mu \right\vert \right) \\
&&+\left\vert \det A_{\frac{\mu }{\left\vert \mu \right\vert }}\right\vert k%
\frac{\mu _{j}}{\left\vert \mu \right\vert ^{2}}\left( \frac{\partial }{%
\partial k}\psi \left( (2k+n)\left\vert \mu \right\vert \right) -\psi
((2k+n)\left\vert \mu \right\vert )+\psi (\left( 2k-2+n\right) \left\vert
\mu \right\vert )\right).
\end{eqnarray*}%
Using Taylor expansion, we get that
\begin{eqnarray*}
&&\frac{\partial }{\partial k}\psi ((2k+n)\left\vert \mu \right\vert )-\psi
((2k+n)\left\vert \mu \right\vert )+\psi ((2k-2+n)\left\vert \mu \right\vert
) \\
&=&4\left\vert \mu \right\vert ^{2}\int_{k-1}^{k}(s+1-k)\psi ^{\prime \prime
}((2s+n)\left\vert \mu \right\vert )ds.
\end{eqnarray*}%
It follows that
\begin{eqnarray*}
&&\int_{\mathbb{R}^{m}}k\frac{\left\vert \mu _{j}\right\vert }{\left\vert
\mu \right\vert ^{2}}\left\vert \frac{\partial }{\partial k}\psi
((2k+n)\left\vert \mu \right\vert )-k\psi ((2k+n)\left\vert \mu \right\vert
)+k\psi ((2k-2+n)\left\vert \mu \right\vert )\right\vert \left\vert \det A_{%
\frac{\mu }{\left\vert \mu \right\vert }}\right\vert \left\vert \mu
\right\vert ^{n}d\mu \\
&=&4k\int_{k-1}^{k}(s+1-k)\left( \int_{\mathbb{R}^{m}}\left\vert \mu
_{j}\right\vert (1-(2s+n)\left\vert \mu \right\vert )^{\delta -2}\left\vert
\det A_{\frac{\mu }{\left\vert \mu \right\vert }}\right\vert \left\vert \mu
\right\vert ^{n}d\mu \right) ds \\
&\leq &C\int_{k-1}^{k}(2s+n)^{-n-m-1}ds \\
&\leq &C(2k+n)^{-n-m}
\end{eqnarray*}%
if $(1-\left\vert \mu \right\vert )_{+}^{\delta -2}$ is integrable on $%
\mathbb{R}^{m}$. At the same time, we see that
\begin{equation*}
\frac{\partial }{\partial \mu _{j}}\left( \left\vert \det A_{\frac{\mu }{%
\left\vert \mu \right\vert }}\right\vert \psi (\sigma )\right) -\left\vert
\det A_{\frac{\mu }{\left\vert \mu \right\vert }}\right\vert k\frac{\mu _{j}%
}{\left\vert \mu \right\vert ^{2}}\frac{\partial }{\partial k}\psi \left(
(2k+n)\left\vert \mu \right\vert \right) =\psi (\sigma )\frac{\partial }{%
\partial \mu _{j}}\left\vert \det A_{\frac{\mu }{\left\vert \mu \right\vert }%
}\right\vert .
\end{equation*}%
So,
\begin{eqnarray*}
&&\int_{\mathbb{R}^{m}}\left\vert \frac{\partial }{\partial \mu _{j}}\left(
\left\vert \det A_{\frac{\mu }{\left\vert \mu \right\vert }}\right\vert \psi
(\sigma )\right) -\left\vert \det A_{\frac{\mu }{\left\vert \mu \right\vert }%
}\right\vert k\frac{\mu _{j}}{\left\vert \mu \right\vert ^{2}}\frac{\partial
}{\partial k}\psi \left( (2k+n)\left\vert \mu \right\vert \right)
\right\vert \left\vert \mu \right\vert ^{n}d\mu \\
&=&\int_{\mathbb{R}^{m}}\psi (\sigma )\left\vert \frac{\partial }{\partial
\mu _{j}}\left\vert \det A_{\frac{\mu }{\left\vert \mu \right\vert }%
}\right\vert \right\vert \left\vert \mu \right\vert ^{n}d\mu \\
&\leq &\int_{\mathbb{R}}\psi ((2k+n)\lambda )\left( \int_{\mathbb{S}%
^{m-1}}\left\vert \frac{\partial }{\partial \eta _{j}}\left\vert \det
A_{\eta }\right\vert \right\vert d\eta \right) \lambda ^{n+m-1}d\lambda \\
&\leq &(2k+n)^{-n-m}.
\end{eqnarray*}%
An iteration of the process shows that $R_{k}\left( \mu ,\left( iu_{j}-\frac{%
\mu _{j}}{\left\vert \mu \right\vert }\frac{1}{4}r^{2}\right) ^{l}F\right) $
also satisfies
\begin{equation*}
\int_{\mathbb{R}^{m}}\left\vert R_{k}\left( \mu ,\left( iu_{j}-\frac{\mu _{j}%
}{\left\vert \mu \right\vert }\frac{1}{4}r^{2}\right) ^{l}F\right)
\right\vert \left\vert \mu \right\vert ^{n}d\mu \leq C_{l}(2k+n)^{-n-m}
\end{equation*}%
provided that $(1-\left\vert \mu \right\vert )_{+}^{\delta -2}$ is
integrable on $\mathbb{R}^{m}$. Thus, when $l=N$ and $\delta >2N-1$ we have
\begin{equation*}
\int_{\mathbb{R}^{m}}\left\vert R_{k}\left( \mu ,\left( iu_{j}-\frac{\mu _{j}%
}{\left\vert \mu \right\vert }\frac{1}{4}r^{2}\right) ^{N}F\right)
\right\vert \left\vert \mu \right\vert ^{n}d\mu \leq C_{N}(2k+n)^{-n-m},
\end{equation*}%
which implies that
\begin{equation*}
\sum_{k=0}^{\infty }\int_{\mathbb{R}^{m}}\left\vert R_{k}\left( \mu ,\left(
iu_{j}-\frac{\mu _{j}}{\left\vert \mu \right\vert }\frac{1}{4}\left\vert
z\right\vert ^{2}\right) ^{N}F\right) \right\vert \frac{\left( k+n-1\right) !%
}{k!}\left\vert \mu \right\vert ^{n}d\mu <\infty .
\end{equation*}%
The proof of Theorem \ref{kernel} is completed.
\end{proof}

\begin{corollary}
\label{cor1}Let $1\leq p\leq \infty $. If $\delta >Q+1$, then $S^{\delta }$
is bounded from $L^{p}(\mathbb{G})$ into $L^{p}(\mathbb{G})$.
\end{corollary}

\begin{proof}
We take $N=\frac{Q}{2}+1$. If $\delta >Q+1$, we have $\delta >2N-1$. Then,
Theorem \ref{kernel} is available. By H\"{o}lder's inequality and Young's
inequality, we conclude that
\begin{eqnarray*}
\left\Vert S^{\delta }f\right\Vert _{p} &\leq &\left\Vert f\right\Vert
_{p}\int_{\mathbb{G}}\left( 1+\left\vert \left( A_{\frac{\mu _{1}}{%
\left\vert \mu _{1}\right\vert }}x,u\right) \right\vert \right) ^{-2N}dxdu \\
&\leq &\left\Vert f\right\Vert _{p}\left\vert \det A_{\eta }\right\vert
^{-1}\int_{\mathbb{G}}\left( 1+\left\vert \left( x,u\right) \right\vert
\right) ^{-2N}dxdu \\
&\leq &C\left\Vert f\right\Vert _{p}\int_{1}^{\infty }t^{-2N+Q-1}du\leq
C\left\Vert f\right\Vert _{p}\text{.}
\end{eqnarray*}%
The proof is completed.
\end{proof}

Next, we show the pointwise estimate of the kernel of the bilinear Riesz
means $S^{\alpha }$.

\begin{theorem}
\label{bilinear-kernel} For any $N\in
\mathbb{N}
^{+}$, if $\alpha >4N-1$, then for any $\omega _{1}=(x_{1},u_{1}),\omega
_{2}=(x_{2},u_{2})\in \mathbb{G}$ and any $\mu _{1},\mu _{2}\in
\mathbb{R}
^{m}\backslash \{0\},$
\begin{equation*}
\left\vert S^{\alpha }\left( \left( x_{1},u_{1}\right) ,\left(
x_{2},u_{2}\right) \right) \right\vert \leq C_{N}\left( 1+\left\vert \left(
A_{\frac{\mu _{1}}{\left\vert \mu _{1}\right\vert }}x_{1},u_{1}\right)
\right\vert \right) ^{-2N}\left( 1+\left\vert \left( A_{\frac{\mu _{2}}{%
\left\vert \mu _{2}\right\vert }}x_{2},u_{2}\right) \right\vert \right)
^{-2N}.
\end{equation*}
\end{theorem}

\begin{proof}
Choose a nonnegative function $\varphi $ $\in C_{0}^{\infty }(\frac{1}{2},2)$
satisfying $\sum_{-\infty }^{\infty }\varphi \left( 2^{j}s\right) =1$, $s>0$%
. For each $j\geq 0$, we set function
\begin{equation*}
\varphi _{j}^{\alpha }\left( s,t\right) =(1-s-t)_{+}^{\alpha }\varphi
(2^{j}\left( 1-s-t\right) ),
\end{equation*}%
and define bilinear operator
\begin{equation*}
T_{j}^{\alpha }(f,g)=\int_{0}^{\infty }\int_{0}^{\infty }\varphi
_{j}^{\alpha }\left( \lambda _{1},\lambda _{2}\right) P_{\lambda
_{1}}fP_{\lambda _{2}}gd\lambda _{1}d\lambda _{2}\text{.}
\end{equation*}%
It is easy to see that
\begin{equation*}
T_{j}^{\alpha }(f,g)(\omega )=\int_{\mathbb{G}}\int_{\mathbb{G}}f(\omega
\omega _{1}^{-1})g(\omega \omega _{2}^{-1})K_{j}^{\alpha }(\omega
_{1},\omega _{2})d\omega _{1}d\omega _{2},
\end{equation*}%
where the kernel $K_{j}^{\alpha }$ is given by
\begin{eqnarray*}
&&K_{j}^{\alpha }((x_{1},u_{1}),(x_{2},u_{2})) \\
&=&\frac{1}{(2\pi )^{Q}}\sum_{k=0}^{\infty }\int_{\mathbb{R}%
^{m}}\int_{\mathbb{R}^{m}}\varphi _{j}^{\alpha }((2k+n)\left\vert \mu
_{1}\right\vert ,(2l+n)\left\vert \mu _{2}\right\vert )e_{k}^{\mu
_{1}}(x_{1},u_{1})e_{k}^{\mu _{2}}(x_{2},u_{2})\left\vert \mu
_{1}\right\vert ^{n}\left\vert \mu _{2}\right\vert ^{n}d\mu _{1}d\mu _{2} \\
&=&\frac{1}{(2\pi )^{Q}}\sum_{k=0}^{\infty }\sum_{l=0}^{\infty }\int_{%
\mathbb{R}^{m}}\int_{\mathbb{R}^{m}}e^{-i\mu _{1}(u_{1})}e^{-i\mu
_{2}(u_{2})}\varphi _{j}^{\alpha }((2k+n)\left\vert \mu _{1}\right\vert
,(2l+n)\left\vert \mu _{2}\right\vert ) \\
&&\text{ \ \ \ \ \ \ \ \ }\times \varphi _{k}^{\left\vert \mu
_{1}\right\vert }\left( A_{\frac{\mu _{1}}{\left\vert \mu _{1}\right\vert }%
}x_{1}\right) \varphi _{l}^{\left\vert \mu _{2}\right\vert }\left( A_{\frac{%
\mu _{2}}{\left\vert \mu _{2}\right\vert }}x_{2}\right) \left\vert \det A_{%
\frac{\mu _{1}}{\left\vert \mu _{1}\right\vert }}\right\vert \left\vert \det
A_{\frac{\mu _{2}}{\left\vert \mu _{2}\right\vert }}\right\vert \left\vert
\mu _{1}\right\vert ^{n}\left\vert \mu _{2}\right\vert ^{n}d\mu _{1}d\mu
_{2}.
\end{eqnarray*}%
Clearly,
\begin{equation*}
S^{\alpha }(\omega _{1},\omega _{2})=\sum_{j=0}^{\infty }K_{j}^{\alpha
}(\omega _{1},\omega _{2})\text{.}
\end{equation*}%
Fixing $j\geq 0$. We first estimate $K_{j}^{\alpha }$. Let $z_{1}=A_{\frac{%
\mu _{1}}{\left\vert \mu _{1}\right\vert }}x_{1}$, $z_{2}=A_{\frac{\mu _{2}}{%
\left\vert \mu _{2}\right\vert }}x_{2}$ and
\begin{eqnarray*}
F((z_{1},u_{1}),(z_{2},u_{2})) &=&\frac{1}{(2\pi )^{Q}}\sum_{k=0}^{\infty
}\sum_{l=0}^{\infty }\int_{\mathbb{R}^{m}}\int_{\mathbb{R}^{m}}e^{-i\mu
_{1}(u_{1})}e^{-i\mu _{2}(u_{2})}\varphi _{j}^{\alpha }((2k+n)\left\vert \mu
_{1}\right\vert ,(2l+n)\left\vert \mu _{2}\right\vert ) \\
&&\text{ \ \ \ \ \ \ \ \ }\times \varphi _{k}^{\left\vert \mu
_{1}\right\vert }\left( z_{1}\right) \varphi _{l}^{\left\vert \mu
_{2}\right\vert }\left( z_{2}\right) \left\vert \det A_{\frac{\mu _{1}}{%
\left\vert \mu _{1}\right\vert }}\right\vert \left\vert \det A_{\frac{\mu
_{2}}{\left\vert \mu _{2}\right\vert }}\right\vert \left\vert \mu
_{1}\right\vert ^{n}\left\vert \mu _{2}\right\vert ^{n}d\mu _{1}d\mu _{2}.
\end{eqnarray*}%
Then, $F((z_{1},u_{1}),(z_{2},u_{2}))$ is bi-radial with respect to $z_{1}$
and $z_{2}$. We set $r_{1}=\left\vert z_{1}\right\vert $, $r_{2}=\left\vert
z_{2}\right\vert $ and define
\begin{eqnarray}
&&R_{k,l}(\mu _{1},\mu _{2},F)  \label{equmain1} \\
&=&\frac{2^{(1-n)}k!}{(k+n-1)!}\frac{2^{(1-n)}l!}{(l+n-1)!}\int_{0}^{\infty
}\int_{0}^{\infty }F^{\mu _{1},\mu _{2}}(r_{1},r_{2})\varphi
_{k}^{\left\vert \mu _{1}\right\vert }(r_{1})\varphi _{l}^{\left\vert \mu
_{2}\right\vert }(r_{2})r_{1}^{2n-1}r_{2}^{2n-1}dr_{1}dr_{2},  \notag
\end{eqnarray}%
where
\begin{equation*}
F^{\mu _{1},\mu _{2}}(r_{1},r_{2})=\int_{\mathbb{R}^{m}}\int_{\mathbb{R}%
^{m}}e^{i\mu _{1}(u_{1})}e^{i\mu
_{2}(u_{2})}F((r_{1},u_{1}),(r_{2},u_{2}))du_{1}du_{2}.
\end{equation*}%
Using the inversion formula of the Fourier transform and the Laguerre
expension of the radial function (see \cite{Thang}), the function $F$ can be
written as
\begin{eqnarray*}
&&F((z_{1},u_{1}),(z_{2},u_{2})) \\
&=&\frac{1}{\left( 2\pi \right) ^{Q}}\int_{\mathbb{R}^{m}}\int_{\mathbb{R}%
^{m}}e^{-i\left( \mu _{1}(u_{1})+\mu _{2}(u_{2})\right) }F^{\mu _{1},\mu
_{2}}(z_{1},z_{2})\left\vert \mu _{1}\right\vert ^{n}\left\vert \mu
_{2}\right\vert ^{n}d\mu _{1}d\mu _{2} \\
&=&\frac{1}{\left( 2\pi \right) ^{Q}}\sum_{k=0}^{\infty }\sum_{l=0}^{\infty
}\int_{\mathbb{R}^{m}}\int_{\mathbb{R}^{m}}e^{-i\left( \mu _{1}(u_{1})+\mu
_{2}(u_{2})\right) }R_{k,l}(\mu _{1},\mu _{2},F)\varphi _{k}^{\left\vert \mu
_{1}\right\vert }(z_{1})\varphi _{l}^{\left\vert \mu _{2}\right\vert
}(z_{2})\left\vert \mu _{1}\right\vert ^{n}\left\vert \mu _{2}\right\vert
^{n}d\mu _{1}d\mu _{2}.
\end{eqnarray*}%
Hence, if we can show that
\begin{equation*}
\sum_{k=0}^{\infty }\sum_{l=0}^{\infty }\int_{\mathbb{R}^{m}}\int_{\mathbb{R}%
^{m}}\left\vert R_{k,l}(\mu _{1},\mu _{2},F)\right\vert \frac{\left(
k+n-1\right) !}{k!}\frac{\left( l+n-1\right) !}{l!}\left\vert \mu
_{1}\right\vert ^{n}\left\vert \mu _{2}\right\vert ^{n}d\mu _{1}d\mu
_{2}<\infty ,
\end{equation*}%
then $F$ is bounded. Let $u_{1}$ be the $p$-component of $u_{1}$, $u_{2}$ be
the $q$-component of $u_{2}$. Similarly, if we can show that for any fixed $%
j\geq 0$ and $N\in
\mathbb{N}
^{+}$, there exists a constant $C_{j,N}$ such that for any $p,q=1,2,\cdots
,m $,
\begin{eqnarray}
&&\sum_{k=0}^{\infty }\sum_{l=0}^{\infty }\int_{%
\mathbb{R}
^{m}}\int_{%
\mathbb{R}
^{m}}\left\vert R_{k,l}(\mu _{1},\mu _{2},\left( iu_{1}^{(p)}-\frac{\mu
_{1}^{(p)}}{\left\vert \mu _{1}\right\vert }\frac{1}{4}\left\vert
z_{1}\right\vert ^{2}\right) ^{N}\left( iu_{2}^{(q)}-\frac{\mu _{2}^{(q)}}{%
\left\vert \mu _{2}\right\vert }\frac{1}{4}\left\vert z_{2}\right\vert
^{2}\right) ^{N}F\right\vert  \label{R_2} \\
&&\times \frac{\left( k+n-1\right) !}{k!}\frac{\left( l+n-1\right) !}{l!}%
\left\vert \mu _{1}\right\vert ^{n}\left\vert \mu _{2}\right\vert ^{n}d\mu
_{1}d\mu _{2}\leq C_{j,N},  \notag
\end{eqnarray}%
then
\begin{eqnarray*}
&&\left\vert (z_{1},u_{1})\right\vert ^{4N}\left\vert
(z_{2},u_{2})\right\vert ^{4N}\left\vert
F((z_{1},u_{1}),(z_{2},u_{2}))\right\vert \\
&=&\left( \sup_{p}\left\vert u_{1}^{(p)}\right\vert ^{2}+\frac{%
\sup_{p}\left\vert \mu _{1}^{(p)}\right\vert ^{2}}{\left\vert \mu
_{1}\right\vert ^{2}}\frac{1}{16}\left\vert z_{1}\right\vert ^{4}\right)
^{N}\left( \sup_{q}\left\vert u_{2}^{(q)}\right\vert ^{2}+\frac{%
\sup_{q}\left\vert \mu _{2}^{(q)}\right\vert ^{2}}{\left\vert \mu
_{2}\right\vert ^{2}}\frac{1}{16}\left\vert z_{2}\right\vert ^{4}\right)
^{N}\left\vert F((z_{1},u_{1}),(z_{2},u_{2}))\right\vert \\
&=&\sup_{p}\left( \left\vert u_{1}^{(p)}\right\vert ^{2}+\frac{\left\vert
\mu _{1}^{(p)}\right\vert ^{2}}{\left\vert \mu _{1}\right\vert ^{2}}\frac{1}{%
16}\left\vert z_{1}\right\vert ^{4}\right) ^{N}\sup_{q}\left( \left\vert
u_{2}^{(q)}\right\vert ^{2}+\frac{\left\vert \mu _{2}^{(q)}\right\vert ^{2}}{%
\left\vert \mu _{2}\right\vert ^{2}}\frac{1}{16}\left\vert z_{2}\right\vert
^{4}\right) ^{N}\left\vert F((z_{1},u_{1}),(z_{2},u_{2}))\right\vert ^{2} \\
&=&\sup_{p}\left\vert \left( iu_{1}^{(p)}-\frac{\mu _{1}^{(p)}}{\left\vert
\mu _{1}\right\vert }\frac{1}{4}\left\vert z_{1}\right\vert ^{2}\right)
^{N}\left( iu_{2}^{(q)}-\frac{\mu _{2}^{(q)}}{\left\vert \mu _{2}\right\vert
}\frac{1}{4}\left\vert z_{2}\right\vert ^{2}\right) ^{N}F\right\vert ^{2} \\
&\leq &C_{j,N}^{2}.
\end{eqnarray*}%
This implies that for any $\mu ,\mu _{2}\in
\mathbb{R}
^{m}\backslash \{0\}$,
\begin{eqnarray}
\left\vert K_{j}^{\alpha }\left( \left( x_{1},u_{1}\right) ,\left(
x_{2},u_{2}\right) \right) \right\vert &=&\left\vert F\left( \left( A_{\frac{%
\mu _{1}}{\left\vert \mu _{1}\right\vert }}x_{1},u_{1}\right) ,\left( A_{%
\frac{\mu _{2}}{\left\vert \mu _{2}\right\vert }}x_{2},u_{2}\right) \right)
\right\vert  \notag \\
&\leq &C_{j,N}\left( 1+\left\vert \left( A_{\frac{\mu _{1}}{\left\vert \mu
_{1}\right\vert }}x_{1},u_{1}\right) \right\vert \right) ^{-2N}\left(
1+\left\vert \left( A_{\frac{\mu _{2}}{\left\vert \mu _{2}\right\vert }%
}x_{2},u_{2}\right) \right\vert \right) ^{-2N}.  \label{Kernel_j}
\end{eqnarray}%
We now prove (\ref{R_2}). Fixing $p,q=1,2,\cdots ,m$. We first calculate
\begin{equation*}
R_{k,l}\left( \mu _{1},\mu _{2},\left( iu_{1}^{(p)}-\frac{\mu _{1}^{(p)}}{%
\left\vert \mu _{1}\right\vert }\frac{1}{4}\left\vert z_{1}\right\vert
^{2}\right) \left( iu_{2}-\frac{\mu _{2}^{(q)}}{\left\vert \mu
_{2}\right\vert }\frac{1}{4}\left\vert z_{2}\right\vert ^{2}\right) F\right).
\end{equation*}%
From (\ref{equmain1}), we know that
\begin{eqnarray*}
&&R_{k,l}(\mu _{1},\mu _{2},iu_{1}^{(p)}\cdot iu_{2}^{(q)}F) \\
&=&\frac{2^{(1-n)}k!}{(k+n-1)!}\frac{2^{(1-n)}l!}{(l+n-1)!}\int_{0}^{\infty
}\int_{0}^{\infty }\left( iu_{1}^{(p)}\cdot iu_{2}^{(q)}F\right) ^{\mu
_{1},\mu _{2}}(r_{1},r_{2})\times \varphi _{k}^{\left\vert \mu
_{1}\right\vert }(r_{1})\varphi _{l}^{\left\vert \mu _{2}\right\vert
}(r_{2})r_{1}^{2n-1}r_{2}^{2n-1}dr_{1}dr_{2}.
\end{eqnarray*}%
Since that%
\begin{equation*}
\left( iu_{1}^{(p)}\cdot iu_{2}^{(q)}F\right) ^{\mu _{1},\mu
_{2}}(r_{1},r_{2})=\frac{\partial }{\partial \mu _{1}^{(p)}}\frac{\partial }{%
\partial \mu _{2}^{(q)}}F^{\mu _{1},\mu _{2}}(r_{1},r_{2}),
\end{equation*}%
we have
\begin{eqnarray*}
&&R_{k,l}(\mu _{1},\mu _{2},iu_{1}^{(p)}\cdot iu_{2}^{(q)}F) \\
&=&\frac{\partial }{\partial \mu _{1}^{(p)}}\frac{\partial }{\partial \mu
_{2}^{(q)}}R_{k,l}(\mu _{1},\mu _{2},F) \\
&&-\frac{2^{(1-n)}k!}{(k+n-1)!}\frac{2^{(1-n)}l!}{(l+n-1)!}\int_{0}^{\infty
}\int_{0}^{\infty }\frac{\partial }{\partial \mu _{1}^{(p)}}F^{\mu _{1},\mu
_{2}}(r_{1},r_{2})\frac{\partial }{\partial \mu _{2}^{(q)}}\left( \varphi
_{k}^{\left\vert \mu _{1}\right\vert }(r_{1})\varphi _{l}^{\left\vert \mu
_{2}\right\vert }(r_{2})r_{1}^{2n-1}r_{2}^{2n-1}\right) dr_{1}dr_{2} \\
&&-\frac{2^{(1-n)}k!}{(k+n-1)!}\frac{2^{(1-n)}l!}{(l+n-1)!}\int_{0}^{\infty
}\int_{0}^{\infty }\frac{\partial }{\partial \mu _{2}^{(p)}}F^{\mu _{1},\mu
_{2}}(r_{1},r_{2})\frac{\partial }{\partial \mu _{1}^{(q)}}\left( \varphi
_{k}^{\left\vert \mu _{1}\right\vert }(r_{1})\varphi _{l}^{\left\vert \mu
_{2}\right\vert }(r_{2})r_{1}^{2n-1}r_{2}^{2n-1}\right) dr_{1}dr_{2} \\
&&-\frac{2^{(1-n)}k!}{(k+n-1)!}\frac{2^{(1-n)}l!}{(l+n-1)!}\int_{0}^{\infty
}\int_{0}^{\infty }F^{\mu _{1},\mu _{2}}(r_{1},r_{2})\frac{\partial }{%
\partial \mu _{1}^{(p)}}\frac{\partial }{\partial \mu _{2}^{(q)}}\left(
\varphi _{k}^{\left\vert \mu _{1}\right\vert }(r_{1})\varphi
_{l}^{\left\vert \mu _{2}\right\vert }(r_{2})r_{1}^{2n-1}r_{2}^{2n-1}\right)
dr_{1}dr_{2}.
\end{eqnarray*}%
(\ref{prime}) tells
\begin{eqnarray*}
\frac{\partial }{\partial \mu _{1}^{(p)}}\varphi _{k}^{\left\vert \mu
_{1}\right\vert }(r_{1}) &=&\frac{\mu _{1}^{(p)}}{\left\vert \mu
_{1}\right\vert }\left( \left\vert \mu _{1}\right\vert ^{-1}k\varphi
_{k}^{\left\vert \mu _{1}\right\vert }(r_{1})-\left\vert \mu _{1}\right\vert
^{-1}(k+n-1)\varphi _{k-1}^{\left\vert \mu _{1}\right\vert }(r_{1})-\frac{1}{%
4}r_{1}^{2}\varphi _{k}^{\left\vert \mu _{1}\right\vert }(r_{1})\right) , \\
\frac{\partial }{\partial \mu _{2}^{(q)}}\varphi _{l}^{\left\vert \mu
_{2}\right\vert }(r) &=&\frac{\mu _{2}^{(q)}}{\left\vert \mu _{2}\right\vert
}\left( \left\vert \mu _{2}\right\vert ^{-1}l\varphi _{l}^{\left\vert \mu
_{2}\right\vert }(r_{2})-\left\vert \mu _{2}\right\vert ^{-1}(l+n-1)\varphi
_{l-1}^{\left\vert \mu _{2}\right\vert }(r_{2})-\frac{1}{4}r_{2}^{2}\varphi
_{l}^{\left\vert \mu _{2}\right\vert }(r_{2})\right) ,
\end{eqnarray*}%
then, it follows that%
\begin{eqnarray}
&&R_{k,l}(\mu _{1},\mu _{2},iu_{1}^{(p)}\cdot iu_{2}^{(q)}F)  \label{t1t2} \\
&=&\frac{\partial }{\partial \mu _{1}^{(p)}}\frac{\partial }{\partial \mu
_{2}^{(q)}}R_{k,l}(\mu _{1},\mu _{2},F)-\frac{\mu _{1}^{(p)}}{\left\vert \mu
_{1}\right\vert }\frac{\mu _{2}^{(q)}}{\left\vert \mu _{2}\right\vert }%
R_{k,l}(\mu _{1},\mu _{2},\frac{1}{4}\left\vert z_{1}\right\vert ^{2}\frac{1%
}{4}\left\vert z_{2}\right\vert ^{2}F)  \notag \\
&&+\frac{\mu _{1}^{(p)}}{\left\vert \mu _{1}\right\vert }\frac{\partial }{%
\partial \mu _{2}^{(q)}}R_{k,l}(\mu _{1},\mu _{2},\frac{1}{4}\left\vert
z_{1}\right\vert ^{2}F)+\frac{\mu _{2}^{(q)}}{\left\vert \mu _{2}\right\vert
}\frac{\partial }{\partial \mu _{1}^{(p)}}R_{k,l}(\mu _{1},\mu _{2},\frac{1}{%
4}\left\vert z_{2}\right\vert ^{2})  \notag \\
&&-l\frac{\mu _{2}^{(q)}}{\left\vert \mu _{2}\right\vert ^{2}}\frac{\partial
}{\partial \mu _{1}^{(p)}}\left( R_{k,l}(\mu _{1},\mu _{2},F)-R_{k,l-1}(\mu
_{1},\mu _{2},F)\right)  \notag \\
&&-k\frac{\mu _{1}^{(p)}}{\left\vert \mu _{1}\right\vert ^{2}}\frac{\partial
}{\partial \mu _{2}^{(q)}}\left( R_{k,l}(\mu _{1},\mu _{2},F)-R_{k-1,l}(\mu
_{1},\mu _{2},F)\right)  \notag \\
&&+kl\frac{\mu _{1}^{(p)}}{\left\vert \mu _{1}\right\vert ^{2}}\frac{\mu
_{2}^{(q)}}{\left\vert \mu _{2}\right\vert ^{2}}\left( R_{k,l-1}(\mu
_{1},\mu _{2},F)-R_{k,l}(\mu _{1},\mu _{2},F)\right)  \notag \\
&&+kl\frac{\mu _{1}^{(p)}}{\left\vert \mu _{1}\right\vert ^{2}}\frac{\mu
_{2}^{(q)}}{\left\vert \mu _{2}\right\vert ^{2}}\left( R_{k-1,l}(\mu
_{1},\mu _{2},F)-R_{k-1,l-1}(\mu _{1},\mu _{2},F)\right)  \notag \\
&&+k\frac{\mu _{1}^{(p)}}{\left\vert \mu _{1}\right\vert ^{2}}\frac{\mu
_{2}^{(q)}}{\left\vert \mu _{2}\right\vert }\left( R_{k,l}(\mu _{1},\mu _{2},%
\frac{1}{4}\left\vert z_{2}\right\vert ^{2}F)-R_{k-1,l}(\mu _{1},\mu _{2},%
\frac{1}{4}\left\vert z_{2}\right\vert ^{2}F)\right)  \notag \\
&&+l\frac{\mu _{1}^{(p)}}{\left\vert \mu _{1}\right\vert }\frac{\mu
_{2}^{(q)}}{\left\vert \mu _{2}\right\vert ^{2}}\left( R_{k,l}(\mu _{1},\mu
_{2},\frac{1}{4}\left\vert z_{1}\right\vert ^{2}F)-R_{k,l-1}(\mu _{1},\mu
_{2},\frac{1}{4}\left\vert z_{1}\right\vert ^{2}F)\right) .  \notag
\end{eqnarray}%
In the same way, we can get that%
\begin{eqnarray}
&&R_{k,l}(\mu _{1},\mu _{2},-iu_{1}^{(p)}\cdot \frac{\mu _{2}^{(q)}}{%
\left\vert \mu _{2}\right\vert }\frac{1}{4}\left\vert z_{2}\right\vert ^{2}F)
\label{T1z2} \\
&=&-\frac{\mu _{2}^{(q)}}{\left\vert \mu _{2}\right\vert }\frac{\partial }{%
\partial \mu _{1}^{(p)}}R_{k,l}(\mu _{1},\mu _{2},\frac{1}{4}\left\vert
z_{2}\right\vert ^{2}F)-\frac{\mu _{1}^{(p)}}{\left\vert \mu _{1}\right\vert
}\frac{\mu _{2}^{(q)}}{\left\vert \mu _{2}\right\vert }R_{k,l}(\mu _{1},\mu
_{2},\frac{1}{4}\left\vert z_{1}\right\vert ^{2}\frac{1}{4}\left\vert
z_{2}\right\vert ^{2}F)  \notag \\
&&+k\frac{\mu _{1}^{(p)}}{\left\vert \mu _{1}\right\vert ^{2}}\frac{\mu
_{2}^{(q)}}{\left\vert \mu _{2}\right\vert }\left( R_{k,l}(\mu _{1},\mu _{2},%
\frac{1}{4}\left\vert z_{2}\right\vert ^{2}F)-R_{k-1,l}(\mu _{1},\mu _{2},%
\frac{1}{4}\left\vert z_{2}\right\vert ^{2}F)\right) ,  \notag
\end{eqnarray}%
and
\begin{eqnarray}
&&R_{k,l}(\mu _{1},\mu _{2},-iu_{2}^{(q)}\cdot \frac{\mu _{1}^{(p)}}{%
\left\vert \mu _{1}\right\vert }\frac{1}{4}\left\vert z_{1}\right\vert ^{2}F)
\label{T2z1} \\
&=&-\frac{\mu _{1}^{(p)}}{\left\vert \mu _{1}\right\vert }\frac{\partial }{%
\partial \mu _{2}^{(q)}}R_{k,l}(\mu _{1},\mu _{2},\frac{1}{4}\left\vert
z_{1}\right\vert ^{2}F)-\frac{\mu _{1}^{(p)}}{\left\vert \mu _{1}\right\vert
}\frac{\mu _{2}^{(q)}}{\left\vert \mu _{2}\right\vert }R_{k,l}(\mu _{1},\mu
_{2},\frac{1}{4}\left\vert z_{1}\right\vert ^{2}\frac{1}{4}\left\vert
z_{2}\right\vert ^{2}F)  \notag \\
&&+l\frac{\mu _{1}^{(p)}}{\left\vert \mu _{1}\right\vert }\frac{\mu
_{2}^{(q)}}{\left\vert \mu _{2}\right\vert ^{2}}\left( R_{k,l}(\mu _{1},\mu
_{2},\frac{1}{4}\left\vert z_{1}\right\vert ^{2}F)-R_{k,l-1}(\mu _{1},\mu
_{2},\frac{1}{4}\left\vert z_{1}\right\vert ^{2}F)\right) .  \notag
\end{eqnarray}%
Combining (\ref{t1t2}), (\ref{T1z2}) and (\ref{T2z1}), we have that
\begin{eqnarray*}
&&R_{k,l}\left( \mu _{1},\mu _{2},\left( iu_{1}^{(p)}-\frac{\mu _{1}^{(p)}}{%
\left\vert \mu _{1}\right\vert }\frac{1}{4}\left\vert z_{1}\right\vert
^{2}\right) \left( iu_{2}^{(q)}-\frac{\mu _{2}^{(q)}}{\left\vert \mu
_{2}\right\vert }\frac{1}{4}\left\vert z_{2}\right\vert ^{2}\right) F\right)
\\
&=&R_{k,l}\left( \mu _{1},\mu _{2},iu_{1}^{(p)}\cdot iu_{2}^{(q)}F\right)
+R_{k,l}\left( \mu _{1},\mu _{2},-iu_{1}^{(p)}\frac{\mu _{2}^{(q)}}{%
\left\vert \mu _{2}\right\vert }\frac{1}{4}\left\vert z_{2}\right\vert
^{2}F\right) \\
&&+R_{k,l}\left( \mu _{1},\mu _{2},-iu_{2}^{(q)}\frac{\mu _{1}^{(p)}}{%
\left\vert \mu _{1}\right\vert }\frac{1}{4}\left\vert z_{1}\right\vert
^{2}F\right) +R_{k,l}\left( \mu _{1},\mu _{2},\frac{\mu _{1}^{(p)}}{%
\left\vert \mu _{1}\right\vert }\frac{1}{4}\left\vert z_{1}\right\vert
^{2}\cdot \frac{\mu _{2}^{(q)}}{\left\vert \mu _{2}\right\vert }\frac{1}{4}%
\left\vert z_{2}\right\vert ^{2}F\right) \\
&=&\frac{\partial }{\partial \mu _{1}^{(p)}}\frac{\partial }{\partial \mu
_{2}^{(q)}}R_{k,l}(\mu _{1},\mu _{2},F)-2\frac{\mu _{1}^{(p)}}{\left\vert
\mu _{1}\right\vert }\frac{\mu _{2}^{(q)}}{\left\vert \mu _{2}\right\vert }%
R_{k,l}(\mu _{1},\mu _{2},\frac{1}{4}\left\vert z_{1}\right\vert ^{2}\frac{1%
}{4}\left\vert z_{2}\right\vert ^{2}F) \\
&&-l\frac{\mu _{2}^{(q)}}{\left\vert \mu _{2}\right\vert ^{2}}\frac{\partial
}{\partial \mu _{1}^{(p)}}\left( R_{k,l}(\mu _{1},\mu _{2},F)-R_{k,l-1}(\mu
_{1},\mu _{2},F)\right) -k\frac{\mu _{1}^{(p)}}{\left\vert \mu
_{1}\right\vert ^{2}}\frac{\partial }{\partial \mu _{2}^{(q)}}\left(
R_{k,l}(\mu _{1},\mu _{2},F)-R_{k-1,l}(\mu _{1},\mu _{2},F)\right) \\
&&+kl\frac{\mu _{1}^{(p)}}{\left\vert \mu _{1}\right\vert ^{2}}\frac{\mu
_{2}^{(q)}}{\left\vert \mu _{2}\right\vert ^{2}}\left( R_{k,l-1}(\mu
_{1},\mu _{2},F)-R_{k,l}(\mu _{1},\mu _{2},F)+R_{k-1,l}(\mu _{1},\mu
_{2},F)-R_{k-1,l-1}(\mu _{1},\mu _{2},F)\right) \\
&&+2k\frac{\mu _{1}^{(p)}}{\left\vert \mu _{1}\right\vert ^{2}}\frac{\mu
_{2}^{(q)}}{\left\vert \mu _{2}\right\vert }\left( R_{k,l}(\mu _{1},\mu _{2},%
\frac{1}{4}\left\vert z_{2}\right\vert ^{2}F)-R_{k-1,l}(\mu _{1},\mu _{2},%
\frac{1}{4}\left\vert z_{2}\right\vert ^{2}F)\right) \\
&&+2l\frac{\mu _{1}^{(p)}}{\left\vert \mu _{1}\right\vert }\frac{\mu
_{2}^{(q)}}{\left\vert \mu _{2}\right\vert ^{2}}\left( R_{k,l}(\mu _{1},\mu
_{2},\frac{1}{4}\left\vert z_{1}\right\vert ^{2}F)-R_{k,l-1}(\mu _{1},\mu
_{2},\frac{1}{4}\left\vert z_{1}\right\vert ^{2}F)\right) .
\end{eqnarray*}

Let $\sigma =(2k+n)\left\vert \mu _{1}\right\vert +(2l+n)\left\vert \mu
_{2}\right\vert $ and
\begin{equation*}
\psi (\sigma )=(1-\sigma )_{+}^{\alpha }\varphi (2^{j}(1-\sigma )).
\end{equation*}%
Notice that
\begin{equation*}
R_{k,l}(\mu _{1},\mu _{2},F)=\psi (\sigma )\left\vert \det A_{\frac{\mu _{1}%
}{\left\vert \mu _{1}\right\vert }}\right\vert \left\vert \det A_{\frac{\mu
_{2}}{\left\vert \mu _{2}\right\vert }}\right\vert \text{.}
\end{equation*}
Then,
\begin{eqnarray*}
&&R_{k,l}\left( \mu _{1},\mu _{2},\left( iu_{1}^{(p)}-\frac{\mu _{1}^{(p)}}{%
\left\vert \mu _{1}\right\vert }\frac{1}{4}\left\vert z_{1}\right\vert
^{2}\right) \left( iu_{2}^{(q)}-\frac{\mu _{2}^{(q)}}{\left\vert \mu
_{2}\right\vert }\frac{1}{4}\left\vert z_{2}\right\vert ^{2}\right) F\right)
\\
&=&\left\vert \det A_{\frac{\mu _{1}}{\left\vert \mu _{1}\right\vert }%
}\right\vert \left\vert \det A_{\frac{\mu _{2}}{\left\vert \mu
_{2}\right\vert }}\right\vert \left( \frac{\partial }{\partial \mu _{1}^{(p)}%
}\frac{\partial }{\partial \mu _{2}^{(q)}}\psi (\sigma )\right. \\
&&-l\frac{\mu _{2}^{(q)}}{\left\vert \mu _{2}\right\vert ^{2}}\frac{\partial
}{\partial \mu _{1}^{(p)}}\left( \psi (\sigma )-\psi (\sigma -2\left\vert
\mu _{2}\right\vert )\right) -k\frac{\mu _{1}^{(p)}}{\left\vert \mu
_{1}\right\vert ^{2}}\frac{\partial }{\partial \mu _{2}^{(q)}}\left( \psi
(\sigma )-\psi (\sigma -2\left\vert \mu _{1}\right\vert )\right) \\
&&-kl\frac{\mu _{1}^{(p)}}{\left\vert \mu _{1}\right\vert ^{2}}\frac{\mu
_{2}^{(q)}}{\left\vert \mu _{2}\right\vert ^{2}}\left( \psi (\sigma )-\psi
(\sigma -2\left\vert \mu _{2}\right\vert )-\psi (\sigma -2\left\vert \mu
_{1}\right\vert )+\psi (\sigma -2\left\vert \mu _{1}\right\vert -2\left\vert
\mu _{2}\right\vert )\right) \\
&&\left. +C_{1}k\frac{\mu _{1}^{(p)}}{\left\vert \mu _{1}\right\vert ^{2}}%
\frac{\mu _{2}^{(q)}}{\left\vert \mu _{2}\right\vert }\left( \psi (\sigma
)-\psi (\sigma -2\left\vert \mu _{1}\right\vert )\right) +C_{2}l\frac{\mu
_{1}^{(p)}}{\left\vert \mu _{1}\right\vert }\frac{\mu _{2}^{(q)}}{\left\vert
\mu _{2}\right\vert ^{2}}\left( \psi (\sigma )-\psi (\sigma -2\left\vert \mu
_{2}\right\vert )\right) -C_{3}\psi (\sigma )\right) \\
&&+\frac{\partial }{\partial \mu _{1}^{(p)}}\frac{\partial }{\partial \mu
_{2}^{(q)}}\left\vert \det A_{\frac{\mu _{1}}{\left\vert \mu _{1}\right\vert
}}\right\vert \left\vert \det A_{\frac{\mu _{2}}{\left\vert \mu
_{2}\right\vert }}\right\vert \psi (\sigma )-l\frac{\mu _{2}^{(q)}}{%
\left\vert \mu _{2}\right\vert ^{2}}\left\vert \det A_{\frac{\mu _{2}}{%
\left\vert \mu _{2}\right\vert }}\right\vert \frac{\partial }{\partial \mu
_{1}^{(p)}}\left( \left\vert \det A_{\frac{\mu _{1}}{\left\vert \mu
_{1}\right\vert }}\right\vert \right) \left( \psi (\sigma )-\psi (\sigma
-2\left\vert \mu _{2}\right\vert )\right) \\
&&-k\frac{\mu _{1}^{(p)}}{\left\vert \mu _{1}\right\vert ^{2}}\left\vert
\det A_{\frac{\mu _{1}}{\left\vert \mu _{1}\right\vert }}\right\vert \left(
\frac{\partial }{\partial \mu _{2}^{(q)}}\left\vert \det A_{\frac{\mu _{2}}{%
\left\vert \mu _{2}\right\vert }}\right\vert \right) \left( \psi (\sigma
)-\psi (\sigma -2\left\vert \mu _{1}\right\vert )\right) .
\end{eqnarray*}%
We set
\begin{eqnarray*}
\psi _{p,q}^{1}(\sigma ) &=&\left\vert \det A_{\frac{\mu _{1}}{\left\vert
\mu _{1}\right\vert }}\right\vert \left\vert \det A_{\frac{\mu _{2}}{%
\left\vert \mu _{2}\right\vert }}\right\vert \left( \frac{\partial }{%
\partial \mu _{1}^{(p)}}\frac{\partial }{\partial \mu _{2}^{(q)}}\psi
(\sigma )\right. \\
&&-l\frac{\mu _{2}^{(q)}}{\left\vert \mu _{2}\right\vert ^{2}}\frac{\partial
}{\partial \mu _{1}^{(p)}}\left( \psi (\sigma )-\psi (\sigma -2\left\vert
\mu _{2}\right\vert )\right) -k\frac{\mu _{1}^{(p)}}{\left\vert \mu
_{1}\right\vert ^{2}}\frac{\partial }{\partial \mu _{2}^{(q)}}\left( \psi
(\sigma )-\psi (\sigma -2\left\vert \mu _{1}\right\vert )\right) \\
&&-kl\frac{\mu _{1}^{(p)}}{\left\vert \mu _{1}\right\vert ^{2}}\frac{\mu
_{2}^{(q)}}{\left\vert \mu _{2}\right\vert ^{2}}\left( \psi (\sigma )-\psi
(\sigma -2\left\vert \mu _{2}\right\vert )-\psi (\sigma -2\left\vert \mu
_{1}\right\vert )+\psi (\sigma -2\left\vert \mu _{1}\right\vert -2\left\vert
\mu _{2}\right\vert )\right) \\
&&\left. +C_{1}k\frac{\mu _{1}^{(p)}}{\left\vert \mu _{1}\right\vert ^{2}}%
\frac{\mu _{2}^{(q)}}{\left\vert \mu _{2}\right\vert }\left( \psi (\sigma
)-\psi (\sigma -2\left\vert \mu _{1}\right\vert )\right) +C_{2}l\frac{\mu
_{1}^{(p)}}{\left\vert \mu _{1}\right\vert }\frac{\mu _{2}^{(q)}}{\left\vert
\mu _{2}\right\vert ^{2}}\left( \psi (\sigma )-\psi (\sigma -2\left\vert \mu
_{2}\right\vert )\right) -C_{3}\psi (\sigma )\right) .
\end{eqnarray*}%
Notice that the function $\psi (\sigma )$ has two properties: the first one
is that $\psi (\sigma )$ is supported in a set of the form $1-2^{-j+1}\leq
(2k+n)\left\vert \mu _{1}\right\vert +(2l+n)\left\vert \mu _{2}\right\vert
\leq 1-2^{-j-1}$ for any $k$ and $l$, and the second is
\begin{equation*}
\int_{%
\mathbb{R}
^{m}}\int_{%
\mathbb{R}
^{m}}\psi ((2k+n)\left\vert \mu _{1}\right\vert ,(2l+n)\left\vert \mu
_{2}\right\vert )\left\vert \mu _{1}\right\vert ^{n}\left\vert \mu
_{2}\right\vert ^{n}d\mu _{1}d\mu _{2}\leq C2^{-j\alpha
}2^{-j}(2k+n)^{-n-m}(2l+n)^{-n-m}\text{. }
\end{equation*}%
Clearly, $\psi _{p,q}^{1}(\sigma )$ also satisfies the first property. We
claim that it satisfies
\begin{equation}
\int_{%
\mathbb{R}
^{m}}\int_{%
\mathbb{R}
^{m}}\psi _{p,q}^{1}((2k+n)\left\vert \mu _{1}\right\vert ,(2l+n)\left\vert
\mu _{2}\right\vert )\left\vert \mu _{1}\right\vert ^{n}\left\vert \mu
_{2}\right\vert ^{n}d\mu _{1}d\mu _{2}\leq C2^{-j(\alpha
-4)}2^{-j}(2k+n)^{-n-m}(2l+n)^{-n-m}.  \label{K2}
\end{equation}%
To verify (\ref{K2}), we rewrite $\psi _{p,q}^{1}(\sigma )$ as
\begin{eqnarray*}
&&\psi _{p,q}^{1}(\sigma ) \\
&=&\frac{nk}{2}\frac{\mu _{1}}{\left\vert \mu _{1}\right\vert ^{2}}\frac{\mu
_{2}}{\left\vert \mu _{2}\right\vert ^{2}}\frac{\partial }{\partial k}\frac{%
\partial }{\partial l}\psi ((2k+n)\left\vert \mu _{1}\right\vert
+(2l+n)\left\vert \mu _{2}\right\vert ) \\
&&-\frac{nk}{2}\frac{\mu _{1}}{\left\vert \mu _{1}\right\vert ^{2}}\frac{\mu
_{2}}{\left\vert \mu _{2}\right\vert ^{2}}\frac{\partial }{\partial l}\left(
\psi ((2k+n)\left\vert \mu _{1}\right\vert +(2l+n)\left\vert \mu
_{2}\right\vert )-\psi ((2k-2+n)\left\vert \mu _{1}\right\vert
+(2l+n)\left\vert \mu _{2}\right\vert )\right) \\
&&+\frac{nl}{2}\frac{\mu _{1}}{\left\vert \mu _{1}\right\vert ^{2}}\frac{\mu
_{2}}{\left\vert \mu _{2}\right\vert ^{2}}\frac{\partial }{\partial k}\frac{%
\partial }{\partial l}\psi ((2k+n)\left\vert \mu _{1}\right\vert
+(2l+n)\left\vert \mu _{2}\right\vert ) \\
&&-\frac{nl}{2}\frac{\mu _{1}}{\left\vert \mu _{1}\right\vert ^{2}}\frac{\mu
_{2}}{\left\vert \mu _{2}\right\vert ^{2}}\frac{\partial }{\partial k}\left(
\psi ((2k+n)\left\vert \mu _{1}\right\vert +(2l+n)\left\vert \mu
_{2}\right\vert )-\psi ((2k+n)\left\vert \mu _{1}\right\vert
+(2l-2+n)\left\vert \mu _{2}\right\vert )\right) \\
&&+2kl\frac{\mu _{1}}{\left\vert \mu _{1}\right\vert ^{2}}\frac{\mu _{2}}{%
\left\vert \mu _{2}\right\vert ^{2}}\frac{\partial }{\partial k}\frac{%
\partial }{\partial l}\psi ((2k+n)\left\vert \mu _{1}\right\vert
+(2l+n)\left\vert \mu _{2}\right\vert ) \\
&&-2kl\frac{\mu _{1}}{\left\vert \mu _{1}\right\vert ^{2}}\frac{\mu _{2}}{%
\left\vert \mu _{2}\right\vert ^{2}}\frac{\partial }{\partial k}\left( \psi
((2k+n)\left\vert \mu _{1}\right\vert +(2l+n)\left\vert \mu _{2}\right\vert
)-\psi ((2k+n)\left\vert \mu _{1}\right\vert +(2l-2+n)\left\vert \mu
_{2}\right\vert )\right) \\
&&+2kl\frac{\mu _{1}}{\left\vert \mu _{1}\right\vert ^{2}}\frac{\mu _{2}}{%
\left\vert \mu _{2}\right\vert ^{2}}\frac{\partial }{\partial k}\frac{%
\partial }{\partial l}\psi ((2k+n)\left\vert \mu _{1}\right\vert
+(2l+n)\left\vert \mu _{2}\right\vert ) \\
&&-2kl\frac{\mu _{1}}{\left\vert \mu _{1}\right\vert ^{2}}\frac{\mu _{2}}{%
\left\vert \mu _{2}\right\vert ^{2}}\frac{\partial }{\partial l}\left( \psi
((2k+n)\left\vert \mu _{1}\right\vert +(2l+n)\left\vert \mu _{2}\right\vert
)-\psi ((2k-2+n)\left\vert \mu _{1}\right\vert +(2l+n)\left\vert \mu
_{2}\right\vert )\right) \\
&&-kl\frac{\mu _{1}}{\left\vert \mu _{1}\right\vert ^{2}}\frac{\mu _{2}}{%
\left\vert \mu _{2}\right\vert ^{2}}\frac{\partial }{\partial k}\frac{%
\partial }{\partial l}\psi ((2k+n)\left\vert \mu _{1}\right\vert
+(2l+n)\left\vert \mu _{2}\right\vert ) \\
&&+kl\frac{\mu _{1}}{\left\vert \mu _{1}\right\vert ^{2}}\frac{\mu _{2}}{%
\left\vert \mu _{2}\right\vert ^{2}}\frac{\partial }{\partial k}\left( \psi
((2k+n)\left\vert \mu _{1}\right\vert +(2l+n)\left\vert \mu _{2}\right\vert
)-\psi ((2k+n)\left\vert \mu _{1}\right\vert +(2l-2+n)\left\vert \mu
_{2}\right\vert )\right) \\
&&+kl\frac{\mu _{1}}{\left\vert \mu _{1}\right\vert ^{2}}\frac{\mu _{2}}{%
\left\vert \mu _{2}\right\vert ^{2}}\frac{\partial }{\partial l}%
(2k+n)\left\vert \mu _{1}\right\vert +(2l+n)\left\vert \mu _{2}\right\vert )
\\
&&-kl\frac{\mu _{1}}{\left\vert \mu _{1}\right\vert ^{2}}\frac{\mu _{2}}{%
\left\vert \mu _{2}\right\vert ^{2}}\left( \psi ((2k+n)\left\vert \mu
_{1}\right\vert +(2l+n)\left\vert \mu _{2}\right\vert )-\psi
((2k+n)\left\vert \mu _{1}\right\vert +(2l-2+n)\left\vert \mu
_{2}\right\vert )\right) \\
&&-kl\frac{\mu _{1}}{\left\vert \mu _{1}\right\vert ^{2}}\frac{\mu _{2}}{%
\left\vert \mu _{2}\right\vert ^{2}}\frac{\partial }{\partial l}%
(2k-2+n)\left\vert \mu _{1}\right\vert +(2l+n)\left\vert \mu _{2}\right\vert
) \\
&&+kl\frac{\mu _{1}}{\left\vert \mu _{1}\right\vert ^{2}}\frac{\mu _{2}}{%
\left\vert \mu _{2}\right\vert ^{2}}\left( \psi ((2k-2+n)\left\vert \mu
_{1}\right\vert +(2l+n)\left\vert \mu _{2}\right\vert )-\psi
((2k-2+n)\left\vert \mu _{1}\right\vert +(2l-2+n)\left\vert \mu
_{2}\right\vert )\right) \\
&&-C_{1}k\frac{\mu _{1}}{\left\vert \mu _{1}\right\vert ^{2}}\frac{\mu _{2}}{%
\left\vert \mu _{2}\right\vert }\frac{\partial }{\partial k}\psi
((2k+n)\left\vert \mu _{1}\right\vert +(2l+n)\left\vert \mu _{2}\right\vert )
\\
&&+C_{1}k\frac{\mu _{1}}{\left\vert \mu _{1}\right\vert ^{2}}\frac{\mu _{2}}{%
\left\vert \mu _{2}\right\vert }\left( \psi ((2k+n)\left\vert \mu
_{1}\right\vert +(2l+n)\left\vert \mu _{2}\right\vert )-\psi
((2k-2+n)\left\vert \mu _{1}\right\vert +(2l+n)\left\vert \mu
_{2}\right\vert )\right) \\
&&-C_{2}l\frac{\mu _{1}}{\left\vert \mu _{1}\right\vert }\frac{\mu _{2}}{%
\left\vert \mu _{2}\right\vert ^{2}}\frac{\partial }{\partial l}\psi
((2k+n)\left\vert \mu _{1}\right\vert +(2l+n)\left\vert \mu _{2}\right\vert )
\\
&&+C_{2}l\frac{\mu _{1}}{\left\vert \mu _{1}\right\vert }\frac{\mu _{2}}{%
\left\vert \mu _{2}\right\vert ^{2}}\left( \psi ((2k+n)\left\vert \mu
_{1}\right\vert +(2l+n)\left\vert \mu _{2}\right\vert )-\psi
((2k+n)\left\vert \mu _{1}\right\vert +(2l-2+n)\left\vert \mu
_{2}\right\vert )\right) \\
&&+\left( \frac{n^{2}}{4}-2kl\right) \frac{\mu _{1}}{\left\vert \mu
_{1}\right\vert ^{2}}\frac{\mu _{2}}{\left\vert \mu _{2}\right\vert ^{2}}%
\frac{\partial }{\partial k}\frac{\partial }{\partial l}\psi
((2k+n)\left\vert \mu _{1}\right\vert +(2l+n)\left\vert \mu _{2}\right\vert )
\\
&&+C_{1}k\frac{\mu _{1}}{\left\vert \mu _{1}\right\vert ^{2}}\frac{\mu _{2}}{%
\left\vert \mu _{2}\right\vert }\frac{\partial }{\partial k}\psi
((2k+n)\left\vert \mu _{1}\right\vert +(2l+n)\left\vert \mu _{2}\right\vert
)+C_{2}l\frac{\mu _{1}}{\left\vert \mu _{1}\right\vert }\frac{\mu _{2}}{%
\left\vert \mu _{2}\right\vert ^{2}}\frac{\partial }{\partial l}\psi
((2k+n)\left\vert \mu _{1}\right\vert +(2l+n)\left\vert \mu _{2}\right\vert )
\\
&&-C_{3}\psi ((2k+n)\left\vert \mu _{1}\right\vert +(2l+n)\left\vert \mu
_{2}\right\vert ).
\end{eqnarray*}%
Using Taylor expansion, we have that
\begin{eqnarray*}
&&\psi _{p,q}^{1}(\sigma ) \\
&=&-kl\frac{\mu _{1}}{\left\vert \mu _{1}\right\vert ^{2}}\frac{\mu _{2}}{%
\left\vert \mu _{2}\right\vert ^{2}}\left( 16\left\vert \mu _{1}\right\vert
^{2}\left\vert \mu _{2}\right\vert
^{2}\int_{k-1}^{k}\int_{l-1}^{l}(t+1-k)(s+1-l)\psi ^{(4)}((2t+n)\left\vert
\mu _{1}\right\vert +(2s+n)\left\vert \mu _{2}\right\vert )dsdt\right) \\
&&+\frac{nk}{2}\frac{\mu _{1}}{\left\vert \mu _{1}\right\vert ^{2}}\frac{\mu
_{2}}{\left\vert \mu _{2}\right\vert ^{2}}\left( 8\left\vert \mu
_{1}\right\vert ^{2}\left\vert \mu _{2}\right\vert \int_{k-1}^{k}(t+1-k)\psi
^{(3)}((2t+n)\left\vert \mu _{1}\right\vert +(2l+n)\left\vert \mu
_{2}\right\vert )dt\right) \\
&&+\frac{nl}{2}\frac{\mu _{1}}{\left\vert \mu _{1}\right\vert ^{2}}\frac{\mu
_{2}}{\left\vert \mu _{2}\right\vert ^{2}}\left( 8\left\vert \mu
_{1}\right\vert \left\vert \mu _{2}\right\vert ^{2}\int_{l-1}^{l}(s+1-l)\psi
^{(3)}((2k+n)\left\vert \mu _{1}\right\vert +(2s+n)\left\vert \mu
_{2}\right\vert )ds\right) \\
&&+2kl\frac{\mu _{1}}{\left\vert \mu _{1}\right\vert ^{2}}\frac{\mu _{2}}{%
\left\vert \mu _{2}\right\vert ^{2}}\left( 8\left\vert \mu _{1}\right\vert
\left\vert \mu _{2}\right\vert ^{2}\int_{l-1}^{l}(s+1-l)\psi
^{(3)}((2k+n)\left\vert \mu _{1}\right\vert +(2s+n)\left\vert \mu
_{2}\right\vert )ds\right) \\
&&+2kl\frac{\mu _{1}}{\left\vert \mu _{1}\right\vert ^{2}}\frac{\mu _{2}}{%
\left\vert \mu _{2}\right\vert ^{2}}\frac{\partial }{\partial l}\left(
8\left\vert \mu _{1}\right\vert ^{2}\left\vert \mu _{2}\right\vert
\int_{k-1}^{k}(t+1-k)\psi ^{(3)}((2t+n)\left\vert \mu _{1}\right\vert
+(2l+n)\left\vert \mu _{2}\right\vert )dt\right) \\
&&-C_{1}k\frac{\mu _{1}}{\left\vert \mu _{1}\right\vert ^{2}}\frac{\mu _{2}}{%
\left\vert \mu _{2}\right\vert }\left( 4\left\vert \mu _{1}\right\vert
^{2}\int_{k-1}^{k}(t+1-k)\psi ^{(2)}((2t+n)\left\vert \mu _{1}\right\vert
+(2l+n)\left\vert \mu _{2}\right\vert )dt\right) \\
&&-C_{2}l\frac{\mu _{1}}{\left\vert \mu _{1}\right\vert }\frac{\mu _{2}}{%
\left\vert \mu _{2}\right\vert ^{2}}\left( 4\left\vert \mu _{2}\right\vert
^{2}\int_{l-1}^{l}(s+1-l)\psi ^{(2)}((2k+n)\left\vert \mu _{1}\right\vert
+(2s+n)\left\vert \mu _{2}\right\vert ds\right) \\
&&+(\frac{n^{2}}{4}-2kl)\frac{\mu _{1}}{\left\vert \mu _{1}\right\vert ^{2}}%
\frac{\mu _{2}}{\left\vert \mu _{2}\right\vert ^{2}}\left( 4\left\vert \mu
_{1}\right\vert \left\vert \mu _{2}\right\vert \psi ^{(2)}((2k+n)\left\vert
\mu _{1}\right\vert +(2l+n)\left\vert \mu _{2}\right\vert )\right) \\
&&+C_{1}k\frac{\mu _{1}}{\left\vert \mu _{1}\right\vert ^{2}}\frac{\mu _{2}}{%
\left\vert \mu _{2}\right\vert }2\left\vert \mu _{1}\right\vert \psi
^{(1)}((2k+n)\left\vert \mu _{1}\right\vert +(2l+n)\left\vert \mu
_{2}\right\vert ) \\
&&+C_{2}l\frac{\mu _{1}}{\left\vert \mu _{1}\right\vert }\frac{\mu _{2}}{%
\left\vert \mu _{2}\right\vert ^{2}}2\left\vert \mu _{2}\right\vert \psi
^{(1)}((2k+n)\left\vert \mu _{1}\right\vert +(2l+n)\left\vert \mu
_{2}\right\vert ) \\
&&-C_{3}\psi ((2k+n)\left\vert \mu _{1}\right\vert +(2l+n)\left\vert \mu
_{2}\right\vert ).
\end{eqnarray*}

It follows that%
\begin{eqnarray*}
&&\int_{\mathbb{R}^{m}}\int_{\mathbb{R}^{m}}\left\vert \psi
_{p,q}^{1}(\sigma )\right\vert \left\vert \mu _{1}\right\vert ^{n}\left\vert
\mu _{2}\right\vert ^{n}d\mu _{1}d\mu _{2} \\
&\leq &c_{1}2^{-j(\alpha
-4)}2^{-j}\int_{l-1}^{l}\int_{k-1}^{k}(2t+n)^{-n-m-1}(2s+n)^{-n-m-1}dtds \\
&&+c_{2}2^{-j(\alpha -3)}2^{-j}(2k+n)^{-n-m}\int_{l-1}^{l}(2s+n)^{-n-m-1}ds
\\
&&+c_{3}2^{-j(\alpha -3)}2^{-j}(2l+n)^{-n-m}\int_{k-1}^{k}(2t+n)^{-n-m-1}dt
\\
&&+c_{4}2^{-j(\alpha -2)}2^{-j}(2l+n)^{-n-m}\int_{k-1}^{k}(2t+n)^{-n-m-1}dt
\\
&&+c_{5}2^{-j(\alpha -2)}2^{-j}(2k+n)^{-n-m}\int_{l-1}^{l}(2s+n)^{-n-m-1}ds
\\
&&+c_{6}2^{-j(\alpha -2)}2^{-j}(2k+n)^{-n-m}(2s+n)^{-n-m} \\
&&+c_{7}2^{-j(\alpha -1)}2^{-j}(2k+n)^{-n-m}(2s+n)^{-n-m} \\
&&+c_{8}2^{-j\alpha }2^{-j}(2k+n)^{-n-m}(2s+n)^{-n-m} \\
&\leq &C2^{-j(\alpha -4)}2^{-j}(2k+n)^{-n-m}(2l+n)^{-n-m}.
\end{eqnarray*}%
This proves that $\psi _{p,q}^{1}(\sigma )$ has the property (\ref{K2}). At
the same time, applying the same way in Theorem \ref{kernel}, we can obtain
that
\begin{eqnarray*}
&&\int_{\mathbb{R}^{m}}\int_{\mathbb{R}^{m}}\left\vert \frac{\partial }{%
\partial \mu _{1}^{(p)}}\frac{\partial }{\partial \mu _{2}^{(q)}}\left\vert
\det A_{\frac{\mu _{1}}{\left\vert \mu _{1}\right\vert }}\right\vert
\left\vert \det A_{\frac{\mu _{2}}{\left\vert \mu _{2}\right\vert }%
}\right\vert -l\frac{\mu _{2}^{(q)}}{\left\vert \mu _{2}\right\vert ^{2}}%
\frac{\partial }{\partial \mu _{1}^{(p)}}\left( \left\vert \det A_{\frac{\mu
_{1}}{\left\vert \mu _{1}\right\vert }}\right\vert \left\vert \det A_{\frac{%
\mu _{2}}{\left\vert \mu _{2}\right\vert }}\right\vert \right) \left( \psi
(\sigma )-\psi (\sigma -2\left\vert \mu _{2}\right\vert )\right) \right. \\
&&\left. -k\frac{\mu _{1}^{(p)}}{\left\vert \mu _{1}\right\vert ^{2}}\left(
\frac{\partial }{\partial \mu _{2}^{(q)}}\left\vert \det A_{\frac{\mu _{1}}{%
\left\vert \mu _{1}\right\vert }}\right\vert \left\vert \det A_{\frac{\mu
_{2}}{\left\vert \mu _{2}\right\vert }}\right\vert \right) \left( \psi
(\sigma )-\psi (\sigma -2\left\vert \mu _{1}\right\vert )\right) \right\vert
\left\vert \mu _{1}\right\vert ^{n}\left\vert \mu _{2}\right\vert ^{n}d\mu
_{1}d\mu _{2} \\
&\leq &C2^{-j(\alpha -2)}2^{-j}(2k+n)^{-n-m}(2l+n)^{-n-m}.
\end{eqnarray*}%
It follows that
\begin{eqnarray*}
&&\int_{\mathbb{R}^{m}}\int_{\mathbb{R}^{m}}\left\vert R_{k,l}\left( \mu
_{1},\mu _{2},\left( iu_{1}^{(p)}-\frac{\mu _{1}^{(p)}}{\left\vert \mu
_{1}\right\vert }\frac{1}{4}\left\vert z_{1}\right\vert ^{2}\right) \left(
iu_{2}^{(q)}-\frac{\mu _{2}^{(q)}}{\left\vert \mu _{2}\right\vert }\frac{1}{4%
}\left\vert z_{2}\right\vert ^{2}\right) F\right) \right\vert \left\vert \mu
_{1}\right\vert ^{n}\left\vert \mu _{2}\right\vert ^{n}d\mu _{1}d\mu _{2} \\
&\leq &C2^{-j(\alpha -4)}2^{-j}(2k+n)^{-n-m}(2l+n)^{-n-m}.
\end{eqnarray*}%
An iteration of the process shows that for any $N\in
\mathbb{N}
^{+}$,
\begin{eqnarray*}
&&\int_{\mathbb{R}^{m}}\int_{\mathbb{R}^{m}}\left\vert R_{k,l}\left( \mu
_{1},\mu _{2},\left( iu_{1}^{(p)}-\frac{\mu _{1}^{(p)}}{\left\vert \mu
_{1}\right\vert }\frac{1}{4}\left\vert z_{1}\right\vert ^{2}\right) \left(
iu_{2}^{(q)}-\frac{\mu _{2}^{(q)}}{\left\vert \mu _{2}\right\vert }\frac{1}{4%
}\left\vert z_{2}\right\vert ^{2}\right) ^{N}F\right) \right\vert \left\vert
\mu _{1}\right\vert ^{n}\left\vert \mu _{2}\right\vert ^{n}d\mu _{1}d\mu _{2}
\\
&\leq &C_{N}2^{-j(\alpha -4N)}2^{-j}(2k+n)^{-n-m}(2l+n)^{-n-m}.
\end{eqnarray*}%
Thus, for any $p,q=1,2,\cdots ,m$, we have
\begin{eqnarray*}
&&\sum_{k=0}^{\infty }\sum_{l=0}^{\infty }\int_{\mathbb{R}^{m}}\int_{\mathbb{%
R}^{m}}\left\vert R_{k,l}\left( \mu _{1},\mu _{2},\left( iu_{1}^{(p)}-\frac{%
\mu _{1}^{(p)}}{\left\vert \mu _{1}\right\vert }\frac{1}{4}\left\vert
z_{1}\right\vert ^{2}\right) \left( iu_{2}^{(q)}-\frac{\mu _{2}^{(q)}}{%
\left\vert \mu _{2}\right\vert }\frac{1}{4}\left\vert z_{2}\right\vert
^{2}\right) ^{N}F\right) \right\vert \\
&&\times \frac{\left( k+n-1\right) !}{k!}\frac{\left( l+n-1\right) !}{l!}%
\left\vert \mu _{1}\right\vert ^{n}\left\vert \mu _{2}\right\vert ^{n}d\mu
_{1}d\mu _{2}\leq C_{N}2^{-j(\alpha -4N+1)},
\end{eqnarray*}%
which yields that
\begin{equation}
\left\vert K_{j}^{\alpha }\left( \left( x_{1},u_{1}\right) ,\left(
x_{2},u_{2}\right) \right) \right\vert \leq C_{N}2^{-j(\alpha -4N+1)}\left(
1+\left\vert \left( A_{\frac{\mu _{1}}{\left\vert \mu _{1}\right\vert }%
}x_{1},u_{1}\right) \right\vert \right) ^{-2N}\left( 1+\left\vert \left( A_{%
\frac{\mu _{2}}{\left\vert \mu _{2}\right\vert }}x_{2},u_{2}\right)
\right\vert \right) ^{-2N}.  \label{Kj-kernel}
\end{equation}%
Consequently, for any fixed $(x_{1},u_{1}),(x_{2},u_{2})\in \mathbb{G}$ and $%
\mu _{1}$, $\mu _{2}\in
\mathbb{R}
^{m}\backslash \{0\}$, whenever $\alpha >4N-1$, we have
\begin{eqnarray*}
\left\vert S^{\alpha }\left( \left( x_{1},u_{1}\right) ,\left(
x_{2},u_{2}\right) \right) \right\vert &\leq &\sum_{j=0}^{\infty
}C_{N}2^{-j(\alpha -4N+1)}\left( 1+\left\vert \left( A_{\frac{\mu _{1}}{%
\left\vert \mu _{1}\right\vert }}x_{1},u_{1}\right) \right\vert \right)
^{-2N}\left( 1+\left\vert \left( A_{\frac{\mu _{2}}{\left\vert \mu
_{2}\right\vert }}x_{2},u_{2}\right) \right\vert \right) ^{-2N} \\
&\leq &C_{N}\left( 1+\left\vert \left( A_{\frac{\mu _{1}}{\left\vert \mu
_{1}\right\vert }}x_{1},u_{1}\right) \right\vert \right) ^{-2N}\left(
1+\left\vert \left( A_{\frac{\mu _{2}}{\left\vert \mu _{2}\right\vert }%
}x_{2},u_{2}\right) \right\vert \right) ^{-2N}.
\end{eqnarray*}%
The proof of Theorem \ref{bilinear-kernel} is completed.
\end{proof}

\begin{corollary}
\label{cor2}Let $1\leq p_{1},p_{2}\leq \infty $ and $1/p=1/p_{1}+1/p_{2}$.
If $\alpha >2Q+3$, then $S^{\alpha }$ is bounded from $L^{p_{1}}(\mathbb{G}%
)\times L^{p_{2}}(\mathbb{G})$ into $L^{p}(\mathbb{G})$.
\end{corollary}

\begin{proof}
We take positive integer $N=\frac{Q}{2}+1$. If $\alpha >2Q+3$, we have $%
\alpha >4N-1$. So, Theorem \ref{bilinear-kernel} is available. Then,
applying the H\"{o}lder's inequality, the Young's inequality and changing
variables, we conclude that
\begin{eqnarray*}
&&\left\Vert S_{R}^{\alpha }(f,g)\right\Vert _{p} \\
&\leq &C\left\Vert f\right\Vert _{p_{1}}\left\Vert g\right\Vert
_{p_{2}}\int_{\mathbb{G}}\left( 1+\left\vert \left( A_{\frac{\mu _{1}}{%
\left\vert \mu _{1}\right\vert }}x_{1},u_{1}\right) \right\vert \right)
^{-2N}dxdu_{1}\int_{\mathbb{G}}\left( 1+\left\vert \left( A_{\frac{\mu _{2}}{%
\left\vert \mu _{2}\right\vert }}x_{2},u_{2}\right) \right\vert \right)
^{-2N}dx_{2}du_{2} \\
&\leq &C\left\Vert f\right\Vert _{p_{1}}\left\Vert g\right\Vert
_{p_{2}}\int_{\mathbb{G}}\left( 1+\left\vert \left( z_{1},u_{1}\right)
\right\vert \right) ^{-2N}\left\vert \det A_{\frac{\mu _{1}}{\left\vert \mu
_{1}\right\vert }}\right\vert ^{-1}dz_{1}du_{1}\int_{\mathbb{G}}\left(
1+\left\vert \left( z_{2},u_{2}\right) \right\vert \right) ^{-2N}\left\vert
\det A_{\frac{\mu _{1}}{\left\vert \mu _{1}\right\vert }}\right\vert
^{-1}dz_{2}du_{2} \\
&\leq &C\left\Vert f\right\Vert _{p_{1}}\left\Vert g\right\Vert
_{p_{2}}\int_{\mathbb{G}}\left( 1+\left\vert \omega _{1}\right\vert \right)
^{-2N}d\omega _{1}\int_{\mathbb{G}}\left( 1+\left\vert \omega
_{2}\right\vert \right) ^{-2N}d\omega _{2} \\
&\leq &C\left\Vert f\right\Vert _{p_{1}}\left\Vert g\right\Vert
_{p_{2}}\left( \int_{1}^{\infty }t^{-2N+Q-1}dt\right) ^{2} \\
&\leq &C\left\Vert f\right\Vert _{p_{1}}\left\Vert g\right\Vert _{p_{2}}.
\end{eqnarray*}%
The proof is completed.
\end{proof}

Note that the index in Corollary \ref{cor1} and Corollary \ref{cor2} are
very high. Thus, in the rest part of this paper, we shall use different
methods to give lower indices.

\section{Boundedness of the Riesz means $S^{\protect\delta}$}

Let $\mathbb{G}$ be a M-type group with Lie algebra $\mathfrak{g}\,=%
\mathfrak{g}_{1}\oplus \mathfrak{g}_{2}$ where $\mathfrak{g}_{2}$ is the
center of $\mathfrak{g}$. As above, we assume that $\dim \mathfrak{g}_{1}=2n$
and $\dim \mathfrak{g}_{2}=m$. The mixed norm on $\mathbb{G}$ is defined by
\begin{equation*}
\left\Vert f\right\Vert _{L^{r}(\mathfrak{g}_{2})L^{p}(\mathfrak{g}%
_{1})}=\left( \int_{%
\mathbb{R}
^{2n}}\left( \int_{%
\mathbb{R}
^{m}}\left\vert f(x,u)\right\vert ^{r}du\right) ^{\frac{p}{r}}dx\right) ^{%
\frac{1}{p}}\text{,\ \ }1\leq p,r\leq \infty \text{. }
\end{equation*}%
Casarino and Ciatti \cite{Cas} proved the following Stein-Tomas restriction
theorem in terms of the mixed norms:

\begin{lemma}
\cite{Cas} \label{mainLemma}Let $\mathbb{G}$ be a M-type group with Lie
algebra $\mathfrak{g}\,=\mathfrak{g}_{1}\oplus \mathfrak{g}_{2}$ where $%
\mathfrak{g}_{2}$ is the center of $\mathfrak{g}$. Let $\dim \mathfrak{g}%
_{1}=2n$ and $\dim \mathfrak{g}_{2}=m$. Then, for all $1\leq p\leq 2\leq
q\leq \infty $, $1\leq r\leq \frac{2m+2}{m+3}$and all $f\in \mathscr{S}(%
\mathbb{G)}$, we have that
\begin{equation*}
\left\Vert P_{\lambda }f\right\Vert _{L^{r^{\prime }}(\mathfrak{g}_{2})L^{q}(%
\mathfrak{g}_{1})}\leq C\lambda ^{m(\frac{2}{r}-1)+n(\frac{1}{p}-\frac{1}{q}%
)-1}\left\Vert f\right\Vert _{L^{r}(\mathfrak{g}_{2})L^{p}(\mathfrak{g}_{1})}%
\text{.}
\end{equation*}
\end{lemma}

Particulerly, when $p=r$ and $q=p\prime $, we have the following result:

\begin{theorem}
\label{mainCo}Let $\mathbb{G}$ be a M-type group with Lie algebra $\mathfrak{%
g}\,=\mathfrak{g}_{1}\oplus \mathfrak{g}_{2}$ where $\mathfrak{g}_{2}$ is
the center of $\mathfrak{g}$. Let $\dim \mathfrak{g}_{1}=2n$ and $\dim
\mathfrak{g}_{2}=m$. Then, for all $1\leq p\leq \frac{2m+2}{m+3}$ and all $%
f\in \mathscr{S}(\mathbb{G)}$, we have that
\begin{equation*}
\left\Vert P_{\lambda }f\right\Vert _{L^{p^{\prime }}(\mathbb{G})}\leq
C\lambda ^{Q(\frac{1}{p}-\frac{1}{2})-1}\left\Vert f\right\Vert _{L^{p}(%
\mathbb{G})}.
\end{equation*}
\end{theorem}

Applying Theorem \ref{mainCo}, we can obtain the $L^{p}$-boundedness of the
Riesz means $S^{\delta}$.

\begin{theorem}
\label{T1}Let $\mathbb{G}$ be a M-type group with Lie algebra $\mathfrak{g}%
\,=\mathfrak{g}_{1}\oplus \mathfrak{g}_{2}$ where $\mathfrak{g}_{2}$ is the
center of $\mathfrak{g}$. Let $\dim \mathfrak{g}_{1}=2n$ and $\dim \mathfrak{%
g}_{2}=m$. Suppose that $1\leq p\leq \frac{2m+2}{m+3}$. If $\delta >Q\left(
\frac{1}{p}-\frac{1}{2}\right) -\frac{1}{2}$, then $S^{\delta }$ is bounded
on $L^{p}(\mathbb{G})$.
\end{theorem}

\begin{proof}
We take a partition of unity $\sum_{-\infty }^{\infty }\varphi (2^{j}s)=1$
where $\varphi \in C_{0}^{\infty }(\frac{1}{2},2)$ is a nonnegative
function, and write $\varphi _{j}^{\delta }(s)=(1-s)_{+}^{\delta }\varphi
(2^{j}(1-s))$ for each $j\geq 0$. Define
\begin{equation*}
T_{j}^{\delta }f=\int_{0}^{\infty }\int_{0}^{\infty }\varphi _{j}^{\alpha
}\left( \lambda \right) P_{\lambda }fd\lambda \text{.}
\end{equation*}%
It is obvious that%
\begin{equation*}
S^{\delta }=\sum_{j=0}^{\infty }T_{j}^{\delta }.
\end{equation*}%
Theorem \ref{T1} will be proved once we show that when $\delta >Q\left(
\frac{1}{p}-\frac{1}{2}\right) -\frac{1}{2}$, there exists an $\varepsilon
>0 $ such that for each $j\geq 0$,
\begin{equation}
\left\Vert T_{j}^{\delta }\right\Vert _{L^{p}\rightarrow L^{p}}\leq
2^{-\varepsilon j}.  \label{p1}
\end{equation}%
In order to prove (\ref{p1}), we define $B_{j}=\{\omega :\left\vert \omega
\right\vert \leq 2^{j(1+\gamma )}\}\subseteq \mathbb{G}$ and split the
kernel $s_{j}^{\delta }(\omega )$ of $T_{j}^{\delta }$ into two parts
\begin{equation*}
s_{j}^{\delta }(\omega )=s_{j}^{1}(\omega )+s_{j}^{2}(\omega ),
\end{equation*}%
where $s_{j}^{1}(\omega )=s_{j}^{\delta }\chi _{B_{j}}(\omega )$, $%
s_{j}^{2}(\omega )=s_{j}^{\delta }\chi _{B_{j}^{c}}(\omega )$. Here $\chi
_{A}$ stands for the characteristic function of set $A$ and $\gamma >0$ is
to be fixed. Clearly, (\ref{p1}) is the consequence of the estimates%
\begin{equation*}
\left\Vert f\ast s_{j}^{l}\right\Vert _{L^{p}}\leq C2^{-\varepsilon
j}\left\Vert f\right\Vert _{p},\text{ \ }l=1,2\text{.}
\end{equation*}

We first consider the convolution with $s_{j}^{2}(\omega )$. Set $R_{\lambda
}^{l}(\omega )$ to be the kernel of the Riesz means $\int_{0}^{\lambda }(1-%
\frac{t}{\lambda })^{l}P_{t}dt$. Then,
\begin{equation*}
\lambda \rightarrow R_{\lambda }^{0}(\omega )
\end{equation*}%
is a function of bounded variation, and the kernel of each $T_{j}^{\delta }$
can be written as
\begin{equation*}
s_{j}^{\delta }(\omega )=\int_{0}^{\infty }\varphi _{j}^{\delta }(\lambda )%
\frac{\partial }{\partial \lambda }R_{\lambda }^{0}(\omega )d\lambda.
\end{equation*}%
Integrating by parts and using the identity%
\begin{equation}
\frac{\partial }{\partial \lambda }(\lambda ^{N}R_{\lambda }^{N}(\omega
))=N\lambda ^{N-1}R_{\lambda }^{N-1}(\omega ),  \label{identity-0}
\end{equation}%
where $N$ is a positive integer, we get the relation
\begin{equation*}
s_{j}^{\delta }(\omega )=c_{N}\int \left( \partial _{\lambda }^{2N+2}\varphi
_{j}^{\delta }(\lambda )\right) \lambda ^{2N+1}R_{\lambda }^{2N+1}(\omega
)d\lambda .
\end{equation*}%
Theorem \ref{kernel} tells that%
\begin{equation}
\sup_{\lambda \in \lbrack 0,1]}\left\vert \lambda ^{2N+1}R_{\lambda
}^{2N+1}(\omega )\right\vert \leq \left( 1+\left\vert \left( A_{\frac{\mu }{%
\left\vert \mu \right\vert }}x,u\right) \right\vert \right) ^{-2N}
\label{es-0}
\end{equation}%
for any $\omega =(x,u)\in \mathbb{G}$. This estimate, together with the
bound
\begin{equation*}
\left\vert \partial _{t}^{2N+2}\varphi _{j}^{\delta }(\lambda )\right\vert
\leq 2^{j(2N+2)}
\end{equation*}%
imply that
\begin{equation*}
\left\vert s_{j}^{\delta }(\omega )\right\vert \leq C2^{j(2N+2)}\left(
1+\left\vert \left( A_{\frac{\mu }{\left\vert \mu \right\vert }}x,u\right)
\right\vert \right) ^{-2N}.
\end{equation*}%
By changing variables, it follows that
\begin{eqnarray*}
\int_{\mathbb{G}}\left\vert s_{j}^{2}(\omega )\right\vert d\omega
&=&\int_{\left\vert \omega \right\vert >2^{j(1+\gamma )}}\left\vert
s_{j}^{\delta }(\omega )\right\vert d\omega \\
&\leq &C2^{j(2N+2)}\int_{\left\vert (x,u)\right\vert \geq 2^{j(1+\gamma
)}}\left( 1+\left\vert \left( A_{\frac{\mu }{\left\vert \mu \right\vert }%
}x,u\right) \right\vert \right) ^{-2N}dxdu \\
&\leq &C2^{j(2N+2)}\left\vert \det A_{\frac{\mu }{\left\vert \mu \right\vert
}}\right\vert ^{-1}\int_{\left\vert \left( A_{\frac{\mu }{\left\vert \mu
\right\vert }}^{-1}z,u\right) \right\vert \geq 2^{j(1+\gamma )}}\left(
1+\left\vert \left( z,u\right) \right\vert \right) ^{-2N}dzdu.
\end{eqnarray*}%
Because there exists positive constant $K,H$ such that
\begin{equation*}
\sqrt{K}\leq \left\vert \det A_{\frac{\mu }{\left\vert \mu \right\vert }%
}\right\vert ^{-1}\leq \frac{1}{\sqrt{K}}\quad \text{and}\quad \frac{1}{H}%
\leq \left\vert A_{\frac{\mu }{\left\vert \mu \right\vert }
}^{-1}z\right\vert \leq H\left\vert z\right\vert
\end{equation*}%
for any $\mu\in \mathbb{R}^{m}\backslash \{0\}$, so we have that
\begin{eqnarray*}
\int_{\mathbb{G}}\left\vert s_{j}^{2}(\omega )\right\vert d\omega &\leq &C
2^{j(2N+2)}\int_{\left\vert (z,u)\right\vert \geq 2^{j(1+\gamma )}}\left(
1+\left\vert \left( z,u\right) \right\vert \right) ^{-2N}dzdu \\
&\leq &C2^{j(2N+2)}\int_{\left\vert \omega \right\vert \geq 2^{j(1+\gamma
)}}\left( 1+\left\vert \omega \right\vert \right) ^{-2N}d\omega \\
&\leq &C2^{j(2N+2)}2^{j(1+\gamma )(-2N+Q)}.
\end{eqnarray*}%
Choosing $N$ large enough such that
\begin{equation*}
2N\gamma >Q(1+\gamma )+2\text{,}
\end{equation*}%
it follows that
\begin{equation*}
\int_{\mathbb{G}}\left\vert s_{j}^{2}(\omega )\right\vert d\omega \leq
C2^{-\varepsilon j}
\end{equation*}%
for some $\varepsilon >0$. By Young's inequality, we conclude for any $1\leq
p\leq \infty $,
\begin{equation}
\left\Vert f\ast s_{j}^{2}\right\Vert _{p}\leq C2^{-\varepsilon j}\left\Vert
f\right\Vert _{p}.  \label{p2}
\end{equation}

Next, we consider the convolution with $s_{j}^{1}$. For any $\xi \in \mathbb{%
G}$ and $R>0$, we set $B_{j}(\xi ,R)=\{\omega:\left\vert \xi
^{-1}\omega \right\vert \leq R2^{j(1+\gamma )}\}\subseteq \mathbb{G}$ and split the function $f$
into three parts: $f=f_{1}+f_{2}+f_{3}$ where $f_{1}=f\chi _{B_{j}(\xi ,%
\frac{3}{4})}$, $f_{2}=f\chi _{B_{j}(\xi ,\frac{5}{4})\backslash B_{j}\left(
\xi ,\frac{3}{4}\right) }$ and $f_{3}=f\chi _{\mathbb{G}\backslash B_{j}(\xi
,\frac{5}{4})}.$ Assume that $|\xi ^{-1}\omega |\leq \frac{1}{4}%
2^{j(1+\gamma )}$. Since that $f_{3}$ is supported on $\mathbb{G}\backslash
B_{j}(\xi ,\frac{5}{4})$, then $f_{3}\neq 0$ leads to
\begin{equation*}
\left\vert \xi ^{-1}\cdot \omega \omega ^{\prime -1}\right\vert \geq \frac{5%
}{4}2^{j(1+\gamma )}.
\end{equation*}%
It follows that
\begin{equation*}
\left\vert \omega ^{\prime }\right\vert \geq 2^{j(1+\gamma )}.
\end{equation*}%
Note that $s_{j}^{1}$ is supported on $B_{j}$. Hence, $f_{3}\ast s_{j}^{1}=0$%
. Since $f_{2}$ is supported on $B_{j}(\xi ,\frac{5}{4})\backslash B_{j}(\xi
,\frac{3}{4})$, then $f_{2}\neq 0$ leads to
\begin{equation*}
\left\vert \omega ^{\prime }\right\vert >\frac{1}{2}2^{j(1+\gamma )}.
\end{equation*}%
Repeating the proof of (\ref{p2}), we can get that
\begin{equation}
\left\Vert f_{2}\ast s_{j}^{1}\right\Vert _{L^{p}(B_{j}(\xi ,\frac{1}{4}%
))}\leq C2^{-\varepsilon j}\left\Vert f_{2}\right\Vert _{L^{p}}\leq
C2^{-\varepsilon j}\left\Vert f\right\Vert _{L^{p}(B_{j}(\xi \,,\frac{5}{4}%
))}.  \label{F6}
\end{equation}%
Taking the $L^{p}$ norm with respect to $\xi $ on the both side of (\ref{F6}%
), it follows that
\begin{equation*}
\left( \int_{\mathbb{G}}\int_{B_{j}(\xi \,,\frac{5}{4})}|f_{2}\ast
s_{j}^{1}(\omega )|^{p}d\omega d\xi \right) ^{\frac{1}{p}}\leq
C2^{-\varepsilon j}\left( \int_{\mathbb{G}}\int_{B_{j}(\xi ,\frac{5}{4}%
)}|f(\omega )|^{p}d\omega d\xi \right) ^{\frac{1}{p}}.
\end{equation*}%
Changing the variable and exchanging the order of integration, the left side
\begin{eqnarray*}
\left( \int_{\mathbb{G}}\int_{B_{j}(\xi \,,\frac{5}{4})}|f_{2}\ast
s_{j}^{1}(\omega )|^{p}d\omega d\xi \right) ^{\frac{1}{p}} &=&\left( \int_{%
\mathbb{G}}\int_{\left\vert \omega \right\vert \leq \frac{1}{4}2^{j(1+\gamma
)}}\left\vert f_{2}\ast s_{j}^{1}(\xi \omega )\right\vert ^{p}d\omega d\xi
\right) ^{\frac{1}{p}} \\
&=&\left( \int_{\left\vert \omega \right\vert \leq \frac{1}{4}2^{j(1+\gamma
)}}\int_{\mathbb{G}}\left\vert f_{2}\ast s_{j}^{1}(\xi \omega )\right\vert
^{p}d\xi d\omega \right) ^{\frac{1}{p}} \\
&=&\left( \frac{1}{4}2^{j(1+\gamma )}\right) ^{\frac{Q}{p}}\left\Vert
f_{2}\ast s_{j}^{1}\right\Vert _{p}.
\end{eqnarray*}%
In the same way, the right side
\begin{equation*}
C2^{-\varepsilon j}\left( \int_{\mathbb{G}}\int_{B_{j}(\xi ,\frac{5}{4}%
)}|f(\omega )|^{p}d\omega d\xi \right) ^{\frac{1}{p}}=C2^{-\varepsilon
j}\left( \frac{5}{4}2^{j(1+\gamma )}\right) ^{\frac{Q}{p}}\left\Vert
f\right\Vert _{p}.
\end{equation*}%
Hence, we have that
\begin{equation}
\left\Vert f_{2}\ast s_{j}^{1}\right\Vert _{p}\leq C2^{-\varepsilon
j}\left\Vert f\right\Vert _{p}.  \label{F2K1}
\end{equation}%
Now it remains to estimate $f_{1}\ast s_{j}^{1}$. Since that $f_{1}\neq 0$
implies
\begin{equation*}
\left\vert \xi ^{-1}\cdot \omega \omega ^{\prime -1}\right\vert \leq \frac{3%
}{4}2^{j(1+\mathbb{\gamma })}\text{,}
\end{equation*}%
it follows that
\begin{equation*}
\left\vert \omega ^{\prime }\right\vert \leq 2^{j(1+\mathbb{\gamma })}\text{.%
}
\end{equation*}%
So, we have
\begin{equation}
f_{1}\ast s_{j}^{1}(\omega )=f_{1}\ast s_{j}^{\delta }(\omega )\quad \text{%
for any}\quad \omega \in B_{j}(\xi ,\frac{1}{4}).  \label{K1J1}
\end{equation}%
Note that
\begin{equation}
f\ast s_{j}^{\delta }=T_{j}^{\delta }f=\int_{1-2^{-j+1}}^{1-2^{-j-1}}\varphi
_{j}^{\delta }(\lambda )P_{\lambda }fd\lambda .  \label{k3}
\end{equation}%
Then,
\begin{eqnarray*}
\left\Vert f\ast s_{j}^{\delta }\right\Vert _{2}^{2} &=&\left\langle
T_{j}^{\delta }f,T_{j}^{\delta }f\right\rangle =\left\langle \left(
T_{j}^{\delta }\right) ^{\ast }T_{j}^{\delta }f,f\right\rangle \\
&=&\left\langle \left( T_{j}^{\delta }\right) ^{2}f,f\right\rangle \leq
\left\Vert \left( T_{j}^{\delta }\right) ^{2}f\right\Vert _{p^{\prime
}}\left\Vert f\right\Vert _{p}\text{.}
\end{eqnarray*}%
Since
\begin{equation*}
\left( T_{j}^{\delta }\right) ^{2}f=\int_{1-2^{-j+1}}^{1-2^{-j-1}}\left(
\varphi _{j}^{\delta }(\lambda )\right) ^{2}P_{\lambda }f\,d\lambda \,,
\end{equation*}%
using Theorem \ref{mainCo}, we get that
\begin{eqnarray*}
\Vert \left( T_{j}^{\delta }\right) ^{2}f\Vert _{p^{\prime }} &\leq
&\int_{1-2^{-j+1}}^{1-2^{-j-1}}\left( \varphi _{j}^{\delta }(\lambda
)\right) ^{2}\Vert P_{\lambda }f\Vert _{p^{\prime }}d\lambda \\
&\leq &C\int_{1-2^{-j+1}}^{1-2^{-j-1}}\left( \varphi _{j}^{\delta }(\lambda
)\right) ^{2}\lambda ^{Q(\frac{1}{p}-\frac{1}{2})-1}\Vert f\Vert _{p}d\lambda
\\
&\leq &C2^{-2j\delta -j}\Vert f\Vert _{p}.
\end{eqnarray*}%
So,
\begin{equation*}
\Vert f\ast s_{j}^{\delta }\Vert _{2}^{2}\leq C2^{-2j\delta -j}\Vert f\Vert
_{p}^{2}\text{.}
\end{equation*}%
This estimate, together with (\ref{K1J1}) and H\"{o}lder's inequality yield
that
\begin{eqnarray*}
\left\Vert f_{1}\ast s_{j}^{1}\right\Vert _{L^{p}(B_{j}(\xi ,\frac{1}{4}))}
&=&\left\Vert f_{1}\ast s_{j}^{\delta }\right\Vert _{L^{p}(B_{j}(\xi ,\frac{1%
}{4}))} \\
&\leq &2^{jQ(1+\gamma )(\frac{1}{p}-\frac{1}{2})}\Vert f_{1}\ast
s_{j}^{\delta }\Vert _{2} \\
&\leq &C2^{-j\delta }2^{-\frac{j}{2}}2^{jQ(1+\gamma )(\frac{1}{p}-\frac{1}{2}%
)}\left\Vert f_{1}\right\Vert _{p} \\
&\leq &C2^{-j\delta }2^{-\frac{j}{2}}2^{jQ(1+\gamma )(\frac{1}{p}-\frac{1}{2}%
)}\Vert f\Vert _{L^{p}(B_{j}(\xi ,\frac{3}{4}))}.
\end{eqnarray*}%
Taking the $L^{p}$ norm with respect to $\xi $, we have
\begin{equation*}
\left\Vert f_{1}\ast s_{j}^{1}\right\Vert _{L^{p}}\leq C2^{-j\delta }2^{-%
\frac{j}{2}}2^{jQ(1+\gamma )(\frac{1}{p}-\frac{1}{2})}\Vert f\Vert _{L^{p}}.
\end{equation*}%
Thus, when $\delta >Q\left( \frac{1}{p}-\frac{1}{2}\right) -\frac{1}{2}$, we
can choose $\gamma >0$ such that
\begin{equation*}
\delta >Q(1+\gamma )\left( \frac{1}{p}-\frac{1}{2}\right) -\frac{1}{2},
\end{equation*}%
which implies that there exists an $\varepsilon >0$ such that
\begin{equation*}
\left\Vert f_{1}\ast s_{j}^{1}\right\Vert _{L^{p}}\leq C2^{-\varepsilon
j}\Vert f\Vert _{L^{p}}.
\end{equation*}%
The proof of Theorem \ref{T1} is completed.
\end{proof}

\section{Boundedness of $S^{\protect\alpha }$ for $1\leq p_{1},p_{2}\leq 2$}

From this section, we begin to investigate the $L^{p_{1}}\times
L^{p_{2}}\rightarrow L^{p}$ boundedness of the bilinear Riesz means $%
S^{\alpha }$. We first consider the case of $1\leq p_{1},p_{2}\leq 2$.

\begin{lemma}
\label{restriction} Suppose $m\in L^{\infty }(\mathbb{R})$. Define operator $%
T_{m}f=\int_{a}^{b}m(\lambda )P_{\lambda }f\,d\lambda $ with $0\leq a<b$.
Then, for any $1\leq p\leq 2$, we have
\begin{equation*}
\left\Vert T_{m}f\right\Vert _{2}\leq C\left\Vert m\right\Vert _{\infty
}\left( (b-a)b^{n+m-1}\right) ^{\left( \frac{1}{p}-\frac{1}{2}\right)
}\left\Vert f\right\Vert _{p}.
\end{equation*}
\end{lemma}

\begin{proof}
Since that
\begin{equation*}
P_{\lambda }f=\sum_{k=0}^{\infty }\frac{\lambda ^{n+m-1}}{\left( 2\pi
(2k+n)\right) ^{n+m}}\int_{\mathbb{S}^{m-1}}f\ast \widetilde{e}_{k}^{\lambda
\eta }(x,u)d\sigma (\eta ),
\end{equation*}%
the operator $T_{m}$ can be written as
\begin{equation*}
T_{m}f=\int_{a}^{b}m(\lambda )P_{\lambda }f\,d\lambda =f\ast G_{m},
\end{equation*}%
where the kernel $G_{m}$ is given by
\begin{eqnarray*}
&&G_{m}(x,u) \\
&=&\int_{a}^{b}m(\lambda )\sum_{k=0}^{\infty }\frac{\lambda ^{n+m-1}}{\left(
2\pi (2k+n)\right) ^{n+m}}\int_{\mathbb{S}^{m-1}}\widetilde{e}_{k}^{\lambda
\eta }(x,u)d\sigma (\eta )d\lambda \\
&=&\frac{1}{(2\pi )^{n+m}}\sum_{k=0}^{\infty }\int_{\frac{a}{2k+n}}^{\frac{b%
}{2k+n}}m((2k+n)\lambda )\int_{\mathbb{S}^{m-1}}e_{k}^{\lambda \eta
}(x,u)d\sigma (\eta )\lambda ^{n+m-1}d\lambda \\
&=&\frac{1}{(2\pi )^{n+m}}\sum_{k=0}^{\infty }\int_{\frac{a}{2k+n}}^{\frac{b%
}{2k+n}}m((2k+n)\lambda )\int_{\mathbb{S}^{m-1}}e^{-i\lambda \eta (u)}\left(
\varphi _{k}^{\lambda }\circ A_{\eta }\right) (x)\left\vert \det A_{\eta
}\right\vert d\sigma (\eta )\lambda ^{n+m-1}d\lambda \\
&=&\frac{1}{(2\pi )^{n+m}}\sum_{k=0}^{\infty }\int_{\mathbb{R}^{m}}\chi
_{\lbrack \frac{a}{2k+n},\frac{b}{2k+n}]}(\left\vert \mu \right\vert
)m((2k+n)\left\vert \mu \right\vert )e^{-i\mu (u)}\left( \varphi
_{k}^{\left\vert \mu \right\vert }\circ A_{\frac{\mu }{\left\vert \mu
\right\vert }}\right) (x)\left\vert \det A_{\frac{\mu }{\left\vert \mu
\right\vert }}\right\vert \left\vert \mu \right\vert ^{n}d\mu .
\end{eqnarray*}%
Applying the Plancherel theorem in the variable $u$ and the orthogonality of
$\varphi _{k}^{\left\vert \mu \right\vert }$, we obtain that
\begin{eqnarray*}
&&\left\Vert G_{m}\right\Vert _{2}^{2} \\
&=&C\int_{\mathbb{R}^{2n}}\int_{\mathbb{R}^{m}}\left\vert \int_{\mathbb{R}%
^{m}}e^{-i\mu (u)}\sum_{k=0}^{\infty }\chi _{\lbrack \frac{a}{2k+n},\frac{b}{%
2k+n}]}(\left\vert \mu \right\vert )m((2k+n)\left\vert \mu \right\vert
)\left( \varphi _{k}^{\left\vert \mu \right\vert }\circ A_{\frac{\mu }{%
\left\vert \mu \right\vert }}\right) (x)\left\vert \det A_{\frac{\mu }{%
\left\vert \mu \right\vert }}\right\vert \left\vert \mu \right\vert ^{n}d\mu
\right\vert ^{2}\,dudx \\
&=&C\int_{\mathbb{R}^{2n}}\int_{\mathbb{R}^{m}}\left\vert \sum_{k=0}^{\infty
}\chi _{\lbrack \frac{a}{2k+n},\frac{b}{2k+n}]}(\left\vert \mu \right\vert
)m((2k+n)\left\vert \mu \right\vert )\left( \varphi _{k}^{\left\vert \mu
\right\vert }\circ A_{\frac{\mu }{\left\vert \mu \right\vert }}\right)
(x)\right\vert ^{2}\left\vert \det A_{\frac{\mu }{\left\vert \mu \right\vert
}}\right\vert ^{2}\,\left\vert \mu \right\vert ^{2n}d\mu dx \\
&\leq &C\int_{\mathbb{R}^{m}}\sum_{k=0}^{\infty }\left\vert \chi _{\lbrack
\frac{a}{2k+n},\frac{b}{2k+n}]}(\left\vert \mu \right\vert
)m((2k+n)\left\vert \mu \right\vert )\right\vert ^{2}\left\Vert \varphi
_{k}^{\left\vert \mu \right\vert }\circ A_{\frac{\mu }{\left\vert \mu
\right\vert }}\right\Vert _{2}^{2}\left\vert \det A_{\frac{\mu }{\left\vert
\mu \right\vert }}\right\vert ^{2}|\mu |^{2n}d\mu.
\end{eqnarray*}%
Changing variables in the integral and using the estimate
\begin{equation}
\left\Vert \varphi _{k}^{\left\vert \mu \right\vert }\right\Vert _{2}\leq
C\left\vert \mu \right\vert ^{-\frac{n}{2}}k^{\frac{n-1}{2}},  \label{varphi}
\end{equation}%
we get that
\begin{equation*}
\left\Vert \varphi _{k}^{\left\vert \mu \right\vert }\circ A_{\frac{\mu }{%
\left\vert \mu \right\vert }}\right\Vert _{2}^{2}=\int_{%
\mathbb{R}
^{2n}}\left\vert \varphi _{k}^{\left\vert \mu \right\vert }\left( A_{\frac{%
\mu }{\left\vert \mu \right\vert }}x\right) \right\vert ^{2}dx=\left\vert
\det A_{\frac{\mu }{\left\vert \mu \right\vert }}\right\vert ^{-1}\left\Vert
\varphi _{k}^{\left\vert \mu \right\vert }\right\Vert _{2}^{2}\leq
C\left\vert \det A_{\frac{\mu }{\left\vert \mu \right\vert }}\right\vert
^{-1}\left\vert \mu \right\vert ^{-n}k^{n-1}.
\end{equation*}%
Thus,
\begin{eqnarray*}
\left\Vert G_{m}\right\Vert _{2}^{2} &\leq &C\int_{\mathbb{R}%
^{m}}\sum_{k=0}^{\infty }\left\vert \chi _{\lbrack \frac{a}{2k+n},\frac{b}{%
2k+n}]}(\left\vert \mu \right\vert )m((2k+n)\left\vert \mu \right\vert
)\right\vert ^{2}\left\Vert \varphi _{k}^{\left\vert \mu \right\vert }\circ
A_{\frac{\mu }{\left\vert \mu \right\vert }}\right\Vert _{2}^{2}\left\vert
\det A_{\frac{\mu }{\left\vert \mu \right\vert }}\right\vert ^{2}|\mu
|^{2n}d\mu \\
&\leq &C\int_{\mathbb{R}^{m}}\sum_{k=0}^{\infty }k^{n-1}\left\vert \chi
_{\lbrack \frac{a}{2k+n},\frac{b}{2k+n}]}(\left\vert \mu \right\vert
)m((2k+n)\left\vert \mu \right\vert )\right\vert ^{2}\left\vert \det A_{%
\frac{\mu }{\left\vert \mu \right\vert }}\right\vert \left\vert \mu
\right\vert ^{n}d\mu \\
&=&C\left( \sum_{k=0}^{\infty }(2k+n)^{-n-m}k^{n-1}\right) \int_{a\leq
\left\vert \mu \right\vert \leq b}\left\vert m(\left\vert \mu \right\vert
)\right\vert ^{2}\left\vert \det A_{\frac{\mu }{\left\vert \mu \right\vert }%
}\right\vert \left\vert \mu \right\vert ^{n}d\mu\\
&\leq &C(b-a)b^{n+m-1}\left\Vert m\right\Vert _{\infty }.
\end{eqnarray*}%
Using Young's inequality, it follows that
\begin{equation*}
\left\Vert T_{m}f\right\Vert _{2}\leq \left\Vert G_{m}\right\Vert
_{2}\left\Vert f\right\Vert _{1}\leq C\left( (b-a)b^{n+m-1}\right) ^{\frac{1%
}{2}}\left\Vert m\right\Vert _{\infty }\left\Vert f\right\Vert _{1}\text{.}
\end{equation*}%
By interpolation with the trivial estimate
\begin{equation*}
\left\Vert T_{m}f\right\Vert _{2}\leq \left\Vert m\right\Vert _{\infty
}\left\Vert f\right\Vert _{2},
\end{equation*}%
we conclude that
\begin{equation*}
\left\Vert T_{m}f\right\Vert _{2}\leq C\left( (b-a)b^{n+m-1}\right) ^{\left(
\frac{1}{p}-\frac{1}{2}\right) }\left\Vert m\right\Vert _{\infty }\left\Vert
f\right\Vert _{p}.
\end{equation*}%
The proof is completed.
\end{proof}

\begin{theorem}
\label{mainTh} Suppose that $1\leq p_{1},p_{2}\leq 2$ and $%
1/p=1/p_{1}+1/p_{2}$. If $\alpha >Q\left( \frac{1}{p}-1\right) $, then $%
S^{\alpha }$ is bounded from $L^{p_{1}}(\mathbb{G})\times L^{p_{2}}(\mathbb{G%
})$ into $L^{p}(\mathbb{G})$.
\end{theorem}

\begin{proof}
In the proof of Theorem \ref{bilinear-kernel}, we have defined the operator
\begin{equation*}
T_{j}^{\alpha }(f,g)=\int_{0}^{\infty }\int_{0}^{\infty }\varphi
_{j}^{\alpha }\left( \lambda _{1},\lambda _{2}\right) P_{\lambda
_{1}}fP_{\lambda _{2}}gd\lambda _{1}d\lambda _{2},
\end{equation*}%
and showed that%
\begin{equation*}
S^{\alpha }=\sum_{j=0}^{\infty }T_{j}^{\alpha }.
\end{equation*}%
So, Theorem \ref{mainTh} would follow if we can show that when $\alpha
>Q\left( \frac{1}{p}-1\right) $, there exists an $\varepsilon >0$ such that
for each $j\geq 0$,
\begin{equation}
\left\Vert T_{j}^{\alpha }\right\Vert _{L^{p_{1}}\times L^{p_{2}}\rightarrow
L^{p}}\leq 2^{-\varepsilon j}.  \label{TJ-}
\end{equation}%
Fixing $j\geq 0$. To prove (\ref{TJ-}), we define $B_{j}=\{\omega
:\left\vert \omega \right\vert \leq 2^{j(1+\gamma )}\}\subseteq \mathbb{G}$
and split the kernel of $T_{j}^{\alpha }$, denoted by $K_{j}^{\alpha }$,
into four parts:
\begin{equation*}
K_{j}^{\alpha }=K_{j}^{1}+K_{j}^{2}+K_{j}^{3}+K_{j}^{4},
\end{equation*}%
where
\begin{eqnarray*}
K_{j}^{1}(\omega _{1},\omega _{2}) &=&K_{j}^{\alpha }(\omega _{1},\omega
_{2})\chi _{B_{j}}(\omega _{1})\chi _{B_{j}}(\omega _{2}), \\
K_{j}^{2}(\omega _{1},\omega _{2}) &=&K_{j}^{\alpha }(\omega _{1},\omega
_{2})\chi _{B_{j}}(\omega _{1})\chi _{B_{j}^{c}}(\omega _{2}), \\
K_{j}^{3}(\omega _{1},\omega _{2}) &=&K_{j}^{\alpha }(\omega _{1},\omega
_{2})\chi _{B_{j}^{c}}(\omega _{1})\chi _{B_{j}}(\omega _{2}), \\
K_{j}^{4}(\omega _{1},\omega _{2}) &=&K_{j}^{\alpha }(\omega _{1},\omega
_{2})\chi _{B_{j}^{c}}(\omega _{1})\chi _{B_{j}^{c}}(\omega _{2}).
\end{eqnarray*}%
Let $T_{j}^{l}$ be the bilinear operator with kernel $K_{j}^{l}$, $l=1,2,3,4$%
. Then, (\ref{TJ-}) would be the consequence of the estimates
\begin{equation*}
\left\Vert T_{j}^{l}\right\Vert _{L^{p_{1}}\times L^{p_{2}}\rightarrow
L^{p}}\leq 2^{-\varepsilon j},\text{ \ \ \ }l=1,2,3,4.
\end{equation*}%
First consider $T_{j}^{4}$. (\ref{Kj-kernel}) tells that for any $N\in
\mathbb{N}
^{+}$ and $\omega _{1}=(x_{1},u_{1})$, $\omega _{2}=(x_{2},u_{2})\in \mathbb{%
G}$,
\begin{equation*}
\left\vert K_{j}^{\alpha }(\left( x_{1},u_{1}\right)
,(x_{2},u_{2}))\right\vert \leq C_{N}2^{-j(\alpha -4N+1)}\left( 1+\left\vert
\left( A_{\frac{\mu _{1}}{\left\vert \mu _{1}\right\vert }%
}x_{1},u_{1}\right) \right\vert \right) ^{-2N}\left( 1+\left\vert \left( A_{%
\frac{\mu _{2}}{\left\vert \mu _{2}\right\vert }}x_{2},u_{2}\right)
\right\vert \right) ^{-2N}.
\end{equation*}%
Applying H\"{o}lder's inequality, Young's inequality and changing variables,
we can easily get that
\begin{eqnarray*}
&&\left\Vert T_{j}^{4}(f,g)\right\Vert _{p} \\
&\leq &C_{N}2^{j\cdot 4N}\left\Vert f\right\Vert _{p_{1}}\left\Vert
g\right\Vert _{p_{2}}\left( \int_{\left\vert (x_{1},u_{1})\right\vert \geq
2^{j(1+\gamma )}}\left( 1+\left\vert \left( A_{\frac{\mu _{1}}{\left\vert
\mu _{1}\right\vert }}x_{1},u_{1}\right) \right\vert \right)
^{-2N}dx_{1}du_{1}\right) \\
&&\times \left( \int_{\left\vert (x_{2},u_{2})\right\vert \geq 2^{j(1+\gamma
)}}\left( 1+\left\vert \left( A_{\frac{\mu _{2}}{\left\vert \mu
_{2}\right\vert }}x_{2},u_{2}\right) \right\vert \right)
^{-2N}dx_{2}du_{2}\right) \\
&\leq &C_{N}2^{j\cdot 4N}\left\Vert f\right\Vert _{p_{1}}\left\Vert
g\right\Vert _{p_{2}}\left( \int_{\left\vert \left( A_{\frac{\mu _{1}}{%
\left\vert \mu _{1}\right\vert }}^{-1}z_{1},u_{1}\right) \right\vert \geq
2^{j(1+\gamma )}}\left( 1+\left\vert \left( z_{1},u_{1}\right) \right\vert
\right) ^{-2N}dz_{1}du_{1}\right) \\
&&\times \left( \int_{\left\vert \left( A_{\frac{\mu _{2}}{\left\vert \mu
_{2}\right\vert }}^{-1}z_{2},u_{2}\right) \right\vert \geq 2^{j(1+\gamma
)}}\left( 1+\left\vert \left( z_{2},u_{2}\right) \right\vert \right)
^{-2N}dz_{2}du_{2}\right) \left\vert \det A_{\frac{\mu _{1}}{\left\vert \mu
_{1}\right\vert }}\right\vert ^{-1}\left\vert \det A_{\frac{\mu _{2}}{%
\left\vert \mu _{2}\right\vert }}\right\vert ^{-1} \\
&\leq &C_{N}2^{j\cdot 4N}\left\Vert f\right\Vert _{p_{1}}\left\Vert
g\right\Vert _{p_{2}}\int_{\left\vert \left( z_{1},u_{1}\right) \right\vert
\geq 2^{j(1+\gamma )}}\left( 1+\left\vert \left( z_{1},u_{1}\right)
\right\vert \right) ^{-2N}dz_{1}du_{1} \\
&&\times \int_{\left\vert \left( z_{2},u_{2}\right) \right\vert \geq
2^{j(1+\gamma )}}\left( 1+\left\vert \left( z_{2},u_{2}\right) \right\vert
\right) ^{-2N}dz_{2}du_{2} \\
&\leq &C_{N}2^{j\cdot 4N}\left\Vert f\right\Vert _{p_{1}}\left\Vert
g\right\Vert _{p_{2}}\left( \int_{\left\vert \omega _{1}\right\vert \geq
2^{j(1+\gamma )}}\left( 1+\left\vert \omega _{1}\right\vert \right)
^{-2N}d\omega _{1}\right) \left( \int_{\left\vert \omega _{2}\right\vert
\geq 2^{j(1+\gamma )}}\left( 1+\left\vert \omega _{2}\right\vert \right)
^{-2N}d\omega _{2}\right) \\
&\leq &C_{N}2^{j\cdot 4N}\left\Vert f\right\Vert _{p_{1}}\left\Vert
g\right\Vert _{p_{2}}\left( \int_{2^{j(1+\gamma )}}^{\infty
}t^{-2N+Q-1}dt\right) ^{2} \\
&\leq &C_{N}2^{j\cdot 4N}2^{j(1+\gamma )(-4N+2Q)}\left\Vert f\right\Vert
_{p_{1}}\left\Vert g\right\Vert _{p_{2}}.
\end{eqnarray*}%
Choosing $N$ large enough such that%
\begin{equation*}
4N\mathbb{\gamma }>2Q(1+\mathbb{\gamma }),
\end{equation*}%
we have
\begin{equation}
\left\Vert T_{j}^{4}\right\Vert _{L^{p_{1}}\times L^{p_{2}}\rightarrow
L^{p}}\leq 2^{-\varepsilon j}  \label{TJ4}
\end{equation}%
for some $\varepsilon >0$.

Consider the estimate of $T_{j}^{3}$. As the proof of Theorem \ref{T1}, the
kernel $K_{j}^{\alpha }$ can be rewritten as%
\begin{equation*}
K_{j}^{\alpha }(\omega _{1},\omega _{2})=c_{N}\int_{0}^{\infty
}\int_{0}^{\infty }\varphi _{j}^{\alpha }(\lambda _{1},\lambda _{2})\frac{%
\partial }{\partial \lambda _{1}}R_{\lambda _{1}}^{0}(\omega _{1})G_{\lambda
_{2}}(\omega _{2})d\lambda _{1}d\lambda _{2},
\end{equation*}%
where%
\begin{equation*}
G_{\lambda _{2}}(\omega _{2})=G_{\lambda _{2}}(x,u)=\sum_{k=0}^{\infty }%
\frac{\lambda _{2}^{n+m-1}}{(2\pi (2k+n))^{n+m}}\int_{\mathbb{S}^{m-1}}%
\widetilde{e}_{k}^{\lambda _{2}\eta _{2}}(x,u)\,d\sigma (\eta _{2}).
\end{equation*}%
Intergration by parts and using the identity (\ref{identity-0}), we get that
\begin{equation*}
K_{j}^{\alpha }(\omega _{1},\omega _{2})=c_{N}\int_{0}^{1}\int_{0}^{1}\left(
\partial _{\lambda _{1}}^{2N+2}\varphi _{j}^{\alpha }(\lambda _{1},\lambda
_{2})\right) \lambda _{1}^{2N+1}R_{\lambda _{1}}^{2N+1}(\omega
_{1})G_{\lambda _{2}}(\omega _{2})d\lambda _{1}d\lambda _{2}.
\end{equation*}%
So,
\begin{eqnarray*}
&&\left\vert K_{j}^{3}(\omega _{1},\omega _{2})\right\vert  \\
&=&\left\vert K_{j}^{\alpha }(\omega _{1},\omega _{2})\chi
_{B_{j}^{c}}(\omega _{1})\chi _{B_{j}}(\omega _{2})\right\vert  \\
&\leq &C\int_{0}^{1}\left\vert \lambda _{1}^{2N+1}R_{\lambda
_{1}}^{2N+1}(\omega _{1})\chi _{B_{j}^{c}}(\omega _{1})\right\vert
\left\vert \int_{0}^{1}\left( \partial _{\lambda _{2}}^{2N+2}\varphi
_{j}^{\alpha }(\lambda _{1},\lambda _{2})\right) G_{\lambda _{2}}(\omega
_{2})\chi _{B_{j}}(\omega _{2})\,d\lambda _{2}\right\vert d\lambda _{1} \\
&\leq &C\sup_{\lambda _{1}\in \lbrack 0,1]}\left\vert \lambda
_{1}^{2N+1}R_{\lambda _{1}}^{2N+1}(\omega _{1})\chi _{B_{j}^{c}}(\omega
_{1})\right\vert \int_{0}^{1}\left\vert \int_{0}^{1}\left( \partial
_{\lambda _{2}}^{2N+2}\varphi _{j}^{\alpha }(\lambda _{1},\lambda
_{2})\right) G_{\lambda _{2}}(\omega _{2})\chi _{B_{j}}(\omega
_{2})\,d\lambda _{2}\right\vert d\lambda _{1}.
\end{eqnarray*}%
It follows that
\begin{eqnarray*}
\left\Vert T_{j}^{3}(f,g)\right\Vert _{p} &\leq &C\left\Vert f\right\Vert
_{p_{1}}\left\Vert g\right\Vert _{p_{2}}\int_{\left\vert \omega
_{1}\right\vert \geq 2^{j(1+\gamma )}}\sup_{\lambda _{1}\in \lbrack
0,1]}\left\vert \lambda _{1}^{2N+1}R_{\lambda _{1}}^{2m+1}(\omega
_{1})\right\vert \,d\omega _{1} \\
&&\times \int_{\left\vert \omega _{2}\right\vert \leq 2^{j(1+\gamma
)}}\int_{0}^{1}\left\vert \int_{0}^{1}\left( \partial _{\lambda
_{2}}^{2N+2}\varphi _{j}^{\alpha }(\lambda _{1},\lambda _{2})\right)
G_{\lambda _{2}}(\omega _{2})\,d\lambda _{2}\right\vert d\lambda _{1}d\omega
_{2}.
\end{eqnarray*}%
Applying the Cauchy-Schwartz's inequality and Lemma \ref{restriction}, we
get that
\begin{eqnarray*}
&&\int_{\left\vert \omega _{2}\right\vert \leq 2^{j(1+\gamma
)}}\int_{0}^{1}\left\vert \int_{0}^{1}\left( \partial _{\lambda
_{2}}^{2N+2}\varphi _{j}^{\alpha }(\lambda _{1},\lambda _{2})\right)
G_{\lambda _{2}}(\omega _{2})\,d\lambda _{2}\right\vert d\lambda _{1}d\omega
_{2} \\
&\leq &2^{j(1+\gamma )\frac{Q}{2}}\left( \int_{\mathbb{G}}\left\vert
\int_{0}^{1}\left( \partial _{\lambda _{2}}^{2N+2}\varphi _{j}^{\alpha
}(\lambda _{1},\lambda _{2})\right) G_{\lambda _{2}}(\omega _{2})\,d\lambda
_{2}\right\vert ^{2}d\omega _{2}\right) ^{\frac{1}{2}} \\
&\leq &C2^{j(1+\gamma )\frac{Q}{2}}\sup_{\lambda _{1},\lambda _{2}\in
\lbrack 0,1]}\left\vert \partial _{\lambda _{2}}^{2N+2}\varphi _{j}^{\alpha
}(\lambda _{1},\lambda _{2})\right\vert  \\
&\leq &C2^{j(1+\gamma )\frac{Q}{2}}2^{j(2N+2)}.
\end{eqnarray*}%
On the other hand, (\ref{es-0}) implies that 
\begin{eqnarray*}
&&\int_{\left\vert \omega _{1}\right\vert \geq 2^{j(1+\gamma
)}}\sup_{\lambda _{1}\in \lbrack 0,1]}\left\vert \lambda
_{1}^{2N+1}R_{\lambda _{1}}^{2N+1}(\omega _{1})\right\vert d\omega _{1} \\
&\leq &C\int_{\left\vert (x_{1},u_{1})\right\vert \geq 2^{j(1+\gamma
)}}\left( 1+\left\vert \left( A_{\frac{\mu _{1}}{\left\vert \mu
_{1}\right\vert }}x_{1},u_{1}\right) \right\vert \right) ^{-2N}dx_{1}du_{1}
\\
&\leq &C\int_{\left\vert \omega _{1}\right\vert \geq 2^{j(1+\gamma )}}\left(
1+\left\vert \omega _{1}\right\vert \right) ^{-2N}d\omega _{1} \\
&\leq &C2^{j(1+\gamma )(-2N+Q)}.
\end{eqnarray*}%
Thus,
\begin{equation*}
\left\Vert T_{j}^{3}(f,g)\right\Vert _{p}\leq C\left\Vert f\right\Vert
_{p_{1}}\left\Vert g\right\Vert _{p_{2}}2^{j\frac{(1+\gamma )Q}{2}%
}2^{j(2N+2)}2^{j(1+\gamma )(-2N+Q)}\text{.}
\end{equation*}%
Choose $N$ large enough such that
\begin{equation*}
2N\gamma \geq \frac{3}{2}Q(1+\gamma )+2,
\end{equation*}%
we have
\begin{equation}
\left\Vert T_{j}^{3}\right\Vert _{L^{p_{1}}\times L^{p_{2}}\rightarrow
L^{p}}\leq C2^{-\varepsilon j}  \label{TJ3}
\end{equation}%
for some $\varepsilon >0$. Obviously, (\ref{TJ3}) also holds for $T_{j}^{2}$.

Now, it remains to estimate $T_{j}^{1}$. $T_{j}^{1}$ is denoted by
\begin{equation*}
T_{j}^{1}(f,g)\left( \omega \right) =\int_{\mathbb{G}}\int_{\mathbb{G}%
}f(\omega \omega _{1}^{-1})g(\omega \omega _{2}^{-1})K_{j}^{1}(\omega
_{1},\omega _{2})d\omega _{1}d\omega _{2}.
\end{equation*}%
We assume that $\omega \in B_{j}(\xi ,\frac{1}{4}%
) $ and slipt the functions $f,g$ into three parts respectively: $%
f=f_{1}+f_{2}+f_{3}$, $g=g_{1}+g_{2}+g_{3}$, where
\begin{eqnarray*}
f_{1} &=&f\chi _{B_{j}(\xi ,\frac{3}{4})},\quad \quad\quad\quad\quad %
g_{1}=g\chi _{B_{j}(\xi ,\frac{3}{4})}, \\
f_{2} &=&f\chi _{B_{j}(\xi ,\frac{5}{4})\backslash B_{j}(\xi ,\frac{3}{4}%
)},\quad\quad g_{2}=g\chi _{B_{j}(\xi ,\frac{5}{4})\backslash
B_{j}(\xi ,\frac{3}{4})}, \\
f_{3} &=&f\chi _{\mathbb{G}\backslash B_{j}(\xi ,\frac{5}{4})},\quad\quad\quad\quad g_{3}=g\chi _{\mathbb{G}\backslash B_{j}(\xi ,\frac{5}{4})}.
\end{eqnarray*}%
Based on this decomposition, we notice that if $f_{3}\neq 0$ or $g_{3}\neq 0$%
,
\begin{equation*}
\left\vert \xi ^{-1}\cdot \omega \omega _{1}^{-1}\right\vert \geq \frac{5}{4}%
2^{j(1+\gamma )}\text{ or }\left\vert \xi ^{-1}\cdot \omega \omega
_{2}^{-1}\right\vert \geq \frac{5}{4}2^{j(1+\gamma )}.
\end{equation*}%
Since $\left\vert \xi ^{-1}\omega
\right\vert \leq \frac{1}{4}2^{j(1+\gamma )}$, by the triangle
inequality, we get that
\begin{equation*}
\left\vert \omega _{1}\right\vert \geq 2^{j(1+\gamma )}\text{ or }\left\vert
\omega _{2}\right\vert \geq 2^{j(1+\gamma )}.
\end{equation*}%
Because the kernel $K_{j}^{1}$ is supported on $B_{j}\times B_{j}$, we have $%
T_{j}^{1}(f_{3},g)=0$ and $T_{j}^{1}(f,g_{3})=0$. Similarly, $f_{2}\neq 0$
and $g_{2}\neq 0$ yield that
\begin{equation*}
\left\vert \omega _{1}\right\vert \geq \frac{1}{2}2^{j(1+\gamma )}\text{ and
}\left\vert \omega _{2}\right\vert \geq \frac{1}{2}2^{j(1+\gamma )}.
\end{equation*}%
So, we can repeat the proof of (\ref{TJ4}) to obtain that
\begin{eqnarray}
\left\Vert T_{j}^{1}(f_{2},g_{2})\right\Vert _{L^{p}(B_{j}(\xi ,\frac{1}{4}%
))} &\leq &C2^{-\varepsilon j}\left\Vert f_{2}\right\Vert _{p_{1}}\left\Vert
g_{2}\right\Vert _{p_{2}}  \notag \\
&\leq &C2^{-\varepsilon j}\left\Vert f\right\Vert _{L^{p_{1}}(B_{j}(\xi ,%
\frac{5}{4}))}\left\Vert g\right\Vert _{L^{p_{2}}(B_{j}(\xi ,\frac{5}{4}))}.
\label{f7}
\end{eqnarray}%
Taking the $L^{p}$ norm with respect to $\xi $, as the proof of Theorem \ref{T1}, we get that
\begin{equation}
\left\Vert T_{j}^{1}(f_{2},g_{2})\right\Vert _{p}\leq C2^{-\varepsilon
j}\left\Vert f\right\Vert _{p_{1}}\left\Vert g\right\Vert _{p_{2}}.
\label{f3}
\end{equation}%
If $f_{1}\neq 0$ and $g_{2}\neq 0$, we have
\begin{equation*}
\left\vert \omega _{1}\right\vert \leq 2^{j(1+\gamma )}\text{ and }%
\left\vert \omega _{2}\right\vert \geq \frac{1}{2}2^{j(1+\gamma )}.
\end{equation*}%
Repeating the proof of (\ref{TJ3}), we can conclude that
\begin{eqnarray*}
\left\Vert T_{j}^{1}(f_{1},g_{2})\right\Vert _{L^{p}(B_{j}(\xi ,\frac{1}{4}%
))} &\leq &C2^{-\varepsilon j}\left\Vert f_{1}\right\Vert _{p_{1}}\left\Vert
g_{2}\right\Vert _{p_{2}} \\
&\leq &C2^{-\varepsilon j}\left\Vert f\right\Vert _{L^{p_{1}}(B_{j}(\xi ,%
\frac{3}{4}))}\left\Vert g\right\Vert _{L^{p_{2}}(B_{j}(\xi ,\frac{5}{4}))}.
\end{eqnarray*}%
It follows that
\begin{equation}
\left\Vert T_{j}^{1}(f_{1},g_{2})\right\Vert _{p}\leq C2^{-\varepsilon
j}\left\Vert f\right\Vert _{p_{1}}\left\Vert g\right\Vert _{p_{2}}.
\label{f2}
\end{equation}%
Obviously, (\ref{f2}) also holds for $T_{j}^{1}(f_{2},g_{1})$. Finally, we
consider $T_{j}^{1}(f_{1},g_{1})$. Because $f_{1},g_{1}\neq 0$ implies that
\begin{equation*}
\left\vert \omega _{1}\right\vert \leq 2^{j(1+\gamma )}\text{ and }%
\left\vert \omega _{1}\right\vert \leq 2^{j(1+\gamma )},
\end{equation*}%
so
\begin{equation}
T_{j}^{1}(f_{1},g_{1})(\omega )=T_{j}^{\alpha }(f_{1},g_{1})(\omega )
\label{f1}
\end{equation}%
for any $\omega \in B_{j}\left( \xi ,\frac{1}{4}\right) $. Notice that $%
T_{j}^{\alpha }$ can be written as
\begin{eqnarray*}
T_{j}^{\alpha }(f,g) &=&\int_{0}^{\infty }\int_{0}^{\infty }\varphi
_{j}^{\alpha }\left( \lambda _{1},\lambda _{2}\right) P_{\lambda
_{1}}fP_{\lambda _{2}}g\,d\lambda _{1}d\lambda _{2} \\
&=&C\int_{[-1,1]^{2}}\varphi _{j}^{\alpha }\left( \left\vert \lambda
_{1}\right\vert ,\left\vert \lambda _{2}\right\vert \right) P_{\left\vert
\lambda _{1}\right\vert }fP_{\left\vert \lambda _{2}\right\vert }g\,d\lambda
_{1}d\lambda _{2}\text{.}
\end{eqnarray*}%
Since that for any fixed $s\in \lbrack -1,1]$, the function
\begin{equation*}
t\rightarrow \varphi _{j}^{\alpha }\left( \left\vert s\right\vert
,\left\vert t\right\vert \right)
\end{equation*}%
is supported in $[-1,1]$ and vanishes at endpoints $\pm 1$, we can expand
this function in Fourier series by considering a periodic extension on $%
\mathbb{R}$ of period $2$, i.e.
\begin{equation*}
\varphi _{j}^{\alpha }(\left\vert s\right\vert ,\left\vert t\right\vert
)=\sum_{k\in \mathbb{Z}}\gamma _{j,k}^{\alpha }(s)e^{i\pi kt},
\end{equation*}%
where the Fourier coefficients are given by
\begin{equation*}
\gamma _{j,k}^{\alpha }(s)=\frac{1}{2}\int_{-1}^{1}\varphi _{j}^{\alpha
}(\left\vert s\right\vert ,\left\vert t\right\vert )e^{-i\pi kt}\,dt\text{.}
\end{equation*}%
It is easy to see that for any $0<\delta <\alpha $,
\begin{equation*}
\sup_{s\in \lbrack -1,1]}\left\vert \gamma _{j,k}^{\alpha }(s)\right\vert
\left( 1+\left\vert k\right\vert \right) ^{1+\delta }\leq C2^{-j(\alpha
-\delta )}.
\end{equation*}%
Then, $T_{j}^{\alpha }$ can be expressed by
\begin{eqnarray*}
T_{j}^{\alpha }(f,g) &=&C\int_{[-1,1]^{2}}\varphi _{j}^{\alpha }\left(
\left\vert \lambda _{1}\right\vert ,\left\vert \lambda _{2}\right\vert
\right) P_{\left\vert \lambda _{1}\right\vert }fP_{\left\vert \lambda
_{2}\right\vert }g\,d\lambda _{1}d\lambda _{2} \\
&=&C\sum_{k\in \mathbb{Z}}\int_{[-1,1]^{2}}\gamma _{j,k}^{\alpha }(\lambda
_{1})e^{i\pi k\lambda _{2}}P_{\left\vert \lambda _{1}\right\vert
}fP_{\left\vert \lambda _{2}\right\vert }g\,d\lambda _{1}d\lambda _{2} \\
&=&C\sum_{k\in \mathbb{Z}}\int_{-1}^{1}\gamma _{j,k}^{\alpha }(\lambda
_{1})P_{\left\vert \lambda _{1}\right\vert }f\,d\lambda
_{1}\int_{-1}^{1}e^{i\pi k\lambda _{2}}P_{\left\vert \lambda _{2}\right\vert
}g\,d\lambda _{2}.
\end{eqnarray*}%
Applying the Cauchy-Schwartz's inequality and Lemma \ref{restriction}, we
have
\begin{eqnarray}
\left\Vert T_{j}^{\alpha }(f,g)\right\Vert _{1} &\leq &C\sum_{k\in \mathbb{Z}%
}\left\Vert \int_{-1}^{1}\gamma _{j,k}^{\alpha }(\lambda _{1})P_{\left\vert
\lambda _{1}\right\vert }f\,d\lambda _{1}\right\Vert _{2}\left\Vert
\int_{-1}^{1}e^{i\pi k\lambda _{2}}P_{\left\vert \lambda _{2}\right\vert
}g\,d\lambda _{2}\right\Vert _{2}  \notag \\
&\leq &C\sum_{k\in \mathbb{Z}}(1+\left\vert k\right\vert )^{-1-\delta
}\left( \sup_{s\in \lbrack -1,1]}\left\vert \gamma _{j,k}^{\alpha
}(s)\right\vert \left( 1+\left\vert k\right\vert \right) ^{1+\delta }\right)
\left\Vert f\right\Vert _{p_{1}}\left\Vert g\right\Vert _{p_{2}}  \notag \\
&\leq &C2^{-j(\alpha -\delta )}\left\Vert f\right\Vert _{p_{1}}\left\Vert
g\right\Vert _{p_{2}}.  \label{ff}
\end{eqnarray}%
Using the H\"{o}lder's inequality and (\ref{f1}), we have
\begin{eqnarray*}
\left\Vert T_{j}^{1}(f_{1},g_{1})\right\Vert _{L^{p}(B_{j}(\xi ,\frac{1}{4}%
))} &\leq &2^{j(1+\gamma )Q\left( \frac{1}{p}-1\right) }\left\Vert
T_{j}^{1}(f_{1},g_{1})\right\Vert _{L^{1}(B_{j}(\xi ,\frac{1}{4}))} \\
&=&2^{j(1+\gamma )Q\left( \frac{1}{p}-1\right) }\left\Vert T_{j}^{\alpha
}(f_{1},g_{1})\right\Vert _{L^{1}(B_{j}(\xi ,\frac{1}{4}))} \\
&\leq &C2^{-j(\alpha -\delta )}2^{j(1+\gamma )Q\left( \frac{1}{p}-1\right)
}\left\Vert f_{1}\right\Vert _{p_{1}}\left\Vert g_{1}\right\Vert _{p_{2}} \\
&\leq &C2^{-j(\alpha -\delta )}2^{j(1+\gamma )Q\left( \frac{1}{p}-1\right)
}\left\Vert f\right\Vert _{L^{p_{1}}(B_{j}(\xi ,\frac{3}{4}))}\left\Vert
g\right\Vert _{L^{p_{2}}(B_{j}(\xi ,\frac{3}{4}))}.
\end{eqnarray*}%
Taking the $L^{p}$ norm with respect to $\xi $, it follows that
\begin{equation}
\left\Vert T_{j}^{1}(f_{1},g_{1})\right\Vert _{p}\leq C2^{-j(\alpha -\delta
)}2^{j(1+\gamma )Q\left( \frac{1}{p}-1\right) }\left\Vert f\right\Vert
_{p_{1}}\left\Vert g\right\Vert _{p_{2}}.  \label{f5}
\end{equation}%
Combining (\ref{f3}), (\ref{f2}) and (\ref{f5}), we can conclude that
\begin{equation*}
\left\Vert T_{j}^{1}(f,g)\right\Vert _{p}\leq C2^{-j(\alpha -\delta
)}2^{j(1+\gamma )Q\left( \frac{1}{p}-1\right) }\left\Vert f\right\Vert
_{p_{1}}\left\Vert g\right\Vert _{p_{2}}.
\end{equation*}%
Thus, whenever $\alpha >Q\left( \frac{1}{p}-1\right) $, we can choose
suitable $\gamma ,\delta >0$ such that
\begin{equation*}
\alpha >Q(1+\gamma )\left( \frac{1}{p}-1\right) +\delta ,
\end{equation*}%
which means that there exists an $\varepsilon >0$ such that
\begin{equation*}
\left\Vert T_{j}^{1}\right\Vert _{L^{1}\times L^{p_{2}}\rightarrow
L^{p}}\leq 2^{-\varepsilon j}.
\end{equation*}%
The proof of Theorem \ref{mainTh} is completed.
\end{proof}

\section{Boundedness of $S^{\protect\alpha}$ for particular points}

In this section, we investigate the boundedness of $S^{\alpha }$ for some
specific triples of points $(p_{1},p_{2},p)$.

\subsection{The point $(1,\infty ,1)$}

\begin{theorem}
\label{1infty} If $\alpha >\frac{Q}{2}$, then $S^{\alpha }$ is bounded from $%
L^{1}(\mathbb{G})\times L^{\infty }(\mathbb{G})$ to $L^{1}(\mathbb{G})$.
\end{theorem}

\begin{proof}
We keep the notations in Section 5. Note that (\ref{TJ4}), (\ref{TJ3}), (\ref%
{f3}) and (\ref{f2}) hold for any $\alpha >0$. So, the proof of Theorem \ref%
{mainTh} is valid apart from the estimate of $T_{j}^{1}(f_{1},g_{1})$.
According to (\ref{ff}), we know that for any $0<\delta <\alpha $,
\begin{equation*}
\left\Vert T_{j}^{\alpha }(f,g)\right\Vert _{1}\leq C2^{-j(\alpha -\delta
)}\left\Vert f\right\Vert _{1}\left\Vert g\right\Vert _{2}.
\end{equation*}%
Using the H\"{o}lder's inequality, we have
\begin{eqnarray*}
\left\Vert T_{j}^{1}(f_{1},g_{1})\right\Vert _{L^{1}(B_{j}(\xi ,\frac{1}{4}%
))} &=&\left\Vert T_{j}^{\alpha }(f_{1},g_{1})\right\Vert _{L^{1}(B_{j}(\xi ,%
\frac{1}{4}))} \\
&\leq &C2^{-j(\alpha -\delta )}\left\Vert f\right\Vert _{L^{1}(B_{j}(\xi ,%
\frac{3}{4}))}\left\Vert g\right\Vert _{L^{2}(B_{j}(\xi ,\frac{3}{4}))} \\
&\leq &C2^{-j(\alpha -\delta )}2^{j(1+\gamma )\frac{Q}{2}}\left\Vert
f\right\Vert _{L^{1}(B_{j}(\xi ,\frac{3}{4}))}\left\Vert g\right\Vert
_{L^{\infty }(B_{j}(\xi ,\frac{1}{4}))}.
\end{eqnarray*}%
Taking the $L^{p}$ norm with respect to $\xi $ yields that
\begin{equation*}
\left\Vert T_{j}^{1}(f_{1},g_{1})\right\Vert _{1}\leq C2^{-j(\alpha -\delta
)}2^{j(1+\gamma )\frac{Q}{2}}\left\Vert f\right\Vert _{1}\left\Vert
g\right\Vert _{\infty }.
\end{equation*}%
Thus, whenever $\alpha >\frac{Q}{2}$, we can choose $\gamma ,\delta >0$ such
that $\alpha >\frac{(1+\gamma )}{2}Q+\delta $, which means that there exists
$\varepsilon >0$ such that
\begin{equation*}
\left\Vert T_{j}^{1}(f_{1},g_{1})\right\Vert _{1}\leq C2^{-\varepsilon
j}\left\Vert f\right\Vert _{1}\left\Vert g\right\Vert _{\infty }.
\end{equation*}
\end{proof}

\subsection{The point $(\infty ,\infty ,\infty )$}

\begin{theorem}
\label{infty}If $\alpha >Q-\frac{1}{2}$, then $S^{\alpha }$ is bounded from $%
L^{\infty }(\mathbb{G})\times L^{\infty }(\mathbb{G})$ into $L^{\infty }(%
\mathbb{G})$.
\end{theorem}

\begin{proof}
We still keep the notation in last Section. To proof this Theorem, it
suffices to estimate $T_{j}^{1}(f_{1},g_{1})$. Notice that
\begin{equation*}
P_{\lambda }f(x,u)=\sum_{k=0}^{\infty }\frac{\lambda ^{n+m-1}}{\left( 2\pi
(2k+n)\right) ^{n+m}}\int_{\mathbb{S}^{m-1}}f\ast \widetilde{e}_{k}^{\lambda
\eta }(x,u)d\sigma (\eta ),
\end{equation*}%
and
\begin{equation*}
f\ast e_{k}^{\lambda \eta }(x,u)=e^{-i\lambda \eta (u)}\left( \left( f_{\eta
}^{\lambda \eta }\times _{\lambda }\varphi _{k}^{\lambda }\right) \circ
A_{\eta }\right) (x)\text{.}
\end{equation*}%
Then, $T_{j}^{\alpha }(f,g)$ can be written as
\begin{eqnarray*}
&&T_{j}^{\alpha }(f,g)(x,u) \\
&=&\frac{1}{(2\pi )^{Q}}\sum_{k=0}^{\infty }\sum_{l=0}^{\infty }\int_{%
\mathbb{R}
}\int_{%
\mathbb{R}
}\varphi _{j}^{\alpha }((2k+n)\lambda _{1},(2l+n)\lambda _{2})\int_{\mathbb{S%
}^{m-1}}f\ast e_{k}^{\lambda _{1}\eta _{1}}(x,u)d\sigma (\eta _{1}) \\
&&\text{ \ \ \ \ \ \ \ \ \ \ \ \ \ \ \ \ \ \ \ \ }\times \int_{\mathbb{S}%
^{m-1}}f\ast e_{l}^{\lambda _{2}\eta _{2}}(x,u)d\sigma (\eta _{2})\lambda
_{1}^{n+m-1}\lambda _{2}^{n+m-1}d\lambda _{1}d\lambda _{2} \\
&=&\frac{1}{(2\pi )^{Q}}\sum_{k=0}^{\infty }\sum_{k=0}^{\infty
}\int_{0}^{\infty }\int_{0}^{\infty }\varphi _{j}^{\alpha }\left(
(2k+n)\lambda _{1},(2l+n)\lambda _{2}\right) \int_{\mathbb{S}%
^{m-1}}e^{-i\lambda _{1}\eta _{1}(u)}\left( \left( f_{\eta _{1}}^{\lambda
_{1}\eta _{1}}\times _{\lambda _{1}}\varphi _{k}^{\lambda _{1}}\right) \circ
A_{\eta _{1}}\right) (x)d\sigma (\eta _{1}) \\
&&\text{ \ \ \ \ \ \ \ \ \ \ \ \ \ \ \ \ \ \ \ \ }\times \int_{\mathbb{S}%
^{m-1}}e^{-i\lambda _{2}\eta _{2}(u)}\left( \left( f_{\eta _{2}}^{\lambda
_{2}\eta _{2}}\times _{\lambda _{2}}\varphi _{l}^{\lambda _{2}}\right) \circ
A_{\eta _{2}}\right) (x)d\sigma (\eta _{2})\,\lambda _{1}^{n+m-1}\lambda
_{2}^{n+m-1}d\lambda _{1}d\lambda _{2} \\
&=&\frac{1}{(2\pi )^{Q}}\int_{%
\mathbb{R}
^{m}}\int_{%
\mathbb{R}
^{m}}\sum_{k=0}^{\infty }\sum_{l=0}^{\infty }e^{-i(\mu _{1}+\mu
_{2})(u)}\varphi _{j}^{\alpha }((2k+n)\left\vert \mu _{1}\right\vert
,(2l+n)\left\vert \mu _{2}\right\vert ) \\
&&\text{ \ \ \ \ \ \ \ \ }\times \left( \left( f_{\frac{\mu _{1}}{\left\vert
\mu _{1}\right\vert }}^{\mu _{1}}\times _{\left\vert \mu _{1}\right\vert
}\varphi _{k}^{\left\vert \mu _{1}\right\vert }\right) \circ A_{\frac{\mu
_{1}}{\left\vert \mu _{1}\right\vert }}\right) (x)\left( \left( f_{\frac{\mu
_{2}}{\left\vert \mu _{2}\right\vert }}^{\mu _{2}}\times _{\left\vert \mu
_{2}\right\vert }\varphi _{l}^{\left\vert \mu _{2}\right\vert }\right) \circ
A_{\frac{\mu _{2}}{\left\vert \mu _{2}\right\vert }}\right) (x)\left\vert
\mu _{1}\right\vert ^{n}\left\vert \mu _{2}\right\vert ^{n}d\mu _{1}d\mu
_{2}.
\end{eqnarray*}%
Using (\ref{varphi}), we get that for any $0<\delta <m$,
\begin{eqnarray*}
&&\left\Vert T_{j}^{\alpha }(f,g)\right\Vert _{\infty } \\
&\leq &C\int_{%
\mathbb{R}
^{m}}\int_{%
\mathbb{R}
^{m}}\sum_{k=0}^{\infty }\sum_{l=0}^{\infty }\left\vert \varphi _{j}^{\alpha
}((2k+n)\left\vert \mu _{1}\right\vert ,(2l+n)\left\vert \mu _{2}\right\vert
)\right\vert \\
&&\times \left\Vert \left( f_{\frac{\mu _{1}}{\left\vert \mu _{1}\right\vert
}}^{\mu _{1}}\times _{\left\vert \mu _{1}\right\vert }\varphi
_{k}^{\left\vert \mu _{1}\right\vert }\right) \circ A_{\frac{\mu _{1}}{%
\left\vert \mu _{1}\right\vert }}\right\Vert _{\infty }\left\Vert \left( g_{%
\frac{\mu _{2}}{\left\vert \mu _{2}\right\vert }}^{\mu _{2}}\times
_{\left\vert \mu _{2}\right\vert }\varphi _{l}^{\left\vert \mu
_{2}\right\vert }\right) \circ A_{\frac{\mu _{2}}{\left\vert \mu
_{2}\right\vert }}\right\Vert _{\infty }\left\vert \mu _{1}\right\vert
^{n}\left\vert \mu _{2}\right\vert ^{n}d\mu _{1}d\mu _{2} \\
&\leq &C\int_{\mathbb{R}^{m}}\int_{\mathbb{R}^{m}}\sum_{k=0}^{\infty
}\sum_{l=0}^{\infty }\left\vert \varphi _{j}^{\alpha }\left(
(2k+n)\left\vert \mu _{1}\right\vert ,(2l+n)\left\vert \mu _{2}\right\vert
\right) \right\vert \\
&&\times \left\Vert f_{\frac{\mu _{1}}{\left\vert \mu _{1}\right\vert }%
}^{\mu _{1}}\times _{\left\vert \mu _{1}\right\vert }\varphi
_{k}^{\left\vert \mu _{1}\right\vert }\right\Vert _{\infty }\left\Vert g_{%
\frac{\mu _{2}}{\left\vert \mu _{2}\right\vert }}^{\mu _{2}}\times
_{\left\vert \mu _{2}\right\vert }\varphi _{l}^{\left\vert \mu
_{2}\right\vert }\right\Vert _{\infty }\left\vert \mu _{1}\right\vert
^{n}\left\vert \mu _{2}\right\vert ^{n}d\mu _{1}d\mu _{2} \\
&\leq &C\int_{\mathbb{R}^{m}}\int_{\mathbb{R}^{m}}\sum_{k=0}^{\infty
}\sum_{l=0}^{\infty }\left\vert \varphi _{j}^{\alpha }\left(
(2k+n)\left\vert \mu _{1}\right\vert ,(2l+n)\left\vert \mu _{2}\right\vert
\right) \right\vert \\
&&\times \left\Vert f_{\frac{\mu _{1}}{\left\vert \mu _{1}\right\vert }%
}^{\mu _{1}}\right\Vert _{2}\left\Vert \varphi _{k}^{\left\vert \mu
_{1}\right\vert }\right\Vert _{2}\left\Vert g_{\frac{\mu _{2}}{\left\vert
\mu _{2}\right\vert }}^{\mu _{2}}\right\Vert _{2}\left\Vert \varphi
_{l}^{\left\vert \mu _{2}\right\vert }\right\Vert _{2}\left\vert \mu
_{1}\right\vert ^{n}\left\vert \mu _{2}\right\vert ^{n}d\mu _{1}d\mu _{2} \\
&\leq &C\int_{\mathbb{R}^{m}}\int_{\mathbb{R}^{m}}\sum_{k=0}^{\infty
}\sum_{l=0}^{\infty }\left\vert \varphi _{j}^{\alpha }\left(
(2k+n)\left\vert \mu _{1}\right\vert ,(2l+n)\left\vert \mu _{2}\right\vert
\right) \right\vert \left\vert \mu _{1}\right\vert ^{\frac{n}{2}}k^{\frac{n-1%
}{2}}\left\vert \mu _{2}\right\vert ^{\frac{n}{2}}l^{\frac{n-1}{2}%
}\left\Vert f_{\frac{\mu _{1}}{\left\vert \mu _{1}\right\vert }}^{\mu
_{1}}\right\Vert _{2}\left\Vert g_{\frac{\mu _{2}}{\left\vert \mu
_{2}\right\vert }}^{\mu _{2}}\right\Vert _{2}\,d\mu _{1}d\mu _{2} \\
&\leq &C\left( \int_{\mathbb{R}^{m}}\int_{\mathbb{R}^{m}}\left(
\sum_{k=0}^{\infty }\sum_{l=0}^{\infty }\left\vert \varphi _{j}^{\alpha
}\left( (2k+n)\left\vert \mu _{1}\right\vert ,(2l+n)\left\vert \mu
_{2}\right\vert \right) \right\vert \left\vert \mu _{1}\right\vert ^{\frac{%
n+\delta }{2}}k^{\frac{n-1}{2}}\left\vert \mu _{2}\right\vert ^{\frac{%
n+\delta }{2}}l^{\frac{n-1}{2}}\right) ^{2}d\mu _{1}d\mu _{2}\right) ^{\frac{%
1}{2}} \\
&&\times \left( \int_{\left\vert \mu _{1}\right\vert \leq 1}\int_{\left\vert
\mu _{2}\right\vert \leq 1}\left\Vert f_{\frac{\mu _{1}}{\left\vert \mu
_{1}\right\vert }}^{\mu _{1}}\right\Vert _{2}^{2}\left\Vert g_{\frac{\mu _{2}%
}{\left\vert \mu _{2}\right\vert }}^{\mu _{2}}\right\Vert _{2}^{2}\left\vert
\mu _{1}\right\vert ^{-\delta }\left\vert \mu _{2}\right\vert ^{-\delta
}\,d\mu _{1}d\mu _{2}\right) ^{\frac{1}{2}} \\
&\leq &C\sum_{k=0}^{\infty }\sum_{l=0}^{\infty }k^{\frac{n-1}{2}}l^{\frac{n-1%
}{2}}\left( \int_{\mathbb{R}^{m}}\int_{\mathbb{R}^{m}}\left\vert \varphi
_{j}^{\alpha }\left( (2k+n)\left\vert \mu _{1}\right\vert ,(2l+n)\left\vert
\mu _{2}\right\vert \right) \right\vert ^{2}\left\vert \mu _{1}\right\vert
^{n+\delta }\left\vert \mu _{2}\right\vert ^{n+\delta }\,d\mu _{1}d\mu
_{2}\right) ^{\frac{1}{2}} \\
&&\times \left( \int_{\left\vert \mu _{1}\right\vert \leq 1}\left\Vert f_{%
\frac{\mu _{1}}{\left\vert \mu _{1}\right\vert }}^{\mu _{1}}\right\Vert
_{2}^{2}\left\vert \mu _{1}\right\vert ^{-\delta }\,d\mu _{1}\right) ^{\frac{%
1}{2}}\left( \int_{\left\vert \mu _{2}\right\vert \leq 1}\left\Vert g_{\frac{%
\mu _{2}}{\left\vert \mu _{2}\right\vert }}^{\mu _{2}}\right\Vert
_{2}^{2}\left\vert \mu _{2}\right\vert ^{-\delta }\,d\mu _{2}\right) ^{\frac{%
1}{2}} \\
&\leq &C\sum_{k=0}^{\infty }\sum_{l=0}^{\infty }k^{\frac{n-1}{2}}l^{\frac{n-1%
}{2}}(2k+n)^{-\frac{n+m+\delta }{2}}(2l+n)^{-\frac{n+m+\delta }{2}}\left(
\int_{\mathbb{R}^{m}}\int_{\mathbb{R}^{m}}\left\vert \varphi _{j}^{\alpha
}\left( \left\vert \mu _{1}\right\vert ,\left\vert \mu _{2}\right\vert
\right) \right\vert ^{2}\left\vert \mu _{1}\right\vert ^{n+\delta
}\left\vert \mu _{2}\right\vert ^{n+\delta }\,d\mu _{1}d\mu _{2}\right) ^{%
\frac{1}{2}} \\
&&\times \left( \int_{\left\vert \mu _{1}\right\vert \leq 1}\left\Vert f_{%
\frac{\mu _{1}}{\left\vert \mu _{1}\right\vert }}^{\mu _{1}}\right\Vert
_{2}^{2}\left\vert \mu _{1}\right\vert ^{-\delta }\,d\mu _{1}\right) ^{\frac{%
1}{2}}\left( \int_{\left\vert \mu _{2}\right\vert \leq 1}\left\Vert g_{\frac{%
\mu _{2}}{\left\vert \mu _{2}\right\vert }}^{\mu _{2}}\right\Vert
_{2}^{2}\left\vert \mu _{2}\right\vert ^{-\delta }\,d\mu _{2}\right) ^{\frac{%
1}{2}} \\
&\leq &C\left( \sum_{l=0}^{\infty }k^{-\frac{m+\delta +1}{2}}\right) \left(
\sum_{k=0}^{\infty }l^{-\frac{m+\delta +1}{2}}\right) \left( \int_{\mathbb{R}%
^{m}}\int_{\mathbb{R}^{m}}\left\vert \varphi _{j}^{\alpha }\left( \left\vert
\mu _{1}\right\vert ,\left\vert \mu _{2}\right\vert \right) \right\vert
^{2}\left\vert \mu _{1}\right\vert ^{n+\delta }\left\vert \mu
_{2}\right\vert ^{n+\delta }\,d\mu _{1}d\mu _{2}\right) ^{\frac{1}{2}} \\
&&\times \left( \int_{\left\vert \mu _{1}\right\vert \leq 1}\left\Vert f_{%
\frac{\mu _{1}}{\left\vert \mu _{1}\right\vert }}^{\mu _{1}}\right\Vert
_{2}^{2}\left\vert \mu _{1}\right\vert ^{-\delta }\,d\mu _{1}\right) ^{\frac{%
1}{2}}\left( \int_{\left\vert \mu _{2}\right\vert \leq 1}\left\Vert g_{\frac{%
\mu _{2}}{\left\vert \mu _{2}\right\vert }}^{\mu _{2}}\right\Vert
_{2}^{2}\left\vert \mu _{2}\right\vert ^{-\delta }\,d\mu _{2}\right) ^{\frac{%
1}{2}} \\
&\leq &C2^{-j(\alpha +\frac{1}{2})}\left( \int_{\left\vert \mu
_{1}\right\vert \leq 1}\left\Vert f_{\frac{\mu _{1}}{\left\vert \mu
_{1}\right\vert }}^{\mu _{1}}\right\Vert _{2}^{2}\left\vert \mu
_{1}\right\vert ^{-\delta }\,d\mu _{1}\right) ^{\frac{1}{2}}\left(
\int_{\left\vert \mu _{2}\right\vert \leq 1}\left\Vert g_{\frac{\mu _{2}}{%
\left\vert \mu _{2}\right\vert }}^{\mu _{2}}\right\Vert _{2}^{2}\left\vert
\mu _{2}\right\vert ^{-\delta }\,d\mu _{2}\right) ^{\frac{1}{2}}.
\end{eqnarray*}%
Because
\begin{equation*}
T_{j}^{1}(f_{1},g_{1})(\omega )=T_{j}^{\alpha }(f_{1},g_{1})(\omega )\quad
\text{for any }\omega \in B_{j}(\xi ,\frac{1}{4}),
\end{equation*}%
we have
\begin{eqnarray}
&&\left\Vert T_{j}^{1}(f_{1},g_{1})\right\Vert _{L^{\infty }(B_{j}(\xi ,%
\frac{1}{4}))}  \label{m0} \\
&=&\left\Vert T_{j}^{\alpha }(f_{1},g_{1})\right\Vert _{L^{\infty
}(B_{j}(\xi ,\frac{1}{4}))}  \notag \\
&\leq &C2^{-j(\alpha +\frac{1}{2})}\left( \int_{\left\vert \mu
_{1}\right\vert \leq 1}\left\Vert (f_{1})_{\frac{\mu _{1}}{\left\vert \mu
_{1}\right\vert }}^{\mu _{1}}\right\Vert _{2}^{2}\left\vert \mu
_{1}\right\vert ^{-\delta }\,d\mu _{1}\right) ^{\frac{1}{2}}\left(
\int_{\left\vert \mu _{2}\right\vert \leq 1}\left\Vert (g_{1})_{\frac{\mu
_{2}}{\left\vert \mu _{2}\right\vert }}^{\mu _{2}}\right\Vert
_{2}^{2}\left\vert \mu _{2}\right\vert ^{-\delta }\,d\mu _{2}\right) ^{\frac{%
1}{2}}.  \notag
\end{eqnarray}%
Considering the integral about $\mu _{1}$, we notice that
\begin{equation*}
\left\Vert (f_{1})_{\frac{\mu _{1}}{\left\vert \mu _{1}\right\vert }}^{\mu
_{1}}\right\Vert _{2}=\left\vert \det A_{\frac{\mu _{1}}{\left\vert \mu
_{1}\right\vert }}\right\vert ^{\frac{1}{2}}\left\Vert f_{1}^{\mu
_{1}}\right\Vert _{2}\leq C2^{j(1+\gamma )n}\left\Vert f_{1}^{\mu
_{1}}\right\Vert _{\infty }\leq C2^{j(1+\gamma )\left( \frac{Q}{2}+m\right)
}\left\Vert f\right\Vert _{L^{\infty }(B_{j}(\xi ,\frac{3}{4}))}.
\end{equation*}%
So, we can get that
\begin{eqnarray*}
&&\int_{\left\vert \mu _{1}\right\vert \leq 1}\left\Vert (f_{1})_{\frac{\mu
_{1}}{\left\vert \mu _{1}\right\vert }}^{\mu _{1}}\right\Vert
_{2}^{2}\left\vert \mu _{1}\right\vert ^{-\delta }\,d\mu _{1} \\
&=&\int_{2^{-2j(1+\gamma )}\leq \left\vert \mu _{1}\right\vert \leq
1}\left\Vert (f_{1})_{\frac{\mu _{1}}{\left\vert \mu _{1}\right\vert }}^{\mu
_{1}}\right\Vert _{2}^{2}\left\vert \mu _{1}\right\vert ^{-\delta }\,d\mu
_{1}+\int_{\left\vert \mu _{1}\right\vert \leq 2^{-2j(1+\gamma )}}\left\Vert
(f_{1})_{\frac{\mu _{1}}{\left\vert \mu _{1}\right\vert }}^{\mu
_{1}}\right\Vert _{2}^{2}\left\vert \mu _{1}\right\vert ^{-\delta }\,d\mu
_{1} \\
&\leq &2^{2j\delta (1+\gamma )}\int_{%
\mathbb{R}
^{m}}\left\Vert (f_{1})_{\frac{\mu _{2}}{\left\vert \mu _{2}\right\vert }%
}^{\mu _{2}}\right\Vert _{2}^{2}\,d\mu _{1}+C2^{j(1+\gamma )\left(
Q+2m\right) }\left\Vert f\right\Vert _{L^{\infty }(B_{j}(\xi ,\frac{3}{4}%
))}^{2}\int_{\left\vert \mu _{1}\right\vert \leq 2^{-2j(1+\gamma
)}}\left\vert \mu _{1}\right\vert ^{-\delta }\,d\mu _{1} \\
&\leq &2^{2j\delta (1+\gamma )}\left\Vert f_{1}\right\Vert
_{2}^{2}+C2^{j(1+\gamma )\left( Q+2m\right) }2^{-2j(1+\gamma )(-\delta
+m)}\left\Vert f\right\Vert _{L^{\infty }(B_{j}(\xi ,\frac{3}{4}))}^{2} \\
&\leq &C2^{j(1+\gamma )(Q+2\delta )}\left\Vert f\right\Vert _{L^{\infty
}(B_{j}(\xi ,\frac{3}{4}))}^{2}.
\end{eqnarray*}%
In the same way, we can also get that
\begin{equation}
\int_{\left\vert \mu _{2}\right\vert \leq 1}\left\Vert (g_{1})_{\frac{\mu
_{2}}{\left\vert \mu _{2}\right\vert }}^{\mu _{2}}\right\Vert
_{2}^{2}\left\vert \mu _{2}\right\vert ^{-\delta }\,d\mu _{2}\leq
C2^{j(1+\gamma )\left( Q+2\delta \right) }\left\Vert g\right\Vert
_{L^{\infty }(B_{j}(\xi ,\frac{3}{4}))}^{2}.  \label{m1}
\end{equation}%
From (\ref{m0}) and the above estimates, we have
\begin{equation*}
\left\Vert T_{j}^{1}(f_{1},g_{1})\right\Vert _{L^{\infty }(B_{j}(\xi ,\frac{1%
}{4}))}\leq C2^{-j(\alpha +\frac{1}{2})}2^{j(1+\gamma )\left( Q+2\delta
\right) }\left\Vert f\right\Vert _{L^{\infty }(B_{j}(\xi ,\frac{3}{4}%
))}\left\Vert g\right\Vert _{L^{\infty }(B_{j}(\xi ,\frac{3}{4}))}.
\end{equation*}%
It follows that
\begin{equation*}
\left\Vert T_{j}^{1}(f_{1},g_{1})\right\Vert _{L^{\infty }}\leq
C2^{-j(\alpha +\frac{1}{2})}2^{j(1+\gamma )\left( Q+2\delta \right)
}\left\Vert f\right\Vert _{L^{\infty }}\left\Vert g\right\Vert _{L^{\infty
}}.
\end{equation*}%
Therefore, whenever $\alpha >Q-\frac{1}{2}$, we can choose $\gamma ,\delta
>0 $ such that $\alpha >(1+\gamma )(Q+2\delta )-\frac{1}{2}$, which implies
there exists an $\varepsilon >0$ such that
\begin{equation*}
\left\Vert T_{j}^{1}(f_{1},g_{1})\right\Vert _{L^{\infty }}\leq
C2^{-\varepsilon j}\left\Vert f\right\Vert _{L^{\infty }}\left\Vert
g\right\Vert _{L^{\infty }}.
\end{equation*}%
The proof of Theorem \ref{infty} is completed.
\end{proof}

\subsection{The point $(2,\infty ,2)$}

\begin{theorem}
\label{Theorem212} If $\alpha >\frac{Q-1}{2}$, then $S^{\alpha }$ is bounded
from $L^{2}(\mathbb{G})\times L^{\infty }(\mathbb{G})$ to $L^{2}(\mathbb{G})$%
.
\end{theorem}

\begin{proof}
As above, it suffices to estimate $T_{j}^{1}(f_{1},g_{1})$. We write $%
T_{j}^{\alpha }(f,g)$ as

\begin{eqnarray*}
&&T_{j}^{\alpha }(f,g)(x,u) \\
&=&\frac{1}{(2\pi )^{Q}}\int_{%
\mathbb{R}
^{m}}\int_{%
\mathbb{R}
^{m}}e^{-i\left( \mu _{1}+\mu _{2}\right) (u)}\sum_{k=0}^{\infty
}\sum_{l=0}^{\infty }\varphi _{j}^{\alpha }((2k+n)\left\vert \mu
_{1}\right\vert ,(2l+n)\left\vert \mu _{2}\right\vert ) \\
&&\times \left( f_{\frac{\mu _{1}}{\left\vert \mu _{1}\right\vert }}^{\mu
_{1}}\times _{\left\vert \mu _{1}\right\vert }\varphi _{k}^{\left\vert \mu
_{1}\right\vert }\right) \left( A_{\frac{\mu _{1}}{\left\vert \mu
_{1}\right\vert }}x\right) \left( g_{\frac{\mu _{2}}{\left\vert \mu
_{2}\right\vert }}^{\mu _{2}}\times _{\left\vert \mu _{2}\right\vert
}\varphi _{l}^{\left\vert \mu _{2}\right\vert }\right) \left( A_{\frac{\mu
_{2}}{\left\vert \mu _{2}\right\vert }}x\right) \left\vert \mu
_{1}\right\vert ^{n}\left\vert \mu _{2}\right\vert ^{n}d\mu _{1}d\mu _{2} \\
&=&\frac{1}{(2\pi )^{Q}}\int_{%
\mathbb{R}
^{m}}e^{-i\mu _{1}(u)}\int_{%
\mathbb{R}
^{m}}\sum_{k=0}^{\infty }\sum_{l=0}^{\infty }\varphi _{j}^{\alpha
}((2k+n)\left\vert \mu _{1}-\mu _{2}\right\vert ,(2l+n)\left\vert \mu
_{2}\right\vert ) \\
&&\times \left( f_{\frac{\mu _{1}-\mu _{2}}{\left\vert \mu _{1}-\mu
_{2}\right\vert }}^{\mu _{1}-\mu _{2}}\times _{\left\vert \mu _{1}-\mu
_{2}\right\vert }\varphi _{k}^{\left\vert \mu _{1}-\mu _{2}\right\vert
}\right) \left( A_{\frac{\mu _{1}-\mu _{2}}{\left\vert \mu _{1-}\mu
_{2}\right\vert }}x\right) \left( g_{\frac{\mu _{2}}{\left\vert \mu
_{2}\right\vert }}^{\mu _{2}}\times _{\left\vert \mu _{2}\right\vert
}\varphi _{l}^{\left\vert \mu _{2}\right\vert }\right) \left( A_{\frac{\mu
_{2}}{\left\vert \mu _{2}\right\vert }}x\right) \left\vert \mu _{1}-\mu
_{2}\right\vert ^{n}\left\vert \mu _{2}\right\vert ^{n}d\mu _{1}d\mu _{2}.
\end{eqnarray*}%
Then, applying the Plancherel theorem in variable $u$, the Minkowski's
inequality, the orthogonality of the special Hermite projections and (\ref%
{varphi}), we get that
\begin{eqnarray*}
&&\int_{\mathbb{R}^{2n}}\int_{\mathbb{R}^{m}}\left\vert T_{j}^{\alpha
}(f,g)(x,u)\right\vert ^{2}dxdu \\
&=&\frac{1}{(2\pi )^{Q}}\int_{\mathbb{R}^{2n}}\int_{\mathbb{R}%
^{m}}\left\vert \int_{%
\mathbb{R}
^{m}}e^{-i\mu _{1}(u)}\int_{%
\mathbb{R}
^{m}}\sum_{k=0}^{\infty }\sum_{l=0}^{\infty }\varphi _{j}^{\alpha
}((2k+n)\left\vert \mu _{1}-\mu _{2}\right\vert ,(2l+n)\left\vert \mu
_{2}\right\vert )\right. \\
&&\left. \times \left( f_{\frac{\mu _{1}-\mu _{2}}{\left\vert \mu _{1}-\mu
_{2}\right\vert }}^{\mu _{1}-\mu _{2}}\times _{\left\vert \mu _{1}-\mu
_{2}\right\vert }\varphi _{k}^{\left\vert \mu _{1}-\mu _{2}\right\vert
}\right) \left( A_{\frac{\mu _{1}-\mu _{2}}{\left\vert \mu _{1-}\mu
_{2}\right\vert }}x\right) \left( g_{\frac{\mu _{2}}{\left\vert \mu
_{2}\right\vert }}^{\mu _{2}}\times _{\left\vert \mu _{2}\right\vert
}\varphi _{l}^{\left\vert \mu _{2}\right\vert }\right) \left( A_{\frac{\mu
_{2}}{\left\vert \mu _{2}\right\vert }}x\right) \left\vert \mu _{1}-\mu
_{2}\right\vert ^{n}\left\vert \mu _{2}\right\vert ^{n}d\mu _{2}d\mu
_{1}\right\vert ^{2}dudx \\
&=&\frac{1}{(2\pi )^{Q}}\int_{\mathbb{R}^{2n}}\int_{\mathbb{R}%
^{m}}\left\vert \int_{%
\mathbb{R}
^{m}}\sum_{k=0}^{\infty }\sum_{l=0}^{\infty }\varphi _{j}^{\alpha
}((2k+n)\left\vert \mu _{1}-\mu _{2}\right\vert ,(2l+n)\left\vert \mu
_{2}\right\vert )\right. \\
&&\left. \times \left( f_{\frac{\mu _{1}-\mu _{2}}{\left\vert \mu _{1}-\mu
_{2}\right\vert }}^{\mu _{1}-\mu _{2}}\times _{\left\vert \mu _{1}-\mu
_{2}\right\vert }\varphi _{k}^{\left\vert \mu _{1}-\mu _{2}\right\vert
}\right) \left( A_{\frac{\mu _{1}-\mu _{2}}{\left\vert \mu _{1-}\mu
_{2}\right\vert }}x\right) \left( g_{\frac{\mu _{2}}{\left\vert \mu
_{2}\right\vert }}^{\mu _{2}}\times _{\left\vert \mu _{2}\right\vert
}\varphi _{l}^{\left\vert \mu _{2}\right\vert }\right) \left( A_{\frac{\mu
_{2}}{\left\vert \mu _{2}\right\vert }}x\right) \left\vert \mu _{1}-\mu
_{2}\right\vert ^{n}\left\vert \mu _{2}\right\vert ^{n}d\mu _{2}\right\vert
^{2}d\mu _{1}dx \\
&\leq &\frac{C}{(2\pi )^{Q}}\int_{\mathbb{R}^{m}}\int_{%
\mathbb{R}
^{2n}}\left( \int_{%
\mathbb{R}
^{m}}\sum_{l=0}^{\infty }\left\vert \sum_{k=0}^{\infty }\varphi _{j}^{\alpha
}((2k+n)\left\vert \mu _{1}-\mu _{2}\right\vert ,(2l+n)\left\vert \mu
_{2}\right\vert )\right. \right. \\
&&\times \left. \left. \left( f_{\frac{\mu _{1}-\mu _{2}}{\left\vert \mu
_{1}-\mu _{2}\right\vert }}^{\mu _{1}-\mu _{2}}\times _{\left\vert \mu
_{1}-\mu _{2}\right\vert }\varphi _{k}^{\left\vert \mu _{1}-\mu
_{2}\right\vert }\right) \left( A_{\frac{\mu _{1}-\mu _{2}}{\left\vert \mu
_{1-}\mu _{2}\right\vert }}x\right) \right) \left\vert \left( g_{\frac{\mu
_{2}}{\left\vert \mu _{2}\right\vert }}^{\mu _{2}}\times _{\left\vert \mu
_{2}\right\vert }\varphi _{l}^{\left\vert \mu _{2}\right\vert }\right)
\left( A_{\frac{\mu _{2}}{\left\vert \mu _{2}\right\vert }}x\right)
\right\vert \left\vert \mu _{1}-\mu _{2}\right\vert ^{n}\left\vert \mu
_{2}\right\vert ^{n}d\mu _{2}\right) ^{2}dxd\mu _{1} \\
&\leq &\frac{C}{(2\pi )^{Q}}\int_{\mathbb{R}^{m}}\left( \int_{%
\mathbb{R}
^{m}}\left\Vert \sum_{l=0}^{\infty }\left\vert \sum_{k=0}^{\infty }\varphi
_{j}^{\alpha }((2k+n)\left\vert \mu _{1}-\mu _{2}\right\vert
,(2l+n)\left\vert \mu _{2}\right\vert )\right. \right. \right. \\
&&\times \left. \left. \left. \left( f_{\frac{\mu _{1}-\mu _{2}}{\left\vert
\mu _{1}-\mu _{2}\right\vert }}^{\mu _{1}-\mu _{2}}\times _{\left\vert \mu
_{1}-\mu _{2}\right\vert }\varphi _{k}^{\left\vert \mu _{1}-\mu
_{2}\right\vert }\right) \left( A_{\frac{\mu _{1}-\mu _{2}}{\left\vert \mu
_{1-}\mu _{2}\right\vert }}x\right) \right\vert \left\vert \left( g_{\frac{%
\mu _{2}}{\left\vert \mu _{2}\right\vert }}^{\mu _{2}}\times _{\left\vert
\mu _{2}\right\vert }\varphi _{l}^{\left\vert \mu _{2}\right\vert }\right)
\circ A_{\frac{\mu _{2}}{\left\vert \mu _{2}\right\vert }}\right\vert
\right\Vert _{2}\left\vert \mu _{1}-\mu _{2}\right\vert ^{n}\left\vert \mu
_{2}\right\vert ^{n}d\mu _{2}\right) ^{2}d\mu _{1} \\
&\leq &\frac{C}{(2\pi )^{Q}}\int_{\mathbb{R}^{m}}\left( \int_{\mathbb{R}%
^{m}}\sum_{l=0}^{\infty }\left\Vert \sum_{k=0}^{\infty }\varphi _{j}^{\alpha
}((2k+n)\left\vert \mu _{1}-\mu _{2}\right\vert ,(2l+n)\left\vert \mu
_{2}\right\vert )\left( f_{\frac{\mu _{1}-\mu _{2}}{\left\vert \mu _{1}-\mu
_{2}\right\vert }}^{\mu _{1}-\mu _{2}}\times _{\left\vert \mu _{1}-\mu
_{2}\right\vert }\varphi _{k}^{\left\vert \mu _{1}-\mu _{2}\right\vert
}\right) \circ A_{\frac{\mu _{1}-\mu _{2}}{\left\vert \mu _{1-}\mu
_{2}\right\vert }}\right\Vert _{2}\right. \\
&&\times \left. \left\Vert \left( g_{\frac{\mu _{2}}{\left\vert \mu
_{2}\right\vert }}^{\mu _{2}}\times _{\left\vert \mu _{2}\right\vert
}\varphi _{l}^{\left\vert \mu _{2}\right\vert }\right) \circ A_{\frac{\mu
_{2}}{\left\vert \mu _{2}\right\vert }}\right\Vert _{\infty }\left\vert \mu
_{1}-\mu _{2}\right\vert ^{n}\left\vert \mu _{2}\right\vert ^{n}d\mu
_{2}\right) ^{2}d\mu _{1} \\
&\leq &\frac{C}{(2\pi )^{Q}}\int_{\mathbb{R}^{m}}\left( \int_{\mathbb{R}%
^{m}}\sum_{l=0}^{\infty }\left( \sum_{k=0}^{\infty }\left\vert \varphi
_{j}^{\alpha }((2k+n)\left\vert \mu _{1}-\mu _{2}\right\vert
,(2l+n)\left\vert \mu _{2}\right\vert )\right\vert ^{2}\right. \right. \\
&&\times \left. \left. \left\Vert \left( f_{\frac{\mu _{1}-\mu _{2}}{%
\left\vert \mu _{1}-\mu _{2}\right\vert }}^{\mu _{1}-\mu _{2}}\times
_{\left\vert \mu _{1}-\mu _{2}\right\vert }\varphi _{k}^{\left\vert \mu
_{1}-\mu _{2}\right\vert }\right) \circ A_{\frac{\mu _{1}-\mu _{2}}{%
\left\vert \mu _{1-}\mu _{2}\right\vert }}\right\Vert _{2}^{2}\right) ^{%
\frac{1}{2}}\left\Vert g_{\frac{\mu _{2}}{\left\vert \mu _{2}\right\vert }%
}^{\mu _{2}}\right\Vert _{2}\left\Vert \varphi _{l}^{\left\vert \mu
_{2}\right\vert }\right\Vert _{2}\left\vert \mu _{1}-\mu _{2}\right\vert
^{n}\left\vert \mu _{2}\right\vert ^{n}d\mu _{2}\right) ^{2}d\mu _{1} \\
&\leq &\frac{C}{(2\pi )^{Q}}\int_{\left\vert \mu _{2}\right\vert \leq
1}\sum_{l\leq \frac{1}{\left\vert \mu _{2}\right\vert }}^{\infty }\left\Vert
g_{\frac{\mu _{2}}{\left\vert \mu _{2}\right\vert }}^{\mu _{2}}\right\Vert
_{2}^{2}\left\Vert \varphi _{l}^{\left\vert \mu _{2}\right\vert }\right\Vert
_{2}^{2}\left\vert \mu _{2}\right\vert ^{2n}d\mu _{2}\int_{\mathbb{R}%
^{m}}\int_{\mathbb{R}^{m}}\sum_{l=0}^{\infty }\sum_{k=0}^{\infty }\left\vert
\varphi _{j}^{\alpha }((2k+n)\left\vert \mu _{1}-\mu _{2}\right\vert
,(2l+n)\left\vert \mu _{2}\right\vert )\right\vert ^{2} \\
&&\times \left\Vert f_{\frac{\mu _{1}-\mu _{2}}{\left\vert \mu _{1}-\mu
_{2}\right\vert }}^{\mu _{1}-\mu _{2}}\times _{\left\vert \mu _{1}-\mu
_{2}\right\vert }\varphi _{k}^{\left\vert \mu _{1}-\mu _{2}\right\vert
}\right\Vert _{2}^{2}\left\vert \det A_{\frac{\mu _{1}-\mu _{2}}{\left\vert
\mu _{1}-\mu _{2}\right\vert }}\right\vert ^{-1}\left\vert \mu _{1}-\mu
_{2}\right\vert ^{2n}d\mu _{2}d\mu _{1} \\
&\leq &\frac{C}{(2\pi )^{Q}}\int_{\left\vert \mu _{2}\right\vert \leq
1}\left( \sum_{l\leq \frac{1}{\left\vert \mu _{2}\right\vert }}^{\infty
}l^{n-1}\right) \left\Vert g_{\frac{\mu _{2}}{\left\vert \mu _{2}\right\vert
}}^{\mu _{2}}\right\Vert _{2}^{2}\left\vert \mu _{2}\right\vert ^{n-\delta
}d\mu _{2}\int_{\mathbb{R}^{m}}\int_{\mathbb{R}^{m}}\sum_{l=0}^{\infty
}\sum_{k=0}^{\infty }\left\vert \varphi _{j}^{\alpha }((2k+n)\left\vert \mu
_{1}\right\vert ,(2l+n)\left\vert \mu _{2}\right\vert )\right\vert ^{2} \\
&&\times \left\Vert f_{\frac{\mu _{1}}{\left\vert \mu _{1}\right\vert }%
}^{\mu _{1}}\times _{\left\vert \mu _{1}\right\vert }\varphi
_{k}^{\left\vert \mu _{1}\right\vert }\right\Vert _{2}^{2}\left\vert \det A_{%
\frac{\mu _{1}}{\left\vert \mu _{1}\right\vert }}\right\vert ^{-1}\left\vert
\mu _{1}\right\vert ^{2n}\left\vert \mu _{2}\right\vert ^{\delta }d\mu
_{2}d\mu _{1} \\
&\leq &\frac{C}{(2\pi )^{Q}}\int_{\left\vert \mu _{2}\right\vert \leq
1}\left\Vert g_{\frac{\mu _{2}}{\left\vert \mu _{2}\right\vert }}^{\mu
_{2}}\right\Vert _{2}^{2}\left\vert \mu _{2}\right\vert ^{-\delta }d\mu
_{2}\sum_{k=0}^{\infty }\int_{\mathbb{R}^{m}}\left\Vert f_{\frac{\mu _{1}}{%
\left\vert \mu _{1}\right\vert }}^{\mu _{1}}\times _{\left\vert \mu
_{1}\right\vert }\varphi _{k}^{\left\vert \mu _{1}\right\vert }\right\Vert
_{2}^{2}\left\vert \det A_{\frac{\mu _{1}}{\left\vert \mu _{1}\right\vert }%
}\right\vert ^{-1}\left\vert \mu _{1}\right\vert ^{2n} \\
&&\times \left( \int_{\mathbb{R}^{m}}\sum_{l=0}^{\infty }\left\vert \varphi
_{j}^{\alpha }((2k+n)\left\vert \mu _{1}\right\vert ,(2l+n)\left\vert \mu
_{2}\right\vert )\right\vert ^{2}\left\vert \mu _{2}\right\vert ^{\delta
}d\mu _{2}\right) d\mu _{1} \\
&\leq &\frac{C}{(2\pi )^{Q}}\int_{\left\vert \mu _{2}\right\vert \leq
1}\left\Vert g_{\frac{\mu _{2}}{\left\vert \mu _{2}\right\vert }}^{\mu
_{2}}\right\Vert _{2}^{2}\left\vert \mu _{2}\right\vert ^{-\delta }d\mu
_{2}\sum_{k=0}^{\infty }\int_{\mathbb{R}^{m}}\left\Vert f_{\frac{\mu _{1}}{%
\left\vert \mu _{1}\right\vert }}^{\mu _{1}}\times _{\left\vert \mu
_{1}\right\vert }\varphi _{k}^{\left\vert \mu _{1}\right\vert }\right\Vert
_{2}^{2}\left\vert \det A_{\frac{\mu _{1}}{\left\vert \mu _{1}\right\vert }%
}\right\vert ^{-1}\left\vert \mu _{1}\right\vert ^{2n} \\
&&\times \left( \sum_{l=0}^{\infty }(2l+n)^{-m-\delta }\int_{\mathbb{R}%
^{m}}\left\vert \varphi _{j}^{\alpha }((2k+n)\left\vert \mu _{1}\right\vert
,\left\vert \mu _{2}\right\vert )\right\vert ^{2}\left\vert \mu
_{2}\right\vert ^{\delta }d\mu _{2}\right) d\mu _{1} \\
&\leq &C2^{-j(2\alpha +1)}\left( \frac{1}{(2\pi )^{Q}}\sum_{k=0}^{\infty
}\int_{\mathbb{R}^{m}}\left\Vert f_{\frac{\mu _{1}}{\left\vert \mu
_{1}\right\vert }}^{\mu _{1}}\times _{\left\vert \mu _{1}\right\vert
}\varphi _{k}^{\left\vert \mu _{1}\right\vert }\right\Vert
_{2}^{2}\left\vert \det A_{\frac{\mu }{\left\vert \mu _{1}\right\vert }%
}\right\vert ^{-1}\left\vert \mu _{1}\right\vert ^{2n}d\mu _{1}\right)
\left( \int_{\left\vert \mu _{2}\right\vert \leq 1}\left\Vert g_{\frac{\mu
_{2}}{\left\vert \mu _{2}\right\vert }}^{\mu _{2}}\right\Vert
_{2}^{2}\left\vert \mu _{2}\right\vert ^{-\delta }d\mu _{2}\right) .
\end{eqnarray*}%
Applying (\ref{expansion}), the Plancherel theorem in variable $u$ and the
orthogonality of the special Hermite projections, we obtain that
\begin{eqnarray*}
\left\Vert f\right\Vert _{2}^{2} &=&\frac{1}{\left( 2\pi \right) ^{Q}}\int_{%
\mathbb{R}
^{2n}}\int_{%
\mathbb{R}
^{m}}\left\vert \int_{%
\mathbb{R}
^{m}}e^{-i\mu _{1}(u)}\sum_{k=0}^{\infty }\left( \left( f_{\frac{\mu _{1}}{%
\left\vert \mu _{1}\right\vert }}^{\mu _{1}}\times _{\left\vert \mu
_{1}\right\vert }\varphi _{k}^{\left\vert \mu _{1}\right\vert }\right) \circ
A_{\frac{\mu _{1}}{\left\vert \mu _{1}\right\vert }}\right) (x)\left\vert
\mu _{1}\right\vert ^{n}d\mu _{1}\right\vert ^{2}dudx \\
&=&\frac{1}{\left( 2\pi \right) ^{Q}}\int_{%
\mathbb{R}
^{2n}}\int_{%
\mathbb{R}
^{m}}\left\vert \sum_{k=0}^{\infty }\left( \left( f_{\frac{\mu _{1}}{%
\left\vert \mu _{1}\right\vert }}^{\mu _{1}}\times _{\left\vert \mu
_{1}\right\vert }\varphi _{k}^{\left\vert \mu _{1}\right\vert }\right) \circ
A_{\frac{\mu _{1}}{\left\vert \mu _{1}\right\vert }}\right) (x)\right\vert
^{2}\left\vert \mu _{1}\right\vert ^{2n}d\mu _{1}dx \\
&=&\frac{1}{\left( 2\pi \right) ^{Q}}\sum_{k=0}^{\infty }\int_{%
\mathbb{R}
^{2n}}\int_{%
\mathbb{R}
^{m}}\left\vert \left( f_{\frac{\mu _{1}}{\left\vert \mu _{1}\right\vert }%
}^{\mu _{1}}\times _{\left\vert \mu _{1}\right\vert }\varphi
_{k}^{\left\vert \mu _{1}\right\vert }\right) \left( A_{\frac{\mu _{1}}{%
\left\vert \mu _{1}\right\vert }}x\right) \right\vert ^{2}\left\vert \mu
_{1}\right\vert ^{2n}d\mu _{1}dx \\
&=&\frac{1}{\left( 2\pi \right) ^{Q}}\sum_{k=0}^{\infty }\int_{%
\mathbb{R}
^{m}}\left\Vert f_{\frac{\mu _{1}}{\left\vert \mu _{1}\right\vert }}^{\mu
_{1}}\times _{\left\vert \mu _{1}\right\vert }\varphi _{k}^{\left\vert \mu
_{1}\right\vert }\right\Vert _{2}^{2}\left\vert \det A_{\frac{\mu _{1}}{%
\left\vert \mu _{1}\right\vert }}\right\vert ^{-1}\left\vert \mu
_{1}\right\vert ^{2n}d\mu _{1}.
\end{eqnarray*}%
At the same time, (\ref{m1}) tells that%
\begin{equation*}
\int_{\left\vert \mu _{2}\right\vert \leq 1}\left\Vert (g_{1})_{\frac{\mu
_{2}}{\left\vert \mu _{2}\right\vert }}^{\mu _{2}}\right\Vert
_{2}^{2}\left\vert \mu _{2}\right\vert ^{-\delta }\,d\mu _{2}\leq
C2^{j(1+\gamma )\left( Q+2\delta \right) }\left\Vert g\right\Vert
_{L^{\infty }(B_{j}(\xi ,\frac{3}{4}))}^{2}.
\end{equation*}%
Thus,
\begin{eqnarray*}
&&\left\Vert T_{j}^{1}(f_{1},g_{1})\right\Vert _{L^{2}(B_{j}(\xi ,\frac{1}{4}%
))}^{2} \\
&=&\left\Vert T_{j}^{\alpha }(f_{1},g_{1})\right\Vert _{L^{2}(B_{j}(\xi ,%
\frac{1}{4}))}^{2} \\
&\leq &C2^{-j(2\alpha +1)}2^{j(1+\gamma )\left( Q+2\delta \right)
}\left\Vert f\right\Vert _{L^{2}(B_{j}(\xi ,\frac{3}{4}))}^{2}\left\Vert
g\right\Vert _{L^{\infty }(B_{j}(\xi ,\frac{3}{4}))}^{2}.
\end{eqnarray*}%
Taking the $L^{2}$ norm with respect to $\xi $, we get that
\begin{equation*}
\left\Vert T_{j}^{1}(f_{1},g_{1})\right\Vert _{2}\leq C2^{-j(\alpha +\frac{1%
}{2})}2^{j(1+\gamma )(\frac{Q}{2}+\delta )}\left\Vert f\right\Vert
_{2}\left\Vert g\right\Vert _{\infty }.
\end{equation*}%
Thus, whenever $\alpha >\frac{Q-1}{2}$, we can choose $\gamma ,\delta >0$
such that $\alpha >(1+\gamma )(\frac{Q}{2}+\delta )-\frac{1}{2}$, which
yields that there exists $\varepsilon >0$ such that
\begin{equation*}
\left\Vert T_{j}^{1}(f_{1},g_{1})\right\Vert _{2}\leq C2^{-\varepsilon
j}\left\Vert f\right\Vert _{2}\left\Vert g\right\Vert _{\infty }.
\end{equation*}%
The proof of Theorem \ref{Theorem212} is completed.
\end{proof}

\section{bilinear interpolation}

Because $p_{1}$ and $p_{2}$ are symmetric, we have obtained, in two sections
above, the boundedness of the bilinear Riesz means $S^{\alpha }$ at some
specific triples of points $(p_{1},p_{2},p)$ like
\begin{equation*}
(1,1,\frac{1}{2}), (2,2,1), (\infty,\infty,\infty), (1,2,\frac{2}{3}), (2,1,
\frac{2}{3}),(1,\infty ,1),(\infty,1,1),(2,\infty,2),(\infty,2,2)\text{.}
\end{equation*}
We can obtain the intermediate boundedness of $S^{\alpha}$ by using of the
bilinear interpolation via complex method adapted to the setting of analytic
families or real method in \cite{Graf}. Bernicot et al. \cite{Bern}
described how to make use of real method. We outline this argument for reader's convenience.

Consider a spherical decomposition of $S^{\alpha }$ as
\begin{equation*}
S^{\alpha }=\sum_{j=0}^{\infty }2^{-j\alpha }T_{j,\alpha },
\end{equation*}%
where
\begin{equation*}
T_{j,\alpha }(f,g)=\int_{0}^{\infty }\int_{0}^{\infty }\varphi _{j,\alpha
}\left( \lambda _{1},\lambda _{2}\right) P_{\lambda _{1}}fP_{\lambda
_{2}}gd\mu (\lambda _{1})d\mu (\lambda _{2})
\end{equation*}%
and
\begin{equation*}
\varphi _{j,\alpha }\left( s,t\right) =2^{j\alpha }(1-s-t)_{+}^{\alpha
}\varphi (2^{j}\left( 1-s-t\right) ).
\end{equation*}%
In the preceding sections, we have actually obtained the estimates of the
form
\begin{equation}
\left\Vert T_{j,\alpha }\right\Vert _{L^{p_{1}}\times L^{p_{2}}\rightarrow
L^{p}}\leq C2^{j\alpha (p_{1},p_{2})}  \label{Tj}
\end{equation}%
at some triples of points $(p_{1},p_{2},p)$. Since $\alpha (p_{1,}p_{2})$
only depends on the point $(p_{1,}p_{2},p)$, (\ref{Tj}) also holds for any
other $T_{j,\alpha ^{\prime }}$, i.e.,
\begin{equation*}
\left\Vert T_{j,\alpha ^{\prime }}\right\Vert _{L^{p_{1}}\times
L^{p_{2}}\rightarrow L^{p}}\leq C2^{j\alpha (p_{1},p_{2})}.
\end{equation*}%
So, fixing $j$ and $\alpha ^{\prime }$ and applying bilinear real
interpolation theorem on $T_{j,\alpha ^{\prime }}$, we can conclude that if
the point $(p_{1},p_{2},p)$ satisfies
\begin{equation*}
\left( \frac{1}{p_{1}},\frac{1}{p_{2}},\frac{1}{p}\right) =(1-\theta )\left(
\frac{1}{p_{1}^{0}},\frac{1}{p_{2}^{0}},\frac{1}{p^{0}}\right) +\theta
\left( \frac{1}{p_{1}^{1}},\frac{1}{p_{2}^{1}},\frac{1}{p^{1}}\right)
\end{equation*}%
for some $\theta \in (0,1)$ and $%
(p_{1}^{0},p_{2}^{0},p^{0}),(p_{1}^{1},p_{2}^{1},p^{1})$, \thinspace we have
that
\begin{equation*}
\left\Vert T_{j,\alpha ^{\prime }}\right\Vert _{L^{p_{1}}\times
L^{p_{2}}\rightarrow L^{p}}\leq C2^{j((1-\theta )\alpha
(p_{1}^{0},p_{2}^{0})+\theta \alpha (p_{1}^{1},p_{2}^{1}))}.
\end{equation*}%
Define $\alpha (p_{1,}p_{2})=(1-\theta )\alpha (p_{1}^{0},p_{2}^{0})+\theta
\alpha (p_{1}^{1},p_{2}^{1})$ and let $\alpha ^{\prime }=\alpha $. It
follows that
\begin{equation*}
\left\Vert S^{\alpha }\right\Vert _{L^{p_{1}}\times L^{p_{2}}\rightarrow
L^{p}}\leq \sum_{j=0}^{\infty }2^{-j\alpha }\left\Vert T_{j,\alpha
}\right\Vert _{L^{p_{1}}\times L^{p_{2}}\rightarrow L^{p}}\leq
C\sum_{j=0}^{\infty }2^{-j\alpha }2^{j\alpha (p_{1,}p_{2})}.
\end{equation*}%
Thus, when $\alpha >\alpha (p_{1,}p_{2})$, we have $\left\Vert S^{\alpha
}\right\Vert _{L^{p_{1}}\times L^{p_{2}}\rightarrow L^{p}}\leq C$, i.e., $%
S^{\alpha }$ is bounded from $L^{p_{1}}(\mathbb{G})\times L^{p_{1}}(\mathbb{G%
})$ to $L^{p}(\mathbb{G})$.

Based on the above argument, we can get the full results on the $%
L^{p_{1}}\times L^{p_{2}}\rightarrow L^{p}$ boundedness of the operator $%
S^{\alpha }$.

\includegraphics[scale=0.9]{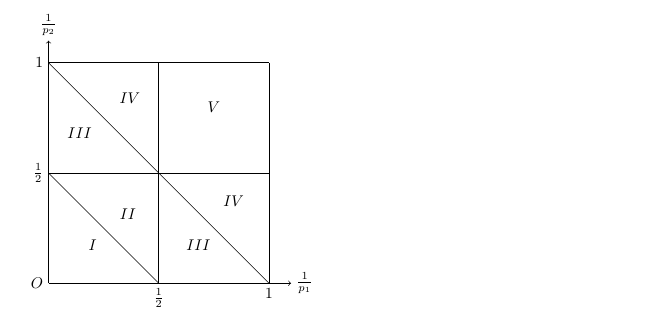}

\begin{theorem}
Let $1\leq p_{1},p_{2}\leq \infty $ and $1/p=1/p_{1}+1/p_{2}$.

(1)(region I) For $2\leq p_{1},p_{2}\leq \infty $ and $p\geq 2$, if $\alpha
>Q\left( 1-\frac{1}{p}\right) -\frac{1}{2}$, then $S^{\alpha }$ is bounded
from $L^{p_{1}}(\mathbb{G})\times L^{p_{1}}(\mathbb{G})$ to $L^{p}(\mathbb{G}%
)$.

(2)(region II) For $2\leq p_{1},p_{2}\leq \infty $ and $1\leq p\leq 2$, if $%
\alpha >(Q-1)\left( 1-\frac{1}{p}\right) $, then $S^{\alpha }$ is bounded
from $L^{p_{1}}(\mathbb{G})\times L^{p_{1}}(\mathbb{G})$ to $L^{p}(\mathbb{G}%
)$.

(3)(region III) For $1\leq p_{1}\leq 2\leq p_{2}\leq \infty $ and $p\geq 1$%
, if $\alpha >Q\left( \frac{1}{2}-\frac{1}{p_{2}}\right) -\left( 1-\frac{1}{p%
}\right) $, then $S^{\alpha }$ is bounded from $L^{p_{1}}(\mathbb{G})\times
L^{p_{1}}(\mathbb{G})$ to $L^{p}(\mathbb{G})$; For $1\leq p_{2}\leq 2\leq
p_{1}\leq \infty $ and $p\geq 1$, if $\alpha >Q\left( \frac{1}{2}-\frac{1}{%
p_{1}}\right) -\left( 1-\frac{1}{p}\right) $, then $S^{\alpha }$ is bounded
from $L^{p_{1}}(\mathbb{G})\times L^{p_{1}}(\mathbb{G})$ to $L^{p}(\mathbb{G}%
)$.

(4)(region IV) For $1\leq p_{1}\leq 2\leq p_{2}\leq \infty $ and $p\leq 1$
, if $\alpha >Q\left( \frac{1}{p_{1}}-\frac{1}{2}\right) $, then $S^{\alpha
} $ is bounded from $L^{p_{1}}(\mathbb{G})\times L^{p_{1}}(\mathbb{G})$ to $%
L^{p}(\mathbb{G})$; For $1\leq p_{2}\leq 2\leq p_{1}\leq \infty $ and $p\leq
1$, if $\alpha >Q\left( \frac{1}{p_{2}}-\frac{1}{2}\right) $, then $%
S^{\alpha }$ is bounded from $L^{p_{1}}(\mathbb{G})\times L^{p_{1}}(\mathbb{G%
})$ to $L^{p}(\mathbb{G})$.

(5)(region V) For $1\leq p_{1},$ $p_{2}\leq 2$, if $\alpha >Q\left( \frac{1%
}{p}-1\right) $, then $S^{\alpha }$ is bounded from $L^{p_{1}}(\mathbb{G}%
)\times L^{p_{1}}(\mathbb{G})$ to $L^{p}(\mathbb{G}).$
\end{theorem}

\end{document}